\documentclass[11pt,reqno,twoside]{amsart}
\usepackage{amssymb,amsmath,amsthm,soul,color,paralist}
\usepackage{t1enc}
\usepackage[cp1250]{inputenc}
\usepackage{a4,indentfirst,latexsym}
\usepackage{graphics}
\usepackage{mathrsfs}
\usepackage{cite,enumitem,graphicx}
\usepackage[colorlinks=true,urlcolor=blue,
citecolor=red,linkcolor=blue,linktocpage,pdfpagelabels,
bookmarksnumbered,bookmarksopen]{hyperref}
\usepackage[english]{babel}
\usepackage[left=2.50cm,right=2.50cm,top=2.72cm,bottom=2.72cm]{geometry}

\usepackage[colorinlistoftodos]{todonotes}
\usepackage[normalem]{ulem}

\makeatletter
\providecommand\@dotsep{5}
\def\listtodoname{List of Todos}
\def\listoftodos{\@starttoc{tdo}\listtodoname}
\makeatother

\numberwithin{equation}{section}

\newcommand{\R}{\mathbb{R}}
\newcommand{\N}{\mathcal{N}}

\newcommand{\C}{\mathcal{C}}
\newcommand{\B}{\mathcal{B}}
\newcommand{\I}{\mathcal{I}}
\newcommand{\J}{\mathcal{J}}
\newcommand{\A}{\mathcal{A}}

\newcommand{\h}{\mathcal{W}_{\varepsilon}}
\newcommand{\p}{p^{*}_{s}}
\newcommand{\X}{\mathbb{X}}

\DeclareMathOperator{\supp}{supp}
\DeclareMathOperator{\e}{\varepsilon}

\newtheorem{proposition}{Proposition}[section]
\newtheorem{lemma}{Lemma}[section]
\newtheorem{theorem}{Theorem}[section]

\newtheorem{remark}{Remark}[section]

\title[Nonlinear Schr\"odinger equations with fractional $p$-Laplacian]{Multiplicity and concentration results for some nonlinear Schr\"odinger equations with the fractional $p$-Laplacian}

\author[V. Ambrosio]{Vincenzo Ambrosio}
\address{Vincenzo Ambrosio\hfill\break\indent 
Department of Mathematics  \hfill\break\indent
EPFL SB CAMA \hfill\break\indent
Station 8 CH-1015 Lausanne, Switzerland}
\email{vincenzo.ambrosio2@unina.it}

\author[T. Isernia]{Teresa Isernia}
\address{Teresa Isernia\hfill\break\indent
Dipartimento di Ingegneria Industriale e Scienze Matematiche \hfill\break\indent
Universit\`a Politecnica delle Marche\hfill\break\indent
Via Brecce Bianche, 12\hfill\break\indent
60131 Ancona (Italy)}
\email{teresa.isernia@unina.it}

\subjclass{Primary: 47G20, 35R11; Secondary: 35A15, 58E05}
\keywords{Fractional Schr\"odinger equation, fractional $p$-Laplacian operator, Nehari manifold, Ljusternik-Schnirelmann theory, critical growth.}

\date{}

\begin{document}

\maketitle

\begin{abstract}
We consider a class of parametric Schr\"odinger equations driven by the fractional $p$-Laplacian operator  and involving continuous positive potentials and nonlinearities with subcritical or critical growth. By using variational methods and Ljusternik-Schnirelmann theory, we study the existence, multiplicity and concentration of positive solutions for small values of the parameter.
\end{abstract}

\section{Introduction}
\noindent
In the first part of this paper we focus our attention on the existence, multiplicity and concentration of positive solutions for the following fractional $p$-Laplacian type problem
\begin{equation}\label{P}
\left\{
\begin{array}{ll}
\e^{sp}(-\Delta)_{p}^{s} u + V(x)|u|^{p-2}u = f(u) &\mbox{ in } \R^{N} \\
u\in W^{s, p}(\R^{N}) \\
u(x)>0 &\mbox{ for all } x\in \R^{N},
\end{array}
\right.
\end{equation}
where $\e>0$ is a parameter, $s\in (0, 1)$, $1<p<\infty$, $N>s p$, $W^{s, p}(\R^{N})$ is defined as the set of the functions $u:\R^{N}\rightarrow \R$ belonging to $L^{p}(\R^{N})$ such that
$$
[u]^{p}_{W^{s, p}(\R^{N})}=\iint_{\R^{2N}} \frac{|u(x)-u(y)|^{p}}{|x-y|^{N+sp}} dx dy<\infty.
$$
Here $(-\Delta)_{p}^{s}$ is the fractional $p$-Laplacian operator which may be defined, up to a normalization constant, by setting
\begin{equation*}
(-\Delta)_{p}^{s}u(x)= 2\lim_{r\rightarrow 0} \int_{\R^{N}\setminus \B_{r}(x)} \frac{|u(x)- u(y)|^{p-2}(u(x)- u(y))}{|x-y|^{N+sp}} dy \quad (x\in \R^{N})
\end{equation*}
for any $u\in \C^{\infty}_{c}(\R^{N})$; see \cite{DPV, MRS} for more details and applications.

Now we introduce the assumptions on the potential $V$ and the nonlinearity $f$.
We require that $V:\R^{N}\rightarrow \R$ is a continuous function satisfying the following condition introduced by Rabinowitz \cite{Rab}:
\begin{equation*}\tag{$V$}
V_{\infty}= \liminf_{|x|\rightarrow \infty} V(x)>V_{0}= \inf_{x\in \R^{N}} V(x)>0,
\end{equation*}
and we consider both cases $V_{\infty}<\infty$ and $V_{\infty}=\infty$.

Concerning the nonlinearity $f: \R\rightarrow \R$ we suppose that
\begin{compactenum}
\item [$(f_{1})$] $f\in \C^{0}(\R, \R)$ and $f(t)=0$ for all $t<0$;
\item [$(f_{2})$] $\displaystyle{\lim_{|t|\rightarrow 0} \frac{|f(t)|}{|t|^{p-1}}=0}$;
\item [$(f_{3})$] there exists $q\in (p, \p)$, with $\p=\frac{Np}{N-sp}$, such that $\displaystyle{\lim_{|t|\rightarrow \infty} \frac{|f(t)|}{|t|^{q-1}}=0}$;
\item [$(f_{4})$] there exists $\vartheta >p$ such that
\begin{equation*}
0<\vartheta F(t)= \vartheta \int_{0}^{t} f(\tau)\, d\tau \leq tf(t) \quad \mbox{ for all } t>0;
\end{equation*}
\item [$(f_{5})$] the map $\displaystyle{t\mapsto \frac{f(t)}{t^{p-1}}}$ is increasing in $(0, +\infty)$.
\end{compactenum}
Since we deal with the multiplicity of solutions to \eqref{P}, we recall that if $Y$ is a given closed set of a topological space $X$, we denote by $cat_{Y}(Y)$ the Ljusternik-Schnirelmann category of $Y$ in $X$, that is the least number of closed and contractible sets in $X$ which cover $Y$; see \cite{MW} for more details.

Let us denote by
\begin{equation*}
M=\{x\in \R^{N} : V(x)=V_{0}\} \quad \mbox{ and } \quad M_{\delta}= \{x\in \R^{N} : dist(x, M)\leq \delta\}, \mbox{ for } \delta>0.
\end{equation*}
Our first main result can be stated as follows:
\begin{theorem}\label{thmAI}
Let $N>sp$, and suppose that $V$ satisfies $(V)$ and $f$ verifies $(f_{1})$-$(f_{5})$. Then, for any $\delta>0$ there exists $\e_{\delta}>0$ such that problem \eqref{P} has at least $cat_{M_{\delta}}(M)$ positive solutions, for any $0<\e<\e_{\delta}$.
Moreover, if $u_{\e}$ denotes one of these solutions and $x_{\e}\in \R^{N}$ its global maximum, then
$$
\lim_{\e\rightarrow 0} V(x_{\e})=V_{0}.
$$
\end{theorem}

\noindent
Due to the variational nature of problem \eqref{P}, we look for critical points of the functional
\begin{equation*}
\I_{\e}(u)=\frac{1}{p} \iint_{\R^{2N}} \frac{|u(x)-u(y)|^{p}}{|x-y|^{N+sp}} dxdy + \frac{1}{p} \int_{\R^{N}} V(\e x) |u|^{p} dx - \int_{\R^{N}} F(u) \, dx
\end{equation*}
defined on a suitable subspace of $W^{s, p}(\R^{N})$.
Since $f$ is only continuous, we can not apply the Nehari manifold arguments developed in \cite{AF} in which the authors considered the corresponding local problem to \eqref{P} under the assumptions $f\in \mathcal{C}^{1}$ and
\begin{compactenum}[$(f_6)$]
\item there exist $C>0$ and $\sigma\in (p, p^{*}_{s})$ such that
$$
f'(t)t^{2}-(p-1)f(t)t\geq Ct^{\sigma} \quad \mbox{ for all } t\geq 0.
$$
\end{compactenum}
To overcome this difficulty, we use some variants of critical point theorems due to Szulkin and Weth \cite{SW}. As usual, the presence of the fractional $p$-Laplacian operator makes our analysis more delicate and intriguing.
In order to obtain multiple critical points, we employ a technique introduced by Benci and Cerami \cite{BC}, which consists in making precise comparisons between the category of some sublevel sets of $\I_{\e}$ and the category of the set $M$.
Then, after proving that the levels of compactness are strongly related to the behavior of the potential $V(x)$ at infinity (see Proposition \ref{prop2.1}), we can apply Ljusternik-Schnirelmann theory to deduce a multiplicity result for problem \eqref{P}.
Finally, we study the concentration of positive solutions $u_{\e}$ of \eqref{P}. More precisely, we first adapt the Moser iteration technique \cite{Moser} in the fractional setting (see Lemma \ref{lemMoser} in Section $3$) in order to obtain $L^{\infty}$-estimates (independent of $\e$) of $u_{\e}$'s. Then, taking into account the H\"older estimates recently established in \cite{IMS} for the fractional $p$-Laplacian, we can also deduce $C^{0, \alpha}$-estimates of $u_{\e}$ uniformly in $\e$. These informations allow us to infer that $u_{\e}(x)$ decay at zero as $|x|\rightarrow \infty$ uniformly in $\e$. Moreover, we prove that our solutions have a polynomial decay; see Remark \ref{Rem3}.
\vspace{0.1cm}

In the second part of this work we consider the following fractional problem involving the critical Sobolev exponent:
\begin{equation}\label{Pc}
\left\{
\begin{array}{ll}
\e^{sp}(-\Delta)_{p}^{s} u + V(x)|u|^{p-2}u = f(u) + |u|^{\p-2} u &\mbox{ in } \R^{N} \\
u\in W^{s, p}(\R^{N}) \\
u(x)>0 &\mbox{ for all } x\in \R^{N},
\end{array}
\right.
\end{equation}
with $s\in (0, 1)$, $1<p<\infty$ and $N\geq s p^{2}$.
In order to deal with the critical growth of the nonlinearity we assume that $f$ verifies $(f_1)$, $(f_2)$, $(f_3)$, $(f_5)$, hypothesis $(f_4)$ with $\vartheta\in (p, q)$, and the following technical condition:
\begin{compactenum}
\item[$(f'_{6})$] there exist $q_{1}\in (p, \p)$ and  $\lambda>0$ such that $f(t)\geq \lambda t^{q_{1}-1}$ for any $t>0$.
\end{compactenum}
Then we are able to obtain our second result:
\begin{theorem}\label{thmAI2}
Let $N\geq sp^{2}$, and suppose that $V$ satisfies $(V)$ and $f$ satisfies $(f_{1})$-$(f_5)$ and $(f'_{6})$. Then, for any $\delta>0$ there exists $\e_{\delta}>0$ such that problem \eqref{Pc} has at least $cat_{M_{\delta}}(M)$ positive solutions, for any $0<\e<\e_{\delta}$. Moreover, if $u_{\e}$ denotes one of these solutions and $x_{\e}\in \R^{N}$ its global maximum, then
$$
\lim_{\e\rightarrow 0} V(x_{\e})=V_{0}.
$$
\end{theorem}

\noindent
We note that Theorem \ref{thmAI2} improves and extends, in the fractional setting, Theorem $1.1$ in \cite{F} in which the author assumed $f\in \C^{1}$. The approach developed in this case follows the arguments used to analyze  the subcritical case. Anyway, this new problem presents an extra difficulty, due to the fact that the level of non-compactness is affected by the critical growth of the nonlinearity.
To overcome this hitch, we adapt some calculations performed in \cite{MPSY} and using the optimal asymptotic behavior of $p$-minimizers established in \cite{BMS} we are able to prove that the functional associated to \eqref{Pc} satisfies the Palais-Smale condition at every level
$$
c<\frac{s}{N} S_{*}^{\frac{N}{sp}},
$$
where
$$
S_{*}=\inf_{u\in \mathcal{D}^{s, p}(\R^{N})} \frac{[u]^{p}_{W^{s, p}(\R^{N})}}{|u|^{p}_{L^{\p}(\R^{N})}},
$$
and
$$
\mathcal{D}^{s, p}(\R^{N}) =\left\{u\in L^{\p}(\R^{N}): [u]_{W^{s, p}(\R^{N})}<\infty\right\}.
$$
Let us point out that the restriction $N\geq sp^{2}$ is crucial to apply Lemma 2.7 in \cite{MPSY} to estimate the $L^{p}$-norm of $p$-minimizers; see Lemma \ref{MPSYlem} and Lemma \ref{lemest2} in Section \ref{Sect4}.

\noindent
When $p=2$, equations \eqref{P} and \eqref{Pc} become fractional Schr\"odinger equations of the type
\begin{equation}\label{epsilon}
\e^{2s}(-\Delta)^{s}u+V(x)u=f(x, u) \quad \mbox{ in } \R^{N},
\end{equation}
which has been widely investigated in the last decade: see \cite{AM, A2, A3, AMF, AI, DMV, FQT, FS, IT,  Secchi, SZY} and references therein.
The study of \eqref{epsilon} is strongly motivated by the seeking of standing waves solutions for the time-dependent fractional Schr\"odinger equation
\begin{equation}\label{NFSE}
\imath \e \frac{\partial \psi}{\partial t}=(-\Delta)^{s}\psi+(V(x)+E) \psi-f(x,\psi) \quad \mbox{ for } (x, t)\in \R^{N}\times \R,
\end{equation}
namely solutions of the form $\psi(x, t)=u(x) e^{-\frac{\imath E t}{\e}}$, where $E$ is a constant. Equation \eqref{NFSE}
is a fundamental equation of the fractional Quantum Mechanics in the study of particles on stochastic fields modelled by L\'evy processes; see \cite{Laskin1, Laskin2} for more physical background.

In recent years there has been a surge of interest around nonlocal and fractional problems involving the fractional $p$-Laplacian operator, and several existence and regularity results have been established by many authors. For instance,  Franzina and Palatucci \cite{FrP}  discussed some basic properties of the eigenfunctions of a class of nonlocal operators whose model is $(-\Delta)^{s}_{p}$ (see also \cite{LL}).
Mosconi et al. \cite{MPSY} used an abstract linking theorem based on the cohomological index to obtain nontrivial solutions to the Brezis-Nirenberg problem for the fractional $p$-Laplacian operator. Di Castro et al. \cite{DCKP} established interior H\"older regularity results for fractional $p$-minimizers (see also \cite{IMS}). In \cite{A1} the first author obtained the existence of infinitely many solutions for a superlinear fractional $p$-Laplacian equation with sign-changing potential.
Fiscella and Pucci \cite{FiP} investigated the existence and asymptotic behavior of nontrivial solutions for some classes of stationary Kirchhoff  problems driven by fractional integro-differential operators and involving a Hardy potential and different critical nonlinearities. Belchior et al. \cite{BBMP} studied the existence of ground state solutions for a fractional Choquard equation with the fractional $p$-Laplacian and involving subcritical nonlinearities.\\
However, for what concerns the existence and multiplicity results for problems \eqref{P} and \eqref{Pc} with $p\neq 2$, the literature seems to be rather incomplete.

The goal of this paper is to consider the question related to the existence and multiplicity of positive solutions for fractional Schr\"odinger equations with the fractional $p$-Laplacian when $\e\rightarrow 0$. More precisely, we aim to extend the results obtained in \cite{FS} and \cite{SZY} in which the authors dealt with equation \eqref{epsilon} and involving nonlinearities with subcritical and critical growth respectively.
Unfortunately, the operator $(-\Delta)^{s}_{p}$ is not linear when $p\neq 2$, so more technical difficulties arise in the study of our problems. For instance, we can not make use of the $s$-harmonic extension by Caffarelli and Silvestre \cite{CS} commonly exploited in the recent literature to apply well-known variational techniques in the study of local degenerate elliptic  problems. Moreover, the arguments used in the study of \eqref{epsilon} (see for example \cite{AA, AM, A2, A4, FS, SZY}) seem not to be trivially adaptable in our context.
Indeed, some appropriate technical lemmas (see Lemma \ref{Psi}, Lemma \ref{lemVince}, and Lemma \ref{lemC}) will be needed to overcome the non-Hilbertian structure of the fractional Sobolev spaces $W^{s, p}(\R^{N})$ when $p\neq 2$.

We would like to emphasize that, to our knowledge, this is the first time that the Ljusternik-Schnirelmann theory is applied to get multiple solutions for fractional Schr\"odinger equations in $\R^{N}$ driven by the fractional $p$-Laplacian operator with $p\neq 2$, and involving nonlinearities with subcritical and critical growth.
\smallskip

\noindent
The paper is organized as follows: in Section \ref{Sect2} we collect some facts about the involved fractional Sobolev spaces  and we provide some useful technical lemmas. In Section \ref{Sect3} we study the existence, multiplicity and concentration of solutions to \eqref{P} proving some convenient properties for the autonomous problem associated to \eqref{P}. In Section \ref{Sect4} we consider critical problem \eqref{Pc} and the corresponding autonomous critical one.

\section{Preliminaries}\label{Sect2}

\noindent
In this preliminary section we recall some facts about the fractional Sobolev spaces and we prove some technical lemmas which we will use later.

Let $1\leq r\leq \infty$ and $A\subset \R^{N}$. We denote by $|u|_{L^{r}(A)}$ the $L^{r}(A)$-norm of a function $u:\R^{N}\rightarrow \R$ belonging to $L^{r}(A)$.
We define $\mathcal{D}^{s, p}(\R^{N})$ as the closure of $\C^{\infty}_{c}(\R^{N})$ with respect to
$$
[u]_{W^{s, p}(\R^{N})}^{p}= \iint_{\R^{2N}} \frac{|u(x)-u(y)|^{p}}{|x-y|^{N+sp}}dxdy.
$$
Let us indicate by $W^{s, p}(\R^{N})$ the set of functions $u\in L^{p}(\R^{N})$ such that $[u]_{W^{s, p}(\R^{N})}<\infty$, endowed with the natural norm
\begin{equation*}
\|u\|_{W^{s, p}(\R^{N})}^{p}= [u]_{W^{s, p}(\R^{N})}^{p}+ |u|_{L^{p}(\R^{N})}^{p}.
\end{equation*}

We begin recalling the following embeddings of the fractional Sobolev spaces into Lebesgue spaces.
\begin{theorem}\cite{DPV}\label{Sembedding}
Let $s\in (0,1)$ and $N>sp$. Then there exists a sharp constant $S_{*}>0$
such that for any $u\in \mathcal{D}^{s, p}(\R^{N})$
\begin{equation*}
|u|^{p}_{L^{p^{*}_{s}}(\R^{N})} \leq S_{*}^{-1} [u]^{p}_{W^{s, p}(\R^{N})}.
\end{equation*}
Moreover $W^{s, p}(\R^{N})$ is continuously embedded in $L^{q}(\R^{N})$ for any $q\in [p, p^{*}_{s}]$ and compactly in $L^{q}_{loc}(\R^{N})$ for any $q\in [1, p^{*}_{s})$.
\end{theorem}

\noindent
Proceeding as in \cite{FQT, Secchi} we can prove the following compactness-Lions type result.
\begin{lemma}\label{Lions}
Let $N>sp$ and $r\in [p, \p)$. If $\{u_{n}\}$ is a bounded sequence in $W^{s, p}(\R^{N})$ and if
\begin{equation}\label{ter4}
\lim_{n\rightarrow \infty} \sup_{y\in \R^{N}} \int_{\B_{R}(y)} |u_{n}|^{r} dx=0,
\end{equation}
where $R>0$, then $u_{n}\rightarrow 0$ in $L^{t}(\R^{N})$ for all $t\in (p, \p)$.
\end{lemma}
\begin{proof}
Let $\tau \in (r, \p)$. Then, H\"older and Sobolev inequality yield, for every $n\in \mathbb{N}$, that
\begin{align*}
|u_{n}|_{L^{\tau}(\B_{R}(y))} &\leq |u_{n}|_{L^{r}(\B_{R}(y))}^{1-\alpha} |u_{n}|_{L^{\p}(\B_{R}(y))}^{\alpha} \\
&\leq C |u_{n}|_{L^{r}(\B_{R}(y))}^{1- \alpha} \|u_{n}\|_{W^{s, p}(\R^{N})}^{\alpha}
\end{align*}
where $\alpha= \frac{\tau-r}{\p-r}\frac{\p}{\tau}$.
Now, covering $\R^{N}$ by balls of radius $R$, in such a way that each point of $\R^{N}$ is contained in at most $N+1$ balls, we find
\begin{align*}
|u_{n}|_{L^{\tau}(\R^{N})}^{\tau} \leq C (N+1) \sup_{y\in \R^{N}} \left(\int_{\B_{R}(y)} |u_{n}|^{r} dx \right)^{(1-\alpha)\tau} \|u_{n}\|^{\alpha \tau}_{W^{s, p}(\R^{N})}.
\end{align*}
Exploiting \eqref{ter4} and the boundedness of $\{u_{n}\}$ in $W^{s, p}(\R^{N})$, we obtain that $u_{n}\rightarrow 0$ in $L^{\tau}(\R^{N})$. By using an interpolation argument, we get the thesis.
\end{proof}

\noindent
The lemma below provides a way to manipulate smooth truncations for the fractional $p$-Laplacian.
Let us note that this result can be seen as a generalization of the second statement of Lemma $5$ in \cite{PP} to the case of the space $W^{s, p}(\R^{N})$ with $p\neq 2$.
\begin{lemma}\label{Psi}
Let $u\in W^{s, p}(\R^{N})$ and $\phi\in \mathcal{C}^{\infty}_{c}(\R^{N})$ such that $0\leq \phi\leq 1$, $\phi=1$ in $\B_{1}(0)$ and $\phi=0$ in $\B_{2}^{c}(0)$. Set $\phi_{r}(x)=\phi(\frac{x}{r})$.  Then
$$
\lim_{r\rightarrow \infty} [u \phi_{r}-u]_{W^{s, p}(\R^{N})}=0 \quad \mbox{ and } \quad \lim_{r\rightarrow \infty} |u\phi_{r}-u|_{L^{p}(\R^{N})}=0.
$$
\end{lemma}
\begin{proof}
Since $\phi_{r}u\rightarrow u$ a.e. in $\R^{N}$ as $r\rightarrow \infty$, $0\leq \phi\leq 1$ and $u\in L^{p}(\R^{N})$, it follows by the Dominated Convergence Theorem that $\lim_{r\rightarrow \infty} |u\phi_{r}-u|_{L^{p}(\R^{N})}=0$.
In what follows, we show the first relation of limit.
Let us note that
\begin{align*}
&[u\phi_{r}-u]^{p}_{W^{s, p}(\R^{N})}\\
&\leq 2^{p-1} \!\left[\iint_{\R^{2N}} \!\! |u(x)|^{p}\frac{|\phi_{r}(x)-\phi_{r}(y)|^{p}}{|x-y|^{N+sp}}dx dy+\!\iint_{\R^{2N}} \!\!\frac{|\phi_{r}(x)-1|^{p}|u(x)-u(y)|^{p}}{|x-y|^{N+sp}}dx dy\right] \\
&=:2^{p-1}[A_{r}+B_{r}].
\end{align*}
Taking into account that $|\phi_{r}(x)-1|\leq 2$, $|\phi_{r}(x)-1|\rightarrow 0$ a.e. in $\R^{N}$ and $u\in W^{s, p}(\R^{N})$, the Dominated Convergence Theorem yields
$$
B_{r}\rightarrow 0 \quad \mbox{ as } r\rightarrow \infty.
$$
Now, we aim to show that
$$
A_{r}\rightarrow 0 \quad \mbox{ as } r\rightarrow \infty.
$$
Firstly, we observe that since $\R^{2N}$ can be written as
\begin{align*}
\R^{2N}&=((\R^{N} \setminus \B_{2r}(0))\times (\R^{N} \setminus \B_{2r}(0)))\cup (\B_{2r}(0)\times \R^{N}) \cup ((\R^{N} \setminus \B_{2r}(0))\times \B_{2r}(0))\\
&=: X^{1}_{r}\cup X^{2}_{r} \cup X^{3}_{r},
\end{align*}
we get
\begin{align}\label{Pa1}
&\iint_{\R^{2N}} |u(x)|^{p} \frac{|\phi_{r}(x)-\phi_{r}(y)|^{p}}{|x-y|^{N+sp}} \, dx dy \nonumber\\
&=\iint_{X^{1}_{r}} |u(x)|^{p} \frac{|\phi_{r}(x)-\phi_{r}(y)|^{p}}{|x-y|^{N+sp}} \, dx dy
+\iint_{X^{2}_{r}} |u(x)|^{p} \frac{|\phi_{r}(x)-\phi_{r}(y)|^{p}}{|x-y|^{N+sp}} \, dx dy \nonumber\\
&+ \iint_{X^{3}_{r}} |u(x)|^{p} \frac{|\phi_{r}(x)-\phi_{r}(y)|^{p}}{|x-y|^{N+sp}} \, dx dy.
\end{align}
We are going to estimate each integral in (\ref{Pa1}).
Recalling that $\phi=0$ in $\R^{N}\setminus \B_{2}(0)$, we have
\begin{align}\label{Pa2}
\iint_{X^{1}_{r}} |u(x)|^{p} \frac{|\phi_{r}(x)-\phi_{r}(y)|^{p}}{|x-y|^{N+sp}} \, dx dy=0.
\end{align}
By using $0\leq \phi\leq 1$, $|\nabla \phi|\leq 2$ and by applying the Mean Value Theorem, we can see that
\begin{align}\label{Pa3}
&\iint_{X^{2}_{r}} |u(x)|^{p} \frac{|\phi_{r}(x)-\phi_{r}(y)|^{p}}{|x-y|^{N+sp}} \, dx dy \nonumber\\
&=\int_{\B_{2r}(0)} \,dx \int_{\{y\in \R^{N}: |y-x|\leq r\}} |u(x)|^{p} \frac{|\phi_{r}(x)-\phi_{r}(y)|^{p}}{|x-y|^{N+sp}} \, dy \nonumber \\
&+\int_{\B_{2r}(0)} \, dx \int_{\{y\in \R^{N}: |y-x|> r\}} |u(x)|^{p} \frac{|\phi_{r}(x)-\phi_{r}(y)|^{p}}{|x-y|^{N+sp}} \, dy  \nonumber\\
&\leq C r^{-p} |\nabla \phi|_{L^{\infty}(\R^{N})}^{p} \int_{\B_{2r}(0)} \, dx \int_{\{y\in \R^{N}: |y-x|\leq r\}} \frac{|u(x)|^{p}}{|x-y|^{N+sp-p}} \, dy \nonumber \\
&+ C \int_{\B_{2r}(0)} \, dx \int_{\{y\in \R^{N}: |y-x|> r\}} \frac{|u(x)|^{p}}{|x-y|^{N+sp}} \, dy \nonumber\\
&\leq C r^{-sp} \int_{\B_{2r}(0)} |u(x)|^{p} \, dx+C r^{-sp} \int_{\B_{2r}(0)} |u(x)|^{p} \, dx \nonumber \\
&=Cr^{-sp} \int_{\B_{2r}(0)} |u(x)|^{p} \, dx.
\end{align}
Regarding the integral on $X^{3}_{r}$ we obtain
\begin{align}\label{Pa4}
&\iint_{X^{3}_{r}} |u(x)|^{p} \frac{|\phi_{r}(x)-\phi_{r}(y)|^{p}}{|x-y|^{N+sp}} \, dx dy \nonumber\\
&=\int_{\R^{N}\setminus \B_{2r}(0)} \, dx \int_{\{y\in \B_{2r}(0): |y-x|\leq r\}} |u(x)|^{p} \frac{|\phi_{r}(x)-\phi_{r}(y)|^{p}}{|x-y|^{N+sp}} \, dy \nonumber\\
&+\int_{\R^{N}\setminus \B_{2r}(0)} \,dx \int_{\{y\in \B_{2r}(0): |y-x|>r\}} |u(x)|^{p} \frac{|\phi_{r}(x)-\phi_{r}(y)|^{p}}{|x-y|^{N+sp}} \, dy=: C_{r}+ D_{r}.
\end{align}
At this point, by using the Mean Value Theorem and noting that if $(x, y) \in (\R^{N}\setminus \B_{2r}(0))\times \B_{2r}(0)$ and $|x-y|\leq r$ then $|x|\leq 3r$, we get
\begin{align}\label{Pa5}
C_{r}&\leq r^{-p} |\nabla \phi|_{L^{\infty}(\R^{N})}^{p} \int_{\B_{3r}(0)} \, dx \int_{\{y\in \B_{2r}(0): |y-x|\leq r\}} \frac{|u(x)|^{p}}{|x-y|^{N+sp-p}} \, dy \nonumber\\
&\leq C r^{-p}  \int_{\B_{3r}(0)} |u(x)|^{p} \, dx \int_{\{z\in \R^{N}: |z|\leq r\}} \frac{1}{|z|^{N+sp-p}} \, dz \nonumber\\
&=C r^{-sp} \int_{\B_{3r}(0)} |u(x)|^{p} \, dx.
\end{align}
Let us observe that for any $K>4$ it holds
$$
X_{r}^{3}=(\R^{N}\setminus \B_{2r}(0))\times \B_{2r}(0) \subset (\B_{K r}(0) \times \B_{2r}(0)) \cup ((\R^{N}\setminus \B_{Kr}(0))\times \B_{2r}(0)).
$$
Then we have
\begin{align}\label{Pa6}
&\int_{\B_{Kr}(0)} \, dx \int_{\{y\in \B_{2r}(0): |y-x|> r\}}  |u(x)|^{p} \frac{|\phi_{r}(x)-\phi_{r}(y)|^{p}}{|x-y|^{N+sp}} \, dy \nonumber\\
&\leq C \int_{\B_{Kr}(0)} \, dx \int_{\{y\in \B_{2r}(0): |y-x|> r\}} \frac{|u(x)|^{p}}{|x-y|^{N+sp}} \, dy \nonumber \\
&\leq C \int_{\B_{Kr}(0)} |u(x)|^{p} \, dx \int_{\{z\in \R^{N}: |z|> r\}} \frac{1}{|z|^{N+sp}} \, dz \nonumber\\
&= C r^{-sp} \int_{\B_{Kr}(0)} |u(x)|^{p} \, dx.
\end{align}
Let us note that if $(x, y)\in (\R^{N}\setminus \B_{Kr}(0))\times \B_{r}(0)$, then $|x-y|\geq |x|- |y|\geq \frac{|x|}{2}+ \frac{K}{2}r -2r >\frac{|x|}{2}$, and by using H\"older inequality we can see that
\begin{align}\label{Pa7}
&\int_{\R^{N}\setminus \B_{Kr}(0)} \, dx \int_{\{y\in \B_{2r}(0): |y-x|>r\}} |u(x)|^{p} \frac{|\phi_{r}(x)-\phi_{r}(y)|^{p}}{|x-y|^{N+sp}} \, dy \nonumber\\
&\leq  C \int_{\R^{N}\setminus \B_{Kr}(0)} \, dx \int_{\{y\in \B_{2r}(0): |y-x|>r \}} \frac{|u(x)|^{p}}{|x-y|^{N+sp}} \, dy \nonumber\\
&\leq C r^{N} \int_{\R^{N}\setminus \B_{Kr}(0)} \frac{|u(x)|^{p}}{|x|^{N+sp}} \, dx \nonumber\\
&\leq C r^{N} \left(\int_{\R^{N}\setminus \B_{Kr}(0)} |u(x)|^{p^{*}_{s}} \, dx\right)^{\frac{p}{p^{*}_{s}}} \left(\int_{\R^{N}\setminus \B_{Kr}(0)} |x|^{-(N+sp)\frac{p^{*}_{s}}{p^{*}_{s}-p}} \, dx\right)^{\frac{p^{*}_{s}-p}{p^{*}_{s}}} \nonumber\\
&\leq C K^{-N} \left(\int_{\R^{N}\setminus \B_{Kr}(0)} |u(x)|^{p^{*}_{s}} \, dx\right)^{\frac{p}{p^{*}_{s}}}.
\end{align}
Therefore, taking into account (\ref{Pa6}) and (\ref{Pa7}), we get
\begin{align}\label{Pa8}
D_{r}\leq C r^{-sp} \int_{\B_{Kr}(0)} |u(x)|^{p} \, dx+C K^{-N}.
\end{align}
Putting together (\ref{Pa1})-(\ref{Pa5}) and (\ref{Pa8}), we obtain
\begin{align*}
\iint_{\R^{2N}} |u(x)|^{p} \frac{|\phi_{r}(x)-\phi_{r}(y)|^{p}}{|x-y|^{N+sp}} \, dx dy \leq Cr^{-sp} \int_{\B_{Kr}(0)} |u(x)|^{p} \, dx+C K^{-N},
\end{align*}
from which we deduce
\begin{align*}
& \limsup_{r\rightarrow \infty} \iint_{\R^{2N}} |u(x)|^{p} \frac{|\phi_{r}(x)-\phi_{r}(y)|^{p}}{|x-y|^{N+sp}} \, dx dy \nonumber\\
&\quad \quad =\lim_{K\rightarrow \infty}\limsup_{r\rightarrow \infty} \iint_{\R^{2N}} |u(x)|^{p} \frac{|\phi_{r}(x)-\phi_{r}(y)|^{p}}{|x-y|^{N+sp}} \, dx dy =0.
\end{align*}

\end{proof}

\noindent
Now, fixed $\e>0$, we introduce the following fractional Sobolev space
\begin{equation*}
\h:= \left\{u\in W^{s, p}(\R^{N}) : \int_{\R^{N}} V(\e x) |u|^{p} dx <+\infty \right\}
\end{equation*}
endowed with the norm
\begin{equation*}
\|u\|_{\e}^{p} := [u]_{W^{s, p}(\R^{N})}^{p} + \int_{\R^{N}} V(\e x) |u|^{p} dx.
\end{equation*}

\noindent
In view of assumption $(V)$ and Theorem \ref{Sembedding}, it is easy to check that the following result holds.
\begin{lemma}\label{embedding}
The space $\h$ is continuously embedded in $W^{s, p}(\R^{N})$.
Therefore, $\h$ is continuously embedded in $L^{r}(\R^{N})$ for any $r\in [p, \p]$ and compactly embedded in $L^{r}_{loc}(\R^{N})$ for any $r\in [1, \p)$.
\end{lemma}

\noindent
Moreover, when $V$ is coercive, we get the following compactness lemma.
\begin{lemma}\label{Cheng}
Let $V_{\infty}=\infty$. Then $\h$ is compactly embedded in $L^{r}(\R^{N})$ for any $r\in [p, p^{*}_{s})$.
\end{lemma}

\begin{proof}
We argue as in \cite{Zh}. Let $r=p$. By using Lemma \ref{embedding} we know that $\h\subset L^{p}(\R^{N})$. Let $\{u_{n}\}$ be a sequence such that $u_{n}\rightharpoonup 0$ in $\h$. Then, $u_{n}\rightharpoonup 0$
 in $W^{s, p}(\R^{N})$. \\
Let us define
\begin{equation}\label{ter1}
M:= \sup_{n\in \mathbb{N}} \|u_{n}\|_{\e} <\infty.
\end{equation}
Since $V$ is coercive, for any $\eta>0$ there exists $R=R_{\eta}>0$ such that
\begin{equation}\label{ter2}
\frac{1}{V(\e x)}<\eta, \quad \mbox{ for any } |x|>R.
\end{equation}
Since $u_{n}\rightarrow 0$ in $L^{p}(\B_{R}(0))$, there exists $n_{0}>0$ such that
\begin{equation}\label{ter3}
\int_{\B_{R}(0)} |u_{n}|^{p} dx \leq \eta \quad \mbox{ for any } n\geq n_{0}.
\end{equation}
Hence, for any $n\geq n_{0}$, by using \eqref{ter1}-\eqref{ter3}, we have
\begin{align*}
\int_{\R^{N}} |u_{n}|^{p} dx &=\int_{\B_{R}(0)} |u_{n}|^{p} dx + \int_{\B_{R}^{c}(0)} |u_{n}|^{p} dx \\
&< \eta +\eta \int_{\B_{R}^{c}(0)} V(\e x) |u_{n}|^{p} dx \leq \eta(1+ M^{p}).
\end{align*}
Therefore, $u_{n}\rightarrow 0$ in $L^{p}(\R^{N})$. \\
For $r>p$, using the conclusion of $r=p$, interpolation inequality and Theorem \ref{Sembedding}, we can see that
\begin{equation*}
|u_{n}|_{L^{r}(\R^{N})} \leq C [u_{n}]^{\alpha}_{W^{s,p}(\R^{N})} |u_{n}|_{L^{p}(\R^{N})}^{1-\alpha},
\end{equation*}
where $\frac{1}{r}=\frac{\alpha}{p}+\frac{1-\alpha}{\p}$, which yields the conclusion as required.
\end{proof}

\noindent
The next two results are technical lemmas which will be very useful in this work; their proofs are obtained following the arguments developed by Brezis and Lieb \cite{BL}.
\begin{lemma}\label{lemPSY}
If $\{u_{n}\}$ is a bounded sequence in $\h$, then
\begin{align*}
\|u_{n}- u\|_{\e}^{p} = \|u_{n}\|_{\e}^{p} - \|u\|_{\e}^{p} + o_{n}(1).
\end{align*}
\end{lemma}
\begin{proof}
From the Brezis-Lieb Lemma \cite{BL} we know that if $r\in (1, \infty)$ and $\{g_{n}\}\subset L^{r}(\R^{k})$ is a bounded sequence such that $g_{n}\rightarrow g$ a.e. in $\R^{k}$, then we have
\begin{align}\label{ggg}
|g_{n}-g|_{L^{r}(\R^{k})}^{r}= |g_{n}|_{L^{r}(\R^{k})}^{r} - |g|_{L^{r}(\R^{k})}^{r} +o_{n}(1).
\end{align}
Therefore
\begin{align*}
\int_{\R^{N}} V(\e x) |u_{n}-u|^{p}= \int_{\R^{N}} V(\e x) |u_{n}|^{p}- \int_{\R^{N}} V(\e x) |u|^{p}+ o_{n}(1),
\end{align*}
and taking
\begin{equation*}
g_{n}=\frac{(u_{n}-u)(x)-(u_{n}-u)(y)}{|x-y|^{\frac{N+sp}{p}}}, \quad  g= \frac{u(x)-u(y)}{|x-y|^{\frac{N+sp}{p}}}, \quad r=p \, \mbox{ and } \, k=2N
\end{equation*}
in \eqref{ggg}, we can see that
\begin{align*}
[u_{n}-u]^{p}_{W^{s, p}(\R^{N})}= [u_{n}]^{p}_{W^{s, p}(\R^{N})}- [u]^{p}_{W^{s, p}(\R^{N})} + o_{n}(1).
\end{align*}
\end{proof}

\begin{lemma}\label{lemVince}
Let $w\in \mathcal{D}^{s, p}(\R^{N})$ and $\{z_{n}\}\subset \mathcal{D}^{s, p}(\R^{N})$ be a sequence such that $z_{n}\rightarrow 0$ a.e. and $[z_{n}]_{W^{s, p}(\R^{N})}\leq C$ for any $n\in \mathbb{N}$. Then we have
\begin{align*}
\iint_{\R^{2N}} |\A(z_{n} + w) - \A(z_{n}) - \A(w)|^{p'} dx= o_{n}(1),
\end{align*}
where $\A(u):=\frac{|u(x)- u(y)|^{p-2}(u(x)- u(y))}{|x-y|^{\frac{N+sp}{p'}}}$ and $p'= \frac{p}{p-1}$ is the conjugate exponent of $p$.
\end{lemma}

\begin{proof}
Firstly we consider the case $p\geq 2$. We resemble some ideas in Lemma $3$ in \cite{Alv}. By using the Mean Value Theorem, Young inequality and $p\geq 2$, we can see that for fixed $\e>0$ there exists $C_{\e}>0$ such that
\begin{equation}\label{abp-2}
||a+b|^{p-2}(a+b)-|a|^{p-2}a|\leq \e |a|^{p-1}+C_{\e}|b|^{p-1} \quad \mbox{ for all } a, b\in \R.
\end{equation}
Taking
$$
a=\frac{z_{n}(x)-z_{n}(y)}{|x-y|^{\frac{N+sp}{p}}} \quad \mbox{ and } \quad b=\frac{w(x)-w(y)}{|x-y|^{\frac{N+sp}{p}}}
$$
in \eqref{abp-2}, we obtain
\begin{align*}
&\Bigl|\frac{|(z_{n}(x)+w(x))- (z_{n}(y)+w(y))|^{p-2}((z_{n}(x)+w(x))- (z_{n}(y)+w(y)))}{|x-y|^{\frac{N+sp}{p'}}} \\
&\qquad- \frac{|z_{n}(x)- z_{n}(y)|^{p-2}(z_{n}(x)- z_{n}(y))}{|x-y|^{\frac{N+sp}{p'}}} \Bigr| \\
&\leq \e \frac{|z_{n}(x)-z_{n}(y)|^{p-1}}{|x-y|^{\frac{N+sp}{p'}}}+ C_{\e}  \frac{|w(x)-w(y)|^{p-1}}{|x-y|^{\frac{N+sp}{p'}}}.
\end{align*}
Let us define the function $H_{\e, n}:\R^{2N}\rightarrow \R_{+}$ by setting
\begin{align*}
H_{\e, n} (x, y):= \max \left\{ |\A(z_{n} + w) - \A(z_{n}) - \A(w)|- \e \frac{|z_{n}(x)-z_{n}(y)|^{p-1}}{|x-y|^{\frac{N+sp}{p'}}}, 0 \right\}.
\end{align*}
We can see that $H_{\e, n} \rightarrow 0$ a.e. in $\R^{2N}$ as $n\rightarrow \infty$, and
\begin{align*}
0\leq H_{\e, n}(x, y) \leq C_{1} \frac{|w(x)-w(y)|^{p-1}}{|x-y|^{\frac{N+sp}{p'}}}\in L^{p'}(\R^{2N}).
\end{align*}
By using the Dominated Convergence Theorem, we have
\begin{align*}
\int_{\R^{2N}} |H_{\e, n} |^{p'} dxdy \rightarrow 0 \quad \mbox{ as } n\rightarrow \infty.
\end{align*}
From the definition of $H_{\e, n}$, we deduce that
\begin{align*}
|\A(z_{n} + w) - \A(z_{n}) - \A(w)|\leq \e \frac{|z_{n}(x)-z_{n}(y)|^{p-1}}{|x-y|^{\frac{N+sp}{p'}}} + H_{\e, n},
\end{align*}
so we obtain
\begin{align*}
|\A(z_{n} + w) -\A(z_{n}) - \A(w)|^{p'}\leq C_{2}\left[\e^{p'} \frac{|z_{n}(x)-z_{n}(y)|^{p}}{|x-y|^{N+sp}} + (H_{\e, n})^{p'}\right].
\end{align*}
Therefore
\begin{align*}
\limsup_{n\rightarrow \infty} &\iint_{\R^{2N}} |\A(z_{n} + w) - \A(z_{n}) - \A(w)|^{p'} dxdy\\
& \leq C_{2}\e^{p'}\limsup_{n\rightarrow \infty} [z_{n}]^{p}_{W^{s, p}(\R^{N})} \\
&\leq C_{3}\e^{p'},
\end{align*}
and by the arbitrariness of $\e$ we get the thesis.

Now we deal with the case $1<p<2$. By using Lemma $3.1$ in \cite{MeW}, we know that
$$
\sup_{c\in \R^{N}, d\neq 0} \left|\frac{|c+d|^{p-2}(c+d)-|c|^{p-2}c}{|d|^{p-1}}\right|<\infty,
$$
so, by setting
$$
c=\frac{z_{n}(x)-z_{n}(y)}{|x-y|^{\frac{N+sp}{p}}} \quad \mbox{ and } \quad d=\frac{w(x)-w(y)}{|x-y|^{\frac{N+sp}{p}}},
$$
we can conclude the proof in view of the Dominated Convergence Theorem.
\end{proof}

\begin{lemma}\label{splitting}
Let $\{u_{n}\}$ be a sequence such that $u_{n}\rightharpoonup u$ in $\h$, and $v_{n}:=u_{n}-u$.
Then we have
\begin{equation}\label{A}
\int_{\R^{N}} (F(v_{n})-F(u_{n})+F(u))\, dx=o_{n}(1)
\end{equation}
and
\begin{equation}\label{B}
\sup_{\|w\|_{\e}\leq 1} \int_{\R^{N}} |(f(v_{n})-f(u_{n})+f(u)) w|\, dx=o_{n}(1).
\end{equation}
\end{lemma}
\begin{proof}
We begin proving \eqref{A}. Let us note that
$$
F(v_{n})-F(u_{n})=\int_{0}^{1} \frac{d}{dt} F(u_{n}-tu)dt=-\int_{0}^{1} u f(u_{n}-tu)dt.
$$
In view of $(f_{2})$ and $(f_{3})$, for any $\delta>0$ there exists $C_{\delta}>0$ such that
\begin{align}\label{Ftiodio}
|f(t)|\leq \delta |t|^{p-1}+ C_{\delta}|t|^{\p-1}\quad \mbox{ for all } t\in \R.
\end{align}
By using \eqref{Ftiodio} with $\delta=1$ and $(a+b)^{r}\leq C_{r}(a^{r}+ b^{r})$ for any $a, b\geq 0$ and $r\geq 1$, we can see that
\begin{align}\label{odio}
|F(v_{n})-F(u_{n})|\leq C_{p} |u_{n}|^{p-1}|u|+C_{p} |u|^{p}+C_{1}C_{\p} |u_{n}|^{\p-1}|u|+C_{1}C_{\p} |u|^{\p}.
\end{align}
By applying the Young inequality $ab \leq \eta a^{r} + C_{\eta} b^{r'}$ with $\frac{1}{r}+\frac{1}{r'}=1$ and $\eta>0$ to the first and third term on the right hand side of \eqref{odio}, we can deduce that
$$
|F(v_{n})-F(u_{n})|\leq \eta (|u_{n}|^{p}+|u_{n}|^{\p})+C_{\eta} (|u|^{p}+|u|^{\p})
$$
which implies that
$$
|F(v_{n})-F(u_{n})+F(u)|\leq \eta (|u_{n}|^{p}+|u_{n}|^{\p})+C'_{\eta} (|u|^{p}+|u|^{\p}).
$$
Let
$$
G_{\eta, n}(x):=\max\left\{|F(v_{n})-F(u_{n})+F(u)|-\eta(|u_{n}|^{p}+|u_{n}|^{\p}), 0\right\}.
$$
Then $G_{\eta, n}\rightarrow 0$ a.e. in $\R^{N}$ as $n\rightarrow \infty$, and $0\leq G_{\eta, n}\leq C'_{\eta} (|u|^{p}+|u|^{\p})\in L^{1}(\R^{N})$. As a consequence of the Dominated Convergence Theorem, we get
$$
\int_{\R^{N}} G_{\eta, n}(x) \, dx\rightarrow 0 \quad \mbox{ as } n\rightarrow \infty.
$$
On the other hand, from the definition of $G_{\eta, n}$, it follows that
$$
|F(v_{n})-F(u_{n})+F(u)|\leq \eta(|u_{n}|^{p}+|u_{n}|^{\p})+G_{\eta, n}
$$
which together with the boundedness of $\{u_{n}\}$ in $L^{p}(\R^{N})\cap L^{\p}(\R^{N})$ yields
$$
\limsup_{n\rightarrow \infty} \int_{\R^{N}} |F(v_{n})-F(u_{n})+F(u)| \, dx\leq C\eta.
$$
The arbitrariness of $\eta$ ends the proof of \eqref{A}.

Now we prove \eqref{B}. Arguing as in Lemma 5.7 in \cite{Ding} we can find a subsequence $\{u_{n_{j}}\}$ such that for all $\eta>0$ there exists $r_{\eta}>0$ satisfying
\begin{equation}\label{Ding1}
\limsup_{j\rightarrow \infty} \int_{\B_{j}(0)\setminus \B_{r}(0)} |u_{n_{j}}|^{\tau} dx\leq \eta \quad \mbox{ and } \quad \int_{\R^{N}\setminus \B_{r}(0)} |u|^{\tau} dx\leq \eta
\end{equation}
for all $r\geq r_{\eta}$, where $\tau\in \{p, q\}$. Let $\phi: [0, \infty)\rightarrow [0,1]$ be a smooth function such that $\phi=1$ if $t\leq 1$ and $\phi=0$ if $t\geq 2$, and we  define $\tilde{u}_{j}(x)=\phi\left(\frac{2|x|}{j}\right)u(x)$. In view of Lemma \ref{Psi} we can see that
\begin{equation}\label{Ding2}
\tilde{u}_{j}\rightarrow u \quad \mbox{ in } W^{s,p}(\R^{N}).
\end{equation}
Set $h_{j}=u-\tilde{u}_{j}$. Then, for all $w\in \h$ it holds
\begin{align}\label{Ding3}
\int_{\R^{N}} &[f(u_{n_{j}})-f(v_{n_{j}})-f(u)]w \,dx= \nonumber \\
&=\int_{\R^{N}} [f(u_{n_{j}})-f(u_{n_{j}}-\tilde{u}_{j})-f(\tilde{u}_{j})]w \,dx \nonumber \\
&+\int_{\R^{N}} [f(v_{n_{j}}+h_{j})-f(v_{n_{j}})]w \,dx+\int_{\R^{N}} [f(\tilde{u}_{j})-f(u)]w \,dx \nonumber\\
&=I_{j}+II_{j}+III_{j}.
\end{align}
By using \eqref{Ding2} we obtain
\begin{equation}\label{Ding4}
\lim_{j\rightarrow \infty} \sup_{\|w\|_{\e}\leq 1} |III_{j}|=0.
\end{equation}
Now, we show that
\begin{equation}\label{Ding5}
\lim_{j\rightarrow \infty} \sup_{\|w\|_{\e}\leq 1}  |I_{j}|=0.
\end{equation}
Invoking Lemma \ref{embedding} and using $(f_2)$-$(f_3)$ we have, for all $r>0$,
\begin{equation}\label{Ding6}
\lim_{j\rightarrow \infty} \sup_{\|w\|_{\e}\leq 1}  \left|\int_{\B_{r}(0)}  [f(u_{n_{j}})-f(u_{n_{j}}-\tilde{u}_{j})-f(\tilde{u}_{j})]w \, dx\right|=0.
\end{equation}
On the other hand, from \eqref{Ding1} and the definition of $\tilde{u}_{j}$, it follows that
\begin{equation}\label{Dingnew}
\limsup_{j\rightarrow \infty} \int_{\B_{j}(0)\setminus \B_{r}(0)} |\tilde{u}_{j}|^{\tau} dx\leq \int_{\R^{N}\setminus \B_{r}(0)} |u|^{\tau} dx\leq \eta
\end{equation}
for all $r\geq r_{\eta}$. Then, by using $(f_2)$-$(f_3)$, H\"older inequality, Lemma \ref{embedding}, \eqref{Ding1} and \eqref{Dingnew} we get
\begin{align}\label{Ding7}
&\limsup_{j\rightarrow \infty} \left|\int_{\R^{N}}  [f(u_{n_{j}})-f(u_{n_{j}}-\tilde{u}_{j})-f(\tilde{u}_{j})]w \,dx \right|  \nonumber\\
&=\limsup_{j\rightarrow \infty} \left|\int_{\B_{j}(0)\setminus \B_{r}(0)}  [f(u_{n_{j}})-f(u_{n_{j}}-\tilde{u}_{j})-f(\tilde{u}_{j})]w \, dx\right| \nonumber \\
&\leq C\limsup_{j\rightarrow \infty} \left[|u_{n_{j}}|^{p-1}_{L^{p}(\B_{j}(0)\setminus \B_{r}(0))}+|\tilde{u}_{j}|^{p-1}_{L^{p}(\B_{j}(0)\setminus \B_{r}(0))}\right] \|w\|_{\e}  \nonumber\\
&+C\limsup_{j\rightarrow \infty} \left[|u_{n_{j}}|^{q-1}_{L^{q}(\B_{j}(0)\setminus \B_{r}(0))}+|\tilde{u}_{j}|^{q-1}_{L^{q}(\B_{j}(0)\setminus \B_{r}(0))}\right] \|w\|_{\e} \nonumber \\
&\leq C\eta^{\frac{p-1}{p}}+C\eta^{\frac{q-1}{q}}.
\end{align}
Putting together \eqref{Ding6} and \eqref{Ding7}, we can deduce that \eqref{Ding5} holds true.
Finally we verify that
\begin{equation}\label{Ding8}
\lim_{j\rightarrow \infty} \sup_{\|w\|_{\e}\leq 1}  |II_{j}|=0.
\end{equation}
Let us define $g(t)=\frac{f(t)}{|t|^{p-1}}$ if $t\neq 0$ and $g(0)=0$. In the light of $(f_1)$-$(f_3)$, we can see that $g\in \C^{0}(\R)$ and $|g(t)|\leq C(1+|t|^{q-p})$ for any $t\in \R$.
For all $a>0$ and $j\in \mathbb{N}$, we set $C_{j}^{a}=\{x\in \R^{N}: |v_{n_{j}}(x)|\leq a\}$ and $D_{j}^{a}=\R^{N}\setminus C_{j}^{a}$. Since $\{v_{n_{j}}\}$ is bounded in $L^{p}(\R^{N})$, we can see that $|D_{j}^{a}|\rightarrow 0$ as $a\rightarrow \infty$.
Then, in view of $(f_2)$-$(f_3)$, H\"older inequality  and the boundedness of $\{h_{j}\}$ we deduce that
\begin{align*}
\left|\int_{D_{j}^{a}} [f(v_{n_{j}}+h_{j})-f(v_{n_{j}})]w \,dx  \right|\leq C(|D_{j}^{a}|^{\frac{\p-p}{\p}}+|D_{j}^{a}|^{\frac{\p-q}{\p}})\|w\|_{\e},
\end{align*}
which implies that there exists $\tilde{a}=\tilde{a}_{\eta}>0$ satisfying
\begin{align}\label{Ding9}
\left|\int_{D_{j}^{\tilde{a}}} [f(v_{n_{j}}+h_{j})-f(v_{n_{j}})]w \,dx  \right|\leq \eta
\end{align}
uniformly in $\|w\|_{\e}\leq 1$. Since $g$ is uniformly continuous in $[-\tilde{a}, \tilde{a}]$, there exists $\delta>0$ such that
\begin{equation}\label{UCg}
|g(t+h)-g(t)|\leq \eta \quad \mbox{ for all } t\in [-\tilde{a}, \tilde{a}] \mbox{ and } |h|<\delta.
\end{equation}
Let $V_{j}^{\delta}=\{x\in \R^{N}: |h_{j}(x)|\leq \delta \}$ and $W_{j}^{\delta}=\R^{N}\setminus V_{j}^{\delta}$.
Noting that $|C_{j}^{\tilde{a}}\cap W_{j}^{\delta}|\leq |W_{j}^{\delta}|\rightarrow 0$ as $j\rightarrow \infty$, we can argue as before to infer that there exists $j_{0}\in \mathbb{N}$ such that
\begin{align}\label{Ding10}
\left|\int_{C_{j}^{\tilde{a}}\cap W_{j}^{\delta}} [f(v_{n_{j}}+h_{j})-f(v_{n_{j}})]w \, dx  \right|\leq \eta \quad \mbox{ for all } j\geq j_{0},
\end{align}
uniformly in $\|w\|_{\e}\leq 1$. Let us observe that
\begin{align*}
[f(v_{n_{j}}+h_{j})-f(v_{n_{j}})]w&=g(v_{n_{j}}+h_{j}) [|v_{n_{j}}+h_{j}|^{p-1}-|v_{n_{j}}|^{p-1}] w \\
& +[g(v_{n_{j}}+h_{j})-g(v_{n_{j}})] |v_{n_{j}}|^{p-1} w.
\end{align*}
In view of \eqref{Ding2}, we can find $j_{1}\in \mathbb{N}$ such that $j_{1}\geq j_{0}$ and
\begin{equation}\label{Ding22}
|h_{j}|_{L^{\tau}(\R^{N})}<\eta \quad \mbox{ for all } j\geq j_{1}.
\end{equation}
Taking into account \eqref{UCg}, \eqref{Ding22} and the boundedness of $\{v_{n_{j}}\}$ we can see that for all $j\geq j_{1}$
$$
\left|\int_{C_{j}^{\tilde{a}}\cap V_{j}^{\delta}} [g(v_{n_{j}}+h_{j})-g(v_{n_{j}})] |v_{n_{j}}|^{p-1} w \, dx\right|\leq C\eta, \quad \mbox{ uniformly in } \|w\|_{\e}\leq 1.
$$
Now, we recall the following inequalities for all $a, b\in \R$
\begin{align*}
||a|^{p-1}-|b|^{p-1}|\leq
\left\{
\begin{array}{ll}
|a-b|^{p-1} &\mbox{ if } p\in (1, 2] \\
(p-1)(|a|+|b|)^{p-2}|a-b|   &\mbox{ if } p> 2.
\end{array}
\right.
\end{align*}
Assume $p\in (1, 2]$. By using $|g(t)|\leq C(1+|t|^{q-p})$, H\"older inequality and \eqref{Ding22} we have
\begin{align*}
&\left|\int_{C_{j}^{\tilde{a}}\cap V_{j}^{\delta}} g(v_{n_{j}}+h_{j}) [|v_{n_{j}}+h_{j}|^{p-1}-|v_{n_{j}}|^{p-1}] w\, dx\right| \\
&\leq C\int_{\R^{N}} (1+|v_{n_{j}}+h_{j}|^{q-p})|h_{j}|^{p-1}|w| \,dx \\
&\leq C[|h_{j}|_{L^{p}(\R^{N})}^{p-1}|w|_{L^{p}(\R^{N})}+|v_{n_{j}}|_{L^{q}(\R^{N})}^{q-p}|h_{j}|^{p-1}_{L^{q}(\R^{N})}|w|_{L^{q}(\R^{N})}+|h_{j}|_{L^{q}(\R^{N})}^{q-1}|w|_{L^{q}(\R^{N})}]\\
&\leq C[\eta^{p-1}+\eta^{q-1}].
\end{align*}
When $p>2$, we can deduce that
\begin{align*}
&\left|\int_{C_{j}^{\tilde{a}}\cap V_{j}^{\delta}} g(v_{n_{j}}+h_{j}) [|v_{n_{j}}+h_{j}|^{p-1}-|v_{n_{j}}|^{p-1}] w \, dx \right| \\
&\leq C\int_{\R^{N}} (1+|v_{n_{j}}+h_{j}|^{q-p})(|v_{n_{j}}|^{p-2}|h_{j}|+|h_{j}|^{p-1})|w|\, dx\\
&\leq C[|v_{n_{j}}|_{L^{p}(\R^{N})}^{p-2}|h_{j}|_{L^{p}(\R^{N})}|w|_{L^{p}(\R^{N})}+|h_{j}|^{p-1}_{L^{p}(\R^{N})}|w|_{L^{p}(\R^{N})}\\
&+|v_{j}|_{L^{q}(\R^{N})}^{q-2}|h_{j}|_{L^{q}(\R^{N})}|w|_{L^{q}(\R^{N})}
+|v_{j}|_{L^{q}(\R^{N})}^{q-p}|h_{j}|_{L^{q}(\R^{N})}^{p-1}|w|_{L^{q}(\R^{N})}\\
&+|h_{j}|_{L^{q}(\R^{N})}^{q-p+1}|v_{j}|_{L^{q}(\R^{N})}^{p-2}|w|_{L^{q}(\R^{N})}+|h_{j}|^{q-1}_{L^{q}(\R^{N})}|w|_{L^{q}(\R^{N})} ]\\
&\leq C[\eta+\eta^{p-1}+\eta^{q-p+1}+\eta^{q-1}].
\end{align*}
From the above estimates, and using \eqref{Ding10} and $C_{j}^{\tilde{a}}=(C_{j}^{\tilde{a}}\cap V_{j}^{\delta})\cup (C_{j}^{\tilde{a}}\cap W_{j}^{\delta})$, we get
\begin{align*}
\left|\int_{C_{j}^{\tilde{a}}} [f(v_{n_{j}}+h_{j})-f(v_{n_{j}})] w \, dx\right|
\leq C\eta \quad \mbox{ for all } j\geq j_{1},
\end{align*}
uniformly in $\|w\|_{\e}\leq 1$, which together with \eqref{Ding9} yields \eqref{Ding8}.
\end{proof}

\section{Subcritical case}\label{Sect3}

\subsection{Functional setting in the subcritical case}\label{Sect3.1}
After a change of variable, we are led to consider the following problem
\begin{equation}\tag{$P_{\e}$}
\left\{
\begin{array}{ll}
(-\Delta)_{p}^{s} u +V(\e x) |u|^{p-2}u = f(u) &\mbox{ in } \R^{N}\\
u\in W^{s, p}(\R^{N}) \\
u(x)>0 &\mbox{ for all } x\in \R^{N}.
\end{array}
\right.
\end{equation}

\noindent
Weak solutions of ($P_{\e}$) can be obtained as critical points of the functional
\begin{equation*}
\I_{\e}(u)=\frac{1}{p} \|u\|_{\e}^{p} - \int_{\R^{N}} F(u) \, dx
\end{equation*}
which is well defined on $\h$.
It is standard to verify that $(f_2)$-$(f_3)$ yield that given $\xi >0$ there exists $C_{\xi}>0$ such that
\begin{align}
&|f(t)|\leq \xi |t|^{p-1} + C_{\xi} |t|^{q-1} \label{tv1}\\
&|F(t)| \leq \frac{\xi}{p} |t|^{p} + \frac{C_{\xi}}{q}|t|^{q} \label{tv2}.
\end{align}
On the other hand, hypothesis $(f_{5})$ implies that
\begin{equation}\label{increasing}
t\mapsto \frac{1}{p}f(t)t- F(t) \quad \mbox{ is increasing for any } t\geq 0.
\end{equation}
By using Lemma \ref{embedding}, it is easy to see that $\I_{\e}\in \C^{1}(\h, \R)$ and its differential $\I'_{\e}$ is given by
\begin{align*}
\langle \I'_{\e}(u), \varphi \rangle &= \iint_{\R^{2N}} \frac{|u(x)-u(y)|^{p-2}(u(x)- u(y))}{|x-y|^{N+sp}} (\varphi(x)- \varphi(y)) dxdy \\
&+  \int_{\R^{N}} V(\e x) |u|^{p-2} u\, \varphi dx - \int_{\R^{N}} f(u)\varphi \, dx
\end{align*}
for any $u, \varphi \in \h$. \\
Now, let us introduce the Nehari manifold associated to $\I_{\e}$, that is
\begin{equation*}
\N_{\e}= \left\{u\in \h\setminus \{0\} : \langle \I'_{\e}(u), u\rangle =0  \right\}.
\end{equation*}
Let us note that $\I_{\e}$ possesses a mountain pass geometry.
\begin{lemma}\label{lem2.1}
The functional $\I_{\e}$ satisfies the following conditions:
\begin{compactenum}[$(i)$]
\item there exist $\alpha, \rho >0$ such that $\I_{\e}(u)\geq \alpha$ with $\|u\|_{\e}=\rho$;
\item there exists $e\in \h$ with $\|e\|_{\e}>\rho$ such that $\I_{\e}(e)<0$.
\end{compactenum}
\end{lemma}

\begin{proof}
$(i)$ By using \eqref{tv2}, $V_{0}\leq V(\e x)$, and Lemma \ref{embedding} we get
\begin{align*}
\I_{\e}(u) &\geq \frac{1}{p} \|u\|_{\e}^{p}- \frac{\xi}{p V_{0}} \int_{\R^{N}} V(\e x) |u|^{p} dx - \frac{C_{\xi}}{q}\int_{\R^{N}} |u|^{q} dx\\
&\geq \left( \frac{1}{p}- \frac{\xi}{pV_{0}}\right) \|u\|_{\e}^{p} - \frac{C_{\xi}C_{q}}{q} \|u\|_{\e}^{q}.
\end{align*}
Choosing $\xi\in (0, V_{0})$, there exist $\alpha, \rho>0$ such that
\begin{equation*}
\I_{\e}(u)\geq \alpha >0 \quad \mbox{ with } \|u\|_{\e}= \rho.
\end{equation*}
$(ii)$ By $(f_4)$ we can infer
\begin{equation*}
F(t)\geq C_{1}|t|^{\vartheta} - C_{2} \quad \mbox{ for any } t\geq 0,
\end{equation*}
for some $C_{1}, C_{2}>0$. Taking $\varphi \in \C^{\infty}_{c}(\R^{N})\setminus \{0\}$ such that $\varphi\geq 0$ we have
\begin{equation*}
\I_{\e}(t\varphi)\leq \frac{t^{p}}{p} \|\varphi\|_{\e}^{p}- t^{\vartheta}C_{1} \int_{\supp \varphi} |\varphi|^{\vartheta} dx + C_{2}|\supp \varphi|\rightarrow -\infty \mbox{ as } t\rightarrow +\infty.
\end{equation*}
\end{proof}

\noindent
Since $f$ is only continuous, the next results are very important because they allow us to overcome the non-differentiability of $\N_{\e}$. We begin proving some properties for the functional $\I_{\e}$.
\begin{lemma}\label{SW1}
Under assumptions $(V)$ and $(f_1)$-$(f_5)$ we have for any $\e>0$:
\begin{compactenum}[$(i)$]
\item $\I'_{\e}$ maps bounded sets of $\h$ into bounded sets of $\h$.
\item $\I'_{\e}$ is weakly sequentially continuous in $\h$.
\item $\I_{\e}(t_{n}u_{n})\rightarrow -\infty$ as $t_{n}\rightarrow \infty$, where $u_{n}\in K$ and $K\subset \h\setminus\{0\}$ is a compact subset.
\end{compactenum}
\end{lemma}

\begin{proof}
$(i)$ Let $\{u_{n}\}$ be a bounded sequence in $\h$ and $v \in \h$. Then, by using $(f_{2})$ and $(f_{3})$, we deduce that
\begin{align*}
\langle \I'_{\e}(u_{n}), v \rangle &\leq C \|u_{n}\|_{\e}^{p-1} \|v\|_{\e} + C \|u_{n}\|_{\e}^{q-1} \|v\|_{\e} \leq C.
\end{align*}
$(ii)$ This is a consequence of $(f_2)$, $(f_3)$ and Lemma \ref{embedding}.

\noindent
$(iii)$ Without loss of generality, we may assume that $\|u\|_{\e}=1$ for each $u\in K$. For $u_{n}\in K$, after passing to a subsequence, we obtain that $u_{n}\rightarrow u\in \mathbb{S}_{\e}$. Then, by using $(f_{4})$ and Fatou's Lemma, we can see that
\begin{align*}
\I_{\e}(t_{n}u_{n}) &=\frac{t_{n}^{p}}{p}\|u_{n}\|_{\e}^{p}- \int_{\R^{N}} F(t_{n}u_{n}) \, dx \\
&\leq t_{n}^{\vartheta} \left(\frac{\|u_{n}\|_{\e}^{p}}{t_{n}^{\vartheta-p}} - \int_{\R^{N}} \frac{F(t_{n}u_{n})}{t_{n}^{\vartheta}} \, dx \right)\rightarrow -\infty \, \mbox{ as } n\rightarrow \infty.
\end{align*}
\end{proof}

\begin{lemma}\label{SW2}
Under the assumptions of Lemma \ref{SW1}, for $\e>0$ we have:
\begin{compactenum}[$(i)$]
\item for all $u\in \mathbb{S}_{\e}$, there exists a unique $t_{u}>0$ such that $t_{u}u\in \N_{\e}$. Moreover, $m_{\e}(u)=t_{u}u$ is the unique maximum of $\I_{\e}$ on $\h$, where $\mathbb{S}_{\e}=\{u\in \h: \|u\|_{\e}=1\}$.
\item The set $\N_{\e}$ is bounded away from $0$. Furthermore $\N_{\e}$ is closed in $\h$.
\item There exists $\alpha>0$ such that $t_{u}\geq \alpha$ for each $u\in \mathbb{S}_{\e}$ and, for each compact subset $W\subset \mathbb{S}_{\e}$, there exists $C_{W}>0$ such that $t_{u}\leq C_{W}$ for all $u\in W$.
\item For each $u\in \N_{\e}$, $m_{\e}^{-1}(u)=\frac{u}{\|u\|_{\e}}\in \N_{\e}$. In particular, $\N_{\e}$ is a regular manifold diffeomorphic to the sphere in $\h$.
\item $c_{\e}=\inf_{\N_{\e}} \I_{\e}\geq \rho>0$ and $\I_{\e}$ is bounded below on $\N_{\e}$, where $\rho$ is independent of $\e$.
\end{compactenum}
\end{lemma}

\begin{proof}
$(i)$ For each $u\in \mathbb{S}_{\e}$ and $t>0$, we define $h(t)=\I_{\e}(tu)$. From the proof of the Lemma \ref{lem2.1} we know that $h(0)=0$, $h(t)<0$ for $t$ large and $h(t)>0$ for $t$ small. Therefore, $\max_{t\geq 0} h(t)$ is achieved at $t=t_{u}>0$ satisfying $h'(t_{u})=0$ and $t_{u}u\in \N_{\e}$. By using $(f_{5})$, it is easy to verify the uniqueness of a such $t_{u}$.

\noindent
$(ii)$ By using \eqref{tv1} and Lemma \ref{embedding} we can see that for any $u\in\N_{\e}$
\begin{align*}
\|u\|_{\e}^{p} = \int_{\R^{N}} f(u)u \, dx \leq C\xi \|u\|_{\e}^{p} + C_{\xi} \|u\|_{\e}^{q}
\end{align*}
which implies that $\|u\|_{\e}\geq \kappa$ for some $\kappa>0$. \\
Now we prove that the set $\N_{\e}$ is closed in $\h$. Let $\{u_{n}\}\subset \N_{\e}$ such that $u_{n}\rightarrow u$ in $\h$. In view of Lemma \ref{SW1} we know that $\I'_{\e}(u_{n})$ is bounded, so we can deduce that
\begin{align*}
\langle \I'_{\e}(u_{n}), u_{n} \rangle - \langle \I'_{\e}(u), u \rangle= \langle \I'_{\e}(u_{n})- \I'_{\e}(u), u \rangle + \langle \I'_{\e}(u_{n}), u_{n}-u \rangle \rightarrow 0
\end{align*}
that is $\langle \I'_{\e}(u), u\rangle=0$. This combined with $\|u\|_{\e}\geq \kappa$ implies that
$$
\|u\|_{\e}= \lim_{n\rightarrow \infty} \|u_{n}\|_{\e} \geq \kappa >0,
$$
that is $u\in \N_{\e}$.

\noindent
$(iii)$ For each $u\in \mathbb{S}_{\e}$ there exists $t_{u}>0$ such that $t_{u}u\in \N_{\e}$. From the proof of $(ii)$, we can see that
$$
t_{u}= \|t_{u}u\|_{\e} \geq \kappa.
$$
Now we prove that $t_{u}\leq C_{W}$ for all $u\in W\subset \mathbb{S}_{\e}$. Assume by contradiction that there exists $\{u_{n}\}\subset W\subset \mathbb{S}_{\e}$ such that $t_{u_{n}}\rightarrow \infty$. Since $W$ is compact, we can find $u\in W$ such that $u_{n}\rightarrow u$ in $\h$ and $u_{n}\rightarrow u$ a.e. in $\R^{N}$. By using Lemma \ref{SW1}-$(iii)$, we can deduce that $\I_{\e}(t_{u_{n}}u_{n})\rightarrow -\infty$ as $n\rightarrow \infty$, which gives a contradiction because $(f_{4})$ implies that
$$
\I_{\e}(u)|_{\N_{\e}}= \int_{\R^{N}} \frac{1}{p}f(u)- F(u)\, dx \geq 0.
$$

\noindent
$(iv)$ Let us define the maps $\hat{m}_{\e}: \h\setminus \{0\} \rightarrow \N_{\e}$ and $m_{\e}: \mathbb{S}_{\e}\rightarrow \N_{\e}$ by setting
\begin{align}\label{me}
\hat{m}_{\e}(u)= t_{u}u \quad \mbox{ and } \quad m_{\e}= \hat{m}_{\e}|_{\mathbb{S}_{\e}}.
\end{align}
In view of $(i)$-$(iii)$ and Proposition 3.1 in \cite{SW} we can deduce that $m_{\e}$ is a homeomorphism between $\mathbb{S}_{\e}$ and $\N_{\e}$ and the inverse of $m_{\e}$ is given by $m_{\e}^{-1}(u)=\frac{u}{\|u\|_{\e}}$. Therefore $\N_{\e}$ is a regular manifold diffeomorphic to $\mathbb{S}_{\e}$.

\noindent
$(v)$ For $\e>0$, $t>0$ and $u\in \h \setminus \{0\}$, we can see that \eqref{tv2} yields
\begin{align*}
\I_{\e}(tu)\geq \frac{t^{p}}{p} (1- C\xi)\|u\|_{\e}^{p}-  t^{q} C_{\e} \|u\|_{\e}^{q},
\end{align*}
so we can find $\rho>0$ such that $\I_{\e}(tu)\geq \rho>0$ for $t>0$ small enough. On the other hand, by using $(i)$-$(iii)$, we know (see \cite{SW}) that
\begin{align}\label{Nguyen1}
c_{\e}= \inf_{u\in \N_{\e}} \I_{\e}(u)= \inf_{u\in \h\setminus \{0\}} \max_{t\geq 0} \I_{\e}(tu) = \inf_{u\in \mathbb{S}_{\e}} \max_{t\geq 0} \I_{\e}(tu)
\end{align}
which implies $c_{\e}\geq \rho$ and $\I_{\e}|_{\N_{\e}}\geq \rho$.
\end{proof}

\noindent
Now we introduce the following functionals $\hat{\Psi}_{\e}: \h\setminus\{0\} \rightarrow \R$ and $\Psi_{\e}: \mathbb{S}_{\e}\rightarrow \R$ defined by
\begin{align*}
\hat{\Psi}_{\e}= \I_{\e}(\hat{m}_{\e}(u)) \quad \mbox{ and } \quad \Psi_{\e}= \hat{\Psi}_{\e}|_{\mathbb{S}_{\e}},
\end{align*}
where $\hat{m}_{\e}(u)= t_{u}u$ is given in \eqref{me}.
As in \cite{SW} we have the following result:

\begin{lemma}\label{SW3}
Under the assumptions of Lemma \ref{SW1}, we have that for $\e>0$:
\begin{compactenum}[$(i)$]
\item $\Psi_{\e}\in \C^{1}(\mathbb{S}_{\e}, \R)$, and
\begin{equation*}
\Psi'_{\e}(w)v= \|m_{\e}(w)\|_{\e} \I'_{\e}(m_{\e}(w))v \quad \mbox{ for } v\in T_{w}(\mathbb{S}_{\e}).
\end{equation*}
\item $\{w_{n}\}$ is a Palais-Smale sequence for $\Psi_{\e}$ if and only if $\{m_{\e}(w_{n})\}$ is a Palais-Smale sequence for $\I_{\e}$. If $\{u_{n}\}\subset \N_{\e}$ is a bounded Palais-Smale sequence for $\I_{\e}$, then $\{m_{\e}^{-1}(u_{n})\}$ is a Palais-Smale sequence for $\Psi_{\e}$.
\item $u\in \mathbb{S}_{\e}$ is a critical point of $\Psi_{\e}$ if and only if $m_{\e}(u)$ is a critical point of $\I_{\e}$. Moreover the corresponding critical values coincide and
\begin{equation*}
\inf_{\mathbb{S}_{\e}} \Psi_{\e}=\inf_{\N_{\e}} \I_{\e}=c_{\e}.
\end{equation*}
\end{compactenum}
\end{lemma}

We conclude this section proving the following useful result.
\begin{lemma}\label{lemB}
If $\{u_{n}\}$ is a Palais-Smale sequence of $\I_{\e}$ at level $c$, then $\{u_{n}\}$ is bounded in $\h$.
\end{lemma}
\begin{proof}
By using assumption $(f_{4})$ we have
\begin{align*}
c+o_{n}(1)\|u_{n}\|_{\e} &= \I_{\e}(u_{n})- \frac{1}{\vartheta} \langle \I'_{\e}(u_{n}), u_{n} \rangle \\
&= \left( \frac{1}{p}- \frac{1}{\vartheta}\right) \|u_{n}\|_{\e}^{p} + \frac{1}{\vartheta} \int_{\R^{N}} \left( f(u_{n})u_{n}- \vartheta F(u_{n})\right) \, dx \\
&\geq \left( \frac{1}{p}- \frac{1}{\vartheta}\right) \|u_{n}\|_{\e}^{p},
\end{align*}
and being $\vartheta>p$ we get the thesis.
\end{proof}

\subsection{Autonomous subcritical problem}\label{Sect3.2}
Let us consider the autonomous problem associated to ($P_{\e}$), that is
\begin{equation*}\tag{$P_{\mu}$}
\left\{
\begin{array}{ll}
(-\Delta)_{p}^{s} u + \mu |u|^{p-2}u = f(u) &\mbox{ in } \R^{N} \\
u\in W^{s, p}(\R^{N}) \\
u(x)>0 &\mbox{ for all } x\in \R^{N}, \mu>0.
\end{array}
\right.
\end{equation*}
The corresponding functional is given by
\begin{equation*}
\J_{\mu}(u)=\frac{1}{p} \|u\|_{\mu}^{p} - \int_{\R^{N}} F(u) \, dx
\end{equation*}
which is well defined on the space $\X_{\mu}=W^{s, p}(\R^{N})$ endowed with the norm
\begin{equation*}
\|u\|_{\mu}^{p}:=  [u]_{W^{s,p}(\R^{N})}^{p}+\mu |u|^{p}_{L^{p}(\R^{N})}.
\end{equation*}
Clearly, $\J_{\mu}\in \C^{1}(\X_{\mu}, \R)$ and its differential $\J'_{\mu}$ is given by
\begin{align*}
\langle \J'_{\mu}(u), \varphi \rangle &= \iint_{\R^{2N}} \frac{|u(x)-u(y)|^{p-2}(u(x)- u(y))}{|x-y|^{N+sp}} (\varphi(x)- \varphi(y)) \,dxdy \\
&+ \mu \int_{\R^{N}} |u|^{p-2} u\, \varphi \, dx - \int_{\R^{N}} f(u)\varphi \, dx
\end{align*}
for any $u, \varphi \in \X_{\mu}$.
Let us define the Nehari manifold associated to $\J_{\mu}$, that is
\begin{equation*}
\N_{\mu}= \left\{u\in \X_{\mu}\setminus \{0\} : \langle \J'_{\mu}(u), u\rangle =0  \right\}.
\end{equation*}
We note that $(f_{4})$ yields
\begin{align}\label{coercive}
\J_{\mu}(u)= \int_{\R^{N}} \frac{1}{p} f(u)u- F(u) \, dx \geq \left(\frac{1}{p}-\frac{1}{\vartheta}\right) \|u\|_{\mu}^{p} \quad \mbox{ for all } u\in \N_{\mu}.
\end{align}
Arguing as in the previous section and using \eqref{coercive}, it is easy to prove the following lemma.

\begin{lemma}\label{SW2A}
Under the assumptions of Lemma \ref{SW1}, for $\mu>0$ we have:
\begin{compactenum}[$(i)$]
\item for all $u\in \mathbb{S}_{\mu}$, there exists a unique $t_{u}>0$ such that $t_{u}u\in \N_{\mu}$. Moreover, $m_{\mu}(u)=t_{u}u$ is the unique maximum of $\J_{\mu}$ on $\h$, where $\mathbb{S}_{\mu}=\{u\in \X_{\mu}: \|u\|_{\mu}=1\}$.
\item The set $\N_{\mu}$ is bounded away from $0$. Furthermore $\N_{\mu}$ is closed in $\X_{\mu}$.
\item There exists $\alpha>0$ such that $t_{u}\geq \alpha$ for each $u\in \mathbb{S}_{\mu}$ and, for each compact subset $W\subset \mathbb{S}_{\mu}$, there exists $C_{W}>0$ such that $t_{u}\leq C_{W}$ for all $u\in W$.
\item $\N_{\mu}$ is a regular manifold diffeomorphic to the sphere in $\X_{\mu}$.
\item $c_{\mu}=\inf_{\N_{\mu}} \J_{\mu}>0$ and $\J_{\mu}$ is bounded below on $\N_{\mu}$ by some positive constant.
\item $\J_{\mu}$ is coercive on $\N_{\mu}$.
\end{compactenum}
\end{lemma}

\noindent
Now we define the following functionals $\hat{\Psi}_{\mu}: \X_{\mu}\setminus\{0\} \rightarrow \R$ and $\Psi_{\mu}: \mathbb{S}_{\mu}\rightarrow \R$ by setting
\begin{align*}
\hat{\Psi}_{\mu}= \J_{\mu}(\hat{m}_{\mu}(u)) \quad \mbox{ and } \quad \Psi_{\mu}= \hat{\Psi}_{\mu}|_{\mathbb{S}_{\mu}}.
\end{align*}
Then we have the following result:

\begin{lemma}\label{SW3A}
Under the assumptions of Lemma \ref{SW1}, we have that for $\mu>0$:
\begin{compactenum}[$(i)$]
\item $\Psi_{\mu}\in \C^{1}(\mathbb{S}_{\mu}, \R)$, and
\begin{equation*}
\Psi'_{\mu}(w)v= \|m_{\mu}(w)\|_{\mu} \J'_{\mu}(m_{\mu}(w))v \quad \mbox{ for } v\in T_{w}(\mathbb{S}_{\mu}).
\end{equation*}
\item $\{w_{n}\}$ is a Palais-Smale sequence for $\Psi_{\mu}$ if and only if $\{m_{\mu}(w_{n})\}$ is a Palais-Smale sequence for $\J_{\mu}$. If $\{u_{n}\}\subset \N_{\mu}$ is a bounded Palais-Smale sequence for $\J_{\mu}$, then $\{m_{\mu}^{-1}(u_{n})\}$ is a Palais-Smale sequence for $\Psi_{\mu}$.
\item $u\in \mathbb{S}_{\mu}$ is a critical point of $\Psi_{\mu}$ if and only if $m_{\mu}(u)$ is a critical point of $\J_{\mu}$. Moreover the corresponding critical values coincide and
\begin{equation*}
\inf_{\mathbb{S}_{\mu}} \Psi_{\mu}=\inf_{\N_{\mu}} \J_{\mu}=c_{\mu}.
\end{equation*}
\end{compactenum}
\end{lemma}

\begin{remark}
As in \eqref{Nguyen1}, from $(i)$-$(iii)$ of Lemma \ref{SW2}, we can see that $c_{\mu}$ admits the following minimax characterization
\begin{align}\label{Nguyen2}
c_{\mu}= \inf_{u\in \N_{\mu}} \J_{\mu}(u)= \inf_{u\in \X_{\mu}\setminus \{0\}} \max_{t\geq 0} \J_{\mu}(tu) = \inf_{u\in \mathbb{S}_{\mu}} \max_{t\geq 0} \J_{\mu}(tu).
\end{align}
\end{remark}

\begin{lemma}\label{lem2.2a}
Let $\{u_{n}\}\subset \N_{\mu}$ be a minimizing sequence for $\J_{\mu}$. Then, $\{u_{n}\}$ is bounded and there exist a sequence $\{y_{n}\}\subset \R^{N}$ and constants $R, \beta>0$ such that
\begin{equation*}
\liminf_{n\rightarrow \infty} \int_{\B_{R}(y_{n})} |u_{n}|^{p} dx \geq \beta >0.
\end{equation*}
\end{lemma}

\begin{proof}
Arguing as in the proof of Lemma \ref{lemB}, we can see that $\{u_{n}\}$ is bounded in $\X_{\mu}$. Now, in order to prove the latter conclusion of this lemma, we argue by contradiction.
Assume that for any $R>0$ it holds
\begin{equation*}
\lim_{n\rightarrow \infty} \sup_{y\in \R^{N}} \int_{\B_{R}(y)} |u_{n}|^{p} dx=0.
\end{equation*}
Since $\{u_{n}\}$ is bounded in $\X_{\mu}$, from Lemma \ref{Lions} it follows that
\begin{equation}\label{tv4N}
u_{n}\rightarrow 0 \mbox{ in } L^{t}(\R^{N}) \quad \mbox{ for any } t\in (p, \p).
\end{equation}
Fix $\xi\in (0, \mu)$. By using $\langle \J'_{\mu}(u_{n}), u_{n} \rangle =0$, \eqref{tv1} and the fact that $\{u_{n}\}$ is bounded in $\X_{\mu}$, we have
\begin{align*}
0\leq \|u_{n}\|_{\mu}^{p}
\leq \xi \int_{\R^{N}} |u_{n}|^{p} dx + C_{\xi} \int_{\R^{N}} |u_{n}|^{q} dx  \leq \frac{\xi}{\mu} \|u_{n}\|_{\mu}^{p} + C_{\xi} |u_{n}|_{L^{q}(\R^{N})}^{q},
\end{align*}
from which we deduce that
$$
\left(1-\frac{\xi}{\mu}\right)\|u_{n}\|_{\mu}^{p}\leq C_{\xi} |u_{n}|_{L^{q}(\R^{N})}^{q}.
$$
In view of \eqref{tv4N}, we can conclude that $u_{n}\rightarrow 0$ in $\X_{\mu}$.
\end{proof}

\noindent
Now, we prove the following useful compactness result for the autonomous problem.
\begin{lemma}\label{lem4.3}
The problem $(P_{\mu})$ has at least one positive ground state solution.
\end{lemma}

\begin{proof}
From $(v)$ of Lemma \ref{SW2A}, we know that $c_{\mu}>0$ for each $\mu>0$. Moreover, if $u\in \N_{\mu}$ verifies $\J_{\mu}(u)=c_{\mu}$, then $m^{-1}_{\mu}(u)$ is a minimizer of $\Psi_{\mu}$ and it is a critical point of $\Psi_{\mu}$. In view of Lemma \ref{SW3A}, we can see that $u$ is a critical point of $\J_{\mu}$. Now we show that there exists a minimizer of $\J_{\mu}|_{\N_{\mu}}$. By applying Ekeland's variational principle \cite{Ek} there exists a sequence $\{\nu_{n}\}\subset \mathbb{S}_{\mu}$ such that $\Psi_{\mu}(\nu_{n})\rightarrow c_{\mu}$ and $\Psi'_{\mu}(\nu_{n})\rightarrow 0$ as $n\rightarrow \infty$. Let $u_{n}=m_{\mu}(\nu_{n}) \in \N_{\mu}$. Then, thanks to Lemma \ref{SW3A}, $\J_{\mu}(u_{n})\rightarrow c_{\mu}$ and $\J'_{\mu}(u_{n})\rightarrow 0$ as $n\rightarrow \infty$.
Therefore, arguing as in Lemma \ref{lemB}, $\{u_{n}\}$ is bounded in $\X_{\mu}$ and $u_{n}\rightharpoonup u$ in $W^{s, p}(\R^{N})$. It is easy to check that $\J_{\mu}'(u)=0$.
Now, we prove that $\J_{\mu}(u)=c_{\mu}$. Assume $u\not \equiv 0$ and we aim to show that
\begin{equation}\label{lim1}
\|u_{n}\|^{p}_{\mu}\rightarrow \|u\|^{p}_{\mu}.
\end{equation}
In fact, once proved the previous limit, we can use Lemma \ref{lemPSY} to deduce that $u_{n}\rightarrow u$ in $\X_{\mu}$, and recalling that $\J_{\mu}(u_{n})\rightarrow c_{\mu}$, we obtain the thesis.\\
Now, we prove \eqref{lim1}. Let us observe that Fatou's Lemma yields
\begin{equation}\label{inf}
\|u\|^{p}_{\mu}\leq \liminf_{n\rightarrow \infty} \|u_{n}\|^{p}_{\mu}.
\end{equation}
Assume by contradiction that
\begin{equation}\label{sup}
\|u\|^{p}_{\mu}<\limsup_{n\rightarrow \infty} \|u_{n}\|^{p}_{\mu}.
\end{equation}
Let us note that
\begin{align}\begin{split}\label{tervince1}
c_{\mu}+o_{n}(1) &= \J_{\mu}(u_{n})-\frac{1}{\vartheta} \langle \J'_{\mu}(u_{n}), u_{n}\rangle \\
&=\left(\frac{1}{p}-\frac{1}{\vartheta}\right)\|u_{n}\|^{p}_{\mu}+ \int_{\R^{N}} \left[\frac{1}{\vartheta} f(u_{n})u_{n}-F(u_{n})\right] dx.
\end{split}\end{align}
Recalling that
$$
\limsup_{n\rightarrow \infty}\, (a_{n}+b_{n})\geq \limsup_{n\rightarrow \infty} a_{n}+\liminf_{n\rightarrow \infty} b_{n}
$$
and $\vartheta >p$, we can see that Fatou's Lemma, \eqref{sup}, \eqref{tervince1} and $\J_{\mu}'(u)=0$ produce
\begin{align*}
c_{\mu}&\geq \left(\frac{1}{p}-\frac{1}{\vartheta}\right)\limsup_{n\rightarrow \infty} \|u_{n}\|^{p}_{\mu}+\liminf_{n\rightarrow \infty} \int_{\R^{N}} \left[\frac{1}{\vartheta} f(u_{n})u_{n}-F(u_{n})\right] dx \\
&>\left(\frac{1}{p}-\frac{1}{\vartheta}\right) \|u\|^{p}_{\mu}+\int_{\R^{N}} \left[\frac{1}{\vartheta} f(u)u-F(u)\right] dx \\
&=\J_{\mu}(u)-\frac{1}{\vartheta} \langle \J'_{\mu}(u), u\rangle=\J_{\mu}(u)\geq c_{\mu}
\end{align*}
which gives a contradiction.

Finally, we consider the case $u=0$. Arguing as in the proof of Lemma \ref{lem2.2a}, we can find a sequence $\{y_{n}\}\subset \R^{N}$ and constants $R, \beta>0$ such that
$$
\int_{\B_{R}(y_{n})} |u_{n}|^{p}dx\geq \beta>0.
$$
Set $v_{n}:=u_{n}(\cdot+y_{n})$. Then, by using the invariance by translations of $\R^{N}$, it is clear that $\{v_{n}\}$ is a $(PS)_{c_{\mu}}$ for $\J_{\mu}$, $\{v_{n}\}\subset \N_{\mu}$ and $v_{n}\rightharpoonup v\neq 0$ in $W^{s, p}(\R^{N})$. Thus, we can proceed as above to deduce that $\{v_{n}\}$ converges strongly in $W^{s, p}(\R^{N})$.
\end{proof}

\begin{remark}
Let us observe that the ground state obtained in Lemma \ref{lem4.3} is positive. Indeed, $\langle \J'_{\mu}(u), u^{-}\rangle=0$, $f(t)=0$ for $t\leq 0$ and $|x-y|^{p-2}(x-y)(x^{-}-y^{-}) \geq |x^{-}-y^{-}|^{p}$, where $x^{-}= \min\{x, 0\}$, yield
\begin{align*}
&\|u^{-}\|^{p}_{\mu}  \\
&\leq \iint_{\R^{2N}} \frac{|u(x)- u(y)|^{p-2}(u(x)- u(y))}{|x-y|^{N+sp}} (u^{-}(x)- u^{-}(y)) \, dxdy + \int_{\R^{N}} \mu |u|^{p-2} u u^{-} \, dx \\
&= \int_{\R^{N}} f(u)u^{-} \, dx =0
\end{align*}
which implies that $u^{-}=0$, that is $u\geq 0$. Arguing as in Lemma \ref{lemMoser} below, we can see that $u\in L^{\infty}(\R^{N})\cap \C^{0}(\R^{N})$, and by applying the maximum principle \cite{DPQ} we deduce that $u>0$ in $\R^{N}$.
\end{remark}

\subsection{Existence result for \eqref{P}}

\noindent
In this section we focus on the existence of a solution to \eqref{P} provided that $\e$ is sufficiently small. Let us begin proving the following useful lemma.
\begin{lemma}\label{lem2.2}
Let $\{u_{n}\}\subset \N_{\e}$ be a sequence such that $\I_{\e}(u_{n})\rightarrow c$ and $u_{n}\rightharpoonup 0$ in $\h$. Then, one of the following alternatives occurs
\begin{compactenum}[$(a)$]
\item $u_{n}\rightarrow 0$ in $\h$;
\item there are a sequence $\{y_{n}\}\subset \R^{N}$ and constants $R, \beta>0$ such that
\begin{equation*}
\liminf_{n\rightarrow \infty} \int_{\B_{R}(y_{n})} |u_{n}|^{p} dx \geq \beta >0.
\end{equation*}
\end{compactenum}
\end{lemma}

\begin{proof}
Assume that $(b)$ does not hold true. Then, for any $R>0$ it holds
\begin{equation*}
\lim_{n\rightarrow \infty} \sup_{y\in \R^{N}} \int_{\B_{R}(y)} |u_{n}|^{p} dx=0.
\end{equation*}
Since $\{u_{n}\}$ is bounded in $\h$, from Lemma \ref{Lions} it follows that
\begin{equation}\label{tv4}
u_{n}\rightarrow 0 \mbox{ in } L^{t}(\R^{N}) \quad \mbox{ for any } t\in (p, \p).
\end{equation}
Now, we can argue as in the proof of Lemma \ref{lem2.2a} to deduce that $\|u_{n}\|_{\e}\rightarrow 0$ as $n\rightarrow \infty$.
\end{proof}

In order to get a compactness result for $\I_{\e}$, we need to prove the following auxiliary lemma.
\begin{lemma}\label{lem2.3}
Assume that $V_{\infty}<\infty$ and let $\{v_{n}\}\subset \N_{\e}$ be a sequence such that $\I_{\e}(v_{n})\rightarrow d$ with $v_{n}\rightharpoonup 0$ in $\h$. If $v_{n}\not \rightarrow 0$ in $\h$, then $d\geq c_{V_{\infty}}$, where $c_{V_{\infty}}$ is the infimum of $\J_{V_{\infty}}$ over $\N_{V_{\infty}}$.
\end{lemma}

\begin{proof}
Let $\{t_{n}\}\subset (0, +\infty)$ be such that $\{t_{n}v_{n}\}\subset \N_{V_{\infty}}$. \\
\textsc{Claim 1}: We aim to prove that
\begin{equation*}
\limsup_{n\rightarrow \infty} t_{n} \leq 1.
\end{equation*}
Assume by contradiction that there exist $\delta>0$ and a subsequence, still denoted by $\{t_{n}\}$, such that
\begin{equation}\label{tv6}
t_{n}\geq 1+ \delta \quad \mbox{ for any } n\in \mathbb{N}.
\end{equation}
Since $\langle \I'_{\e}(v_{n}), v_{n} \rangle =0$, we deduce that
\begin{equation}\label{tv7}
[v_{n}]^{p}_{W^{s, p}(\R^{N})} + \int_{\R^{N}} V(\e x) |v_{n}|^{p} dx = \int_{\R^{N}} f(v_{n}) v_{n} dx.
\end{equation}
In view of $t_{n}v_{n} \in \N_{V_{\infty}}$, we also have
\begin{equation}\label{tv8}
t_{n}^{p} [v_{n}]^{p}_{W^{s, p}(\R^{N})}+ t_{n}^{p} V_{\infty}\int_{\R^{N}} |v_{n}|^{p} dx = \int_{\R^{N}} f(t_{n}v_{n}) t_{n}v_{n} dx.
\end{equation}
Putting together \eqref{tv7} and \eqref{tv8} we obtain
\begin{equation*}
\int_{\R^{N}} \left( \frac{f(t_{n}v_{n}) (v_{n})^{p}}{(t_{n} v_{n})^{p-1}} - \frac{f(v_{n}) (v_{n})^{p}}{(v_{n})^{p-1}} \right)  \,dx= \int_{\R^{N}} \left( V_{\infty} - V(\e x)\right) |v_{n}|^{p} dx.
\end{equation*}
By hypothesis $(V)$ we can see that, given $\zeta>0$ there exists $R=R(\zeta)>0$ such that
\begin{equation}\label{tv9}
V(\e x) \geq V_{\infty} - \zeta \quad \mbox{ for any } |x|\geq R.
\end{equation}
Now, taking into account $v_{n}\rightarrow 0$ in $L^{p}(\B_{R}(0))$ and the boundedness of $\{v_{n}\}$ in $\h$, we can infer that
\begin{align*}
\int_{\R^{N}} & \left( V_{\infty} - V(\e x)\right) |v_{n}|^{p} dx\\
&= \int_{\B_{R}(0)} \left( V_{\infty} - V(\e x)\right) |v_{n}|^{p} dx+ \int_{\B_{R}^{c}(0)} \left( V_{\infty} - V(\e x)\right) |v_{n}|^{p} dx\\
&\leq V_{\infty}\int_{\B_{R}(0)} |v_{n}|^{p} dx + \zeta \int_{\B_{R}^{c}(0)} |v_{n}|^{p} dx\\
&\leq o_{n}(1) + \frac{\zeta}{V_{0}} \|v_{n}\|_{\e}^{p}\leq o_{n}(1)+ \zeta C.
\end{align*}
Thus,
\begin{equation}\label{tv10}
\int_{\R^{N}} \left( \frac{f(t_{n}v_{n})}{(t_{n} v_{n})^{p-1}} - \frac{f(v_{n})}{(v_{n})^{p-1}} \right) v_{n}^{p}  \,dx\leq \zeta C +o_{n}(1).
\end{equation}
Since $v_{n} \not \rightarrow 0$ in $\h$, we can apply Lemma \ref{lem2.2} to deduce the existence of a sequence $\{y_{n}\}\subset \R^{N}$, and two positive numbers $\bar{R}, \beta$ such that
\begin{equation}\label{tv11}
\int_{\B_{\bar{R}}(y_{n})} |v_{n}|^{p} dx \geq \beta>0.
\end{equation}
Let us consider $\bar{v}_{n}= v_{n}(x+y_{n})$. From condition $(V)$ and the boundedness of $\{v_{n}\}$ in $\h$, we can see that
\begin{align*}
\|\bar{v}_{n}\|_{V_{0}}^{p}= \|v_{n}\|_{V_{0}}^{p} \leq [v_{n}]^{p}_{W^{s, p}(\R^{N})}+ \int_{\R^{N}} V(\e x) |v_{n}|^{p} dx = \|v_{n}\|_{\e}^{p}\leq C,
\end{align*}
therefore $\{\bar{v}_{n}\}$ is bounded in $W^{s, p}(\R^{N})$. Taking into account that $W^{s, p}(\R^{N})$ is a reflexive Banach space, we may assume that $\bar{v}_{n}\rightharpoonup \bar{v}$ in $W^{s, p}(\R^{N})$. By \eqref{tv11} there exists $\Omega \subset \R^{N}$ with positive measure and such that $\bar{v}>0$ in $\Omega$. By using \eqref{tv6}, assumption $(f_{5})$ and \eqref{tv10} we can infer
\begin{align*}
0<\int_{\Omega} \left(\frac{f((1+\delta) \bar{v}_{n})}{((1+\delta) \bar{v}_{n})^{p-1}}- \frac{f(\bar{v}_{n})}{(\bar{v}_{n})^{p-1}} \right) \bar{v}_{n}^{p}dx \leq \zeta C+o_{n}(1).
\end{align*}
Letting the limit as $n\rightarrow \infty$ and by applying Fatou's Lemma we obtain
\begin{align*}
0<\int_{\Omega} \left(\frac{f((1+\delta) \bar{v})}{((1+\delta) \bar{v})^{p-1}}- \frac{f(\bar{v})}{(\bar{v})^{p-1}} \right) \bar{v}^{p}dx \leq \zeta C
\end{align*}
for any $\zeta>0$, and this is a contradiction. \\
Now, we distinguish the following cases: \\
\textsc{Case 1:} Assume that $\limsup_{n\rightarrow \infty} t_{n}=1$. Thus there exists $\{t_{n}\}$ such that $t_{n}\rightarrow 1$. Recalling that  $\I_{\e}(v_{n})\rightarrow d$, we have
\begin{align}\label{tv12new}
d+ o_{n}(1)&= \I_{\e}(v_{n})\nonumber \\
&=\I_{\e}(v_{n}) - \J_{V_{\infty}}(t_{n}v_{n})+ \J_{V_{\infty}}(t_{n}v_{n}) \nonumber \\
&\geq \I_{\e}(v_{n}) -\J_{V_{\infty}}(t_{n}v_{n}) + c_{V_{\infty}}.
\end{align}
Let us compute $\I_{\e}(v_{n}) -\J_{V_{\infty}}(t_{n}v_{n})$:
\begin{align}\begin{split}\label{tv12}
&\I_{\e}(v_{n}) -\J_{V_{\infty}}(t_{n}v_{n}) \\
&= \frac{1-t_{n}^{p}}{p} [v_{n}]^{p}_{W^{s, p}(\R^{N})} + \frac{1}{p} \int_{\R^{N}} (V(\e x) - t_{n}^{p} V_{\infty}) |v_{n}|^{p} dx + \int_{\R^{N}} \!(F(t_{n} v_{n}) -F(v_{n}))  \,dx.
\end{split} \end{align}
Now, by using condition $(V)$,  $v_{n}\rightarrow 0$ in $L^{p}(\B_{R}(0))$, $t_{n}\rightarrow 1$, \eqref{tv9}, and
\begin{align*}
V(\e x) - t_{n}^{p} V_{\infty} =\left(V(\e x) - V_{\infty} \right) + (1- t_{n}^{p}) V_{\infty}\geq -\zeta + (1- t_{n}^{p}) V_{\infty} \quad \mbox{ for } |x|\geq R,
\end{align*}
we get
\begin{align}\label{tv13}
&\int_{\R^{N}} \left( V(\e x) - t_{n}^{p} V_{\infty}\right) |v_{n}|^{p} dx \nonumber \\
&= \int_{\B_{R}(0)} \left( V(\e x) - t_{n}^{p} V_{\infty}\right) |v_{n}|^{p} dx+ \int_{\B_{R}^{c}(0)} \left( V(\e x) - t_{n}^{p} V_{\infty}\right) |v_{n}|^{p} dx \nonumber \\
&\geq (V_{0}- t_{n}^{p}V_{\infty}) \int_{\B_{R}(0)} |v_{n}|^{p} dx - \zeta \int_{\B_{R}^{c}(0)} |v_{n}|^{p} dx+ V_{\infty}(1- t_{n}^{p}) \int_{\B_{R}^{c}(0)} |v_{n}|^{p} dx \nonumber \\
&\geq o_{n}(1)- \zeta C.
\end{align}
On the other hand, since $\{v_{n}\}$ is bounded in $\h$, we can see that
\begin{align}\label{tv14}
\frac{(1-t_{n}^{p})}{p} [v_{n}]^{p}_{W^{s, p}(\R^{N})}= o_{n}(1).
\end{align}
Hence, putting together \eqref{tv12}, \eqref{tv13} and \eqref{tv14}, we obtain
\begin{align}\label{tv15}
\I_{\e}(v_{n}) -\J_{V_{\infty}}(t_{n}v_{n}) = \int_{\R^{N}} \left( F(t_{n} v_{n}) -F(v_{n}) \right) \, dx +o_{n}(1)- \zeta C.
\end{align}
At this point, we show that
\begin{align}\label{tv16}
\int_{\R^{N}} \left( F(t_{n} v_{n}) -F(v_{n}) \right) \, dx=o_{n}(1).
\end{align}
Indeed, by using the Mean Value Theorem and \eqref{tv1} we have
\begin{align*}
\int_{\R^{N}} | F(t_{n} v_{n}) -F(v_{n}) | \, dx \leq C|t_{n}-1| \left( |v_{n}|_{L^{p}(\R^{N})}^{p}+ |v_{n}|_{L^{p}(\R^{N})}^{p} \right),
\end{align*}
and taking into account the boundedness of $\{v_{n}\}$ in $\h$ we get the thesis. Now, putting together \eqref{tv12new}, \eqref{tv15} and \eqref{tv16} we can infer that
\begin{align*}
d+ o_{n}(1)\geq o_{n}(1) - \zeta C + c_{V_{\infty}},
\end{align*}
and passing to the limit as $\zeta \rightarrow 0$ we get $d \geq c_{V_{\infty}}$. \\
\textsc{Case 2:} Assume that $\limsup_{n\rightarrow \infty} t_{n}=t_{0}<1$. Then there is a subsequence, still denoted by $\{t_{n}\}$, such that $t_{n}\rightarrow t_{0} (<1)$ and $t_{n}<1$ for any $n\in \mathbb{N}$.
Let us observe that
\begin{align}\label{tv17}
d+o_{n}(1)= \I_{\e}(v_{n}) - \frac{1}{p}\langle \I'_{\e}(v_{n}), v_{n} \rangle = \int_{\R^{N}} \left(\frac{1}{p}f(v_{n}) v_{n} - F(v_{n})\right) \,dx.
\end{align}
Exploiting the facts that $t_{n}v_{n}\in \N_{V_{\infty}}$,  \eqref{increasing} and \eqref{tv17}, we obtain
\begin{align*}
c_{V_{\infty}} \leq \J_{V_{\infty}}(t_{n}v_{n})  &= \J_{V_{\infty}}(t_{n}v_{n}) - \frac{1}{p} \langle \J'_{V_{\infty}}(t_{n}v_{n}), t_{n}v_{n} \rangle \\
&= \int_{\R^{N}} \left(\frac{1}{p} f(t_{n}v_{n}) t_{n}v_{n}- F(t_{n}v_{n})\right) \, dx \\
&\leq \int_{\R^{N}} \left(\frac{1}{p}f(v_{n}) v_{n} - F(v_{n})\right) \,dx =d +o_{n}(1).
\end{align*}
Taking the limit as $n\rightarrow \infty$ we get $d\geq c_{V_{\infty}}$.
\end{proof}

\noindent
At this point we are able to prove the following compactness result.
\begin{proposition}\label{prop2.1}
Let $\{u_{n}\}\subset \N_{\e}$ be such that $\I_{\e}(u_{n})\rightarrow c$, where $c<c_{V_{\infty}}$ if $V_{\infty}<\infty$ and $c\in \R$ if $V_{\infty}=\infty$. Then $\{u_{n}\}$ has a convergent subsequence in $\h$.
\end{proposition}

\begin{proof}
It is easy to see that $\{u_{n}\}$ is bounded in $\h$. Then, up to a subsequence, we may assume that
\begin{align}\begin{split}\label{conv}
&u_{n}\rightharpoonup u \mbox{ in } \h, \\
&u_{n}\rightarrow u \mbox{ in } L^{q}_{loc}(\R^{N}) \quad \mbox{ for any } q\in [p, \p), \\
&u_{n} \rightarrow u \mbox{ a.e. in } \R^{N}.
\end{split}\end{align}
By using assumptions $(f_{2})$-$(f_{3})$, \eqref{conv} and the fact that $\C^{\infty}_{c}(\R^{N})$ is dense in $W^{s, p}(\R^{N})$, it is standard to check that $\I'_{\e}(u)=0$. \\
Now, let $v_{n}= u_{n}-u$.
By using Lemma \ref{lemPSY} and \eqref{A} of Lemma \ref{splitting} we have
\begin{align}\label{tv19}
\I_{\e}(v_{n})&=\frac{\|u_{n}\|_{\e}^{p}}{p} - \frac{\|u\|_{\e}^{p}}{p}- \int_{\R^{N}} F(u_{n}) \, dx+ \int_{\R^{N}} F(u) \, dx + o_{n}(1) \nonumber \\
&=\I_{\e}(u_{n})-\I_{\e}(u)+o_{n}(1) \nonumber\\
&=c-\I_{\e}(u)+o_{n}(1)=:d+o_{n}(1).
\end{align}
Now, we prove that $\I'_{\e}(v_{n})=o_{n}(1)$. By using Lemma \ref{lemVince} with $z_{n}= v_{n}$ and $w=u$ we get
\begin{equation}\label{D}
\iint_{\R^{2N}} |\A(u_{n}) - \A(v_{n}) - \A(u)|^{p'} dx= o_{n}(1).
\end{equation}
Arguing as in the proof of Lemma $3.3$ in \cite{MeW}, we can see that
\begin{equation}\label{C}
\int_{\R^{N}} V(\e x) ||v_{n}|^{p-2}v_{n}-|u_{n}|^{p-2}u_{n}+|u|^{p-2}u|^{p'} dx=o_{n}(1).
\end{equation}
Hence, by using H\"older inequality, for any $\varphi\in \h$ such that $\|\varphi\|_{\e}\leq 1$, it holds
\begin{align*}
&|\langle \I'_{\e}(v_{n})-\I'_{\e}(u_{n})+\I'_{\e}(u), \varphi\rangle| \\
&\leq \left(\iint_{\R^{2N}}  |\A(u_{n}) - \A(v_{n}) - \A(u)|^{p'} dx dy\right)^{\frac{1}{p'}} [\varphi]_{W^{s, p}(\R^{N})} \\
&+\left(\int_{\R^{N}} V(\e x) ||v_{n}|^{p-2}v_{n}-|u_{n}|^{p-2}u_{n}+|u|^{p-2}u|^{p'} dx\right)^{p'} \left(\int_{\R^{N}} V(\e x) |\varphi|^{p}dx \right)^{\frac{1}{p}} \\
&+\int_{\R^{N}} |(f(v_{n})-f(u_{n})+f(u)) \varphi| dx,
\end{align*}
and in view of \eqref{B} of Lemma \ref{splitting}, \eqref{D}, \eqref{C}, $\I'_{\e}(u_{n})=0$ and $\I'_{\e}(u)=0$ we obtain the thesis.

Now, we note that by using $(f_4)$ we can see that
\begin{equation}\label{tv199}
\I_{\e}(u)=\I_{\e}(u)-\frac{1}{p} \langle\I'_{\e}(u),u\rangle=\int_{\R^{N}} \left[\frac{1}{p} f(u)u-F(u)\right]dx\geq 0.
\end{equation}
Assume $V_{\infty}<\infty$. From \eqref{tv19} and \eqref{tv199} it follows that
$$
d\leq c<c_{V_{\infty}}
$$
which together Lemma \ref{lem2.3} gives $v_{n}\rightarrow 0$ in $\h$, that is $u_{n}\rightarrow u$ in $\h$.\\
Let us consider the case $V_{\infty}=\infty$. Then, we can use Lemma \ref{Cheng} to deduce that $v_{n}\rightarrow 0$ in $L^{r}(\R^{N})$ for all $r\in [p, \p)$. This combined with assumptions $(f_2)$ and $(f_3)$ implies that
\begin{equation}\label{fn0}
\int_{\R^{N}} f(v_{n})v_{n}dx=o_{n}(1).
\end{equation}
Since $\langle\I'_{\e}(v_{n}), v_{n}\rangle=o_{n}(1)$ and applying \eqref{fn0} we can infer that
$$
\|v_{n}\|^{p}_{\e}=o_{n}(1),
$$
which yields $u_{n}\rightarrow u$ in $\h$.
\end{proof}

\noindent
We end this section giving the proof of the existence of a positive solution to ($P_{\e}$) whenever $\e>0$ is small enough.
\begin{theorem}\label{thm3.1}
Assume that $(V)$ and $(f_1)$-$(f_5)$ hold. Then there exists $\e_{0}>0$ such that problem ($P_{\e}$) admits a ground state solution for any $\e\in (0, \e_{0})$.
\end{theorem}
\begin{proof}
From $(v)$ of Lemma \ref{SW2}, we know that $c_{\e}\geq \rho>0$ for each $\e>0$. Moreover, if $u_{\e}\in \N_{\e}$ verifies $\I_{\e}(u)=c_{\e}$, then $m_{\e}^{-1}(u)$ is a minimizer of $\Psi_{\e}$ and it is a critical point of $\Psi_{\e}$. In view of Lemma \ref{SW3} we can see that $u$ is a critical point of $\I_{\e}$.

Now we show that there exists a minimizer of $\I_{\e}|_{\N_{\e}}$. By applying Ekeland's variational principle \cite{Ek} there exists a sequence $\{v_{n}\}\subset \mathbb{S}_{\e}$ such that $\Psi_{\e}(v_{n})\rightarrow c_{\e}$ and $\Psi'_{\e}(v_{n})\rightarrow 0$ as $n\rightarrow \infty$. Let $u_{n}=m_{\e}(v_{n}) \in \N_{\e}$. Then, from Lemma \ref{SW3} we deduce that $\I_{\e}(u_{n})\rightarrow c_{\e}$, $\langle \I'_{\e}(u_{n}), u_{n}\rangle =0$ and $\I'_{\e}(u_{n})\rightarrow 0$ as $n\rightarrow \infty$.
Therefore, $\{u_{n}\}$ is a Palais-Smale sequence for $\I_{\e}$ at level $c_{\e}$.
It is standard to check that $\{u_{n}\}$ is bounded in $\h$ and we denote by $u$ its weak limit. It is easy to verify that $\I_{\e}'(u)=0$.

Let us consider $V_{\infty}=\infty$. By using Lemma \ref{embedding} we have $\I_{\e}(u)=c_{\e}$ and $\I'_{\e}(u)=0$.\\
Now, we deal with the case $V_{\infty}<\infty$. In view of Proposition \ref{prop2.1} it is enough to show that $c_{\e}<c_{V_{\infty}}$ for small $\e$. Without loss of generality, we may suppose that
$$
V(0)=V_{0}=\inf_{x\in \R^{N}} V(x).
$$
Let $\mu\in \R$ such that $\mu\in (V_{0}, V_{\infty})$. Clearly $c_{V_{0}}<c_{\mu}<c_{V_{\infty}}$. By Lemma \ref{lem4.3}, it follows that there exists a positive ground state $w\in W^{s, p}(\R^{N})$ to the autonomous problem $(P_{\mu})$.\\
Let $\eta_{r}\in C^{\infty}_{c}(\R^{N})$ be a cut-off function such that $\eta_{r}=1$ in $\B_{r}(0)$ and $\eta_{r}=0$ in $\B_{2r}^{c}(0)$. Let us define $w_{r}(x):=\eta_{r}(x) w(x)$, and take $t_{r}>0$ such that
\begin{equation*}
\J_{\mu}(t_{r} w_{r})=\max_{t\geq 0} \J_{\mu}(t w_{r}).
\end{equation*}
Now we prove that there exists $r$ sufficiently large for which $\J_{\mu}(t_{r} w_{r})<c_{V_{\infty}}$. \\
Assume by contradiction $\J_{\mu}(t_{r} w_{r})\geq c_{V_{\infty}}$ for any $r>0$.
Taking into account $w_{r}\rightarrow w$ in $W^{s,p}(\R^{N})$ as $r\rightarrow \infty$ in view of Lemma \ref{Psi}, $t_{r}w_{r}$ and $w$ belong to $\N_{\mu}$ and by using assumption $(f_{5})$, we have that $t_{r}\rightarrow 1$. Therefore,
$$
c_{V_{\infty}}\leq \liminf_{r\rightarrow \infty} \J_{\mu}(t_{r} w_{r})=\J_{\mu}(w)=c_{\mu},
$$
which leads to a contradiction being $c_{V_{\infty}}>c_{\mu}$.
Hence, there exists $r>0$ such that
\begin{align}\label{tv200}
\J_{\mu}(t_{r} w_{r})=\max_{\tau\geq 0} \J_{\mu}(\tau (t_{r} w_{r}))\quad  \mbox{ and } \quad \J_{\mu}(t_{r}w_{r})<c_{V_{\infty}}.
\end{align}
Now, condition $(V)$ implies that there exists  $\e_{0}>0$ such that
\begin{equation}\label{tv20}
V(\e x)\leq \mu \mbox{ for all } x\in \supp(w_{r}), \quad \e\in (0, \e_{0}).
\end{equation}
Therefore, by using \eqref{tv200} and \eqref{tv20}, we deduce that for all $\e\in (0, \e_{0})$
$$
c_{\e}\leq \max_{\tau\geq 0} \I_{\e}(\tau (t_{r} w_{r}))\leq \max_{\tau\geq 0} \J_{\mu}(\tau (t_{r} w_{r}))=\J_{\mu}(t_{r} w_{r})<c_{V_{\infty}}
$$
which implies that $c_{\e}<c_{V_{\infty}}$ for any $\e>0$ sufficiently small.
\end{proof}

\subsection{Multiple solutions for \eqref{P}}\label{Sect3.3}
This section is devoted to the study of the multiplicity of solutions to \eqref{P}. We begin proving the following result which will be needed to implement the barycenter machinery.
\begin{proposition}\label{prop4.1}
Let $\e_{n}\rightarrow 0^{+}$ and $\{u_{n}\}\subset \N_{\e_{n}}$ be such that $\I_{\e_{n}}(u_{n})\rightarrow c_{V_{0}}$. Then there exists $\{\tilde{y}_{n}\}\subset \R^{N}$ such that the translated sequence
\begin{equation*}
v_{n}(x):=u_{n}(x+ \tilde{y}_{n})
\end{equation*}
has a subsequence which converges in $W^{s, p}(\R^{N})$. Moreover, up to a subsequence, $\{y_{n}\}:=\{\e_{n}\tilde{y}_{n}\}$ is such that $y_{n}\rightarrow y\in M$.
\end{proposition}

\begin{proof}
Since $\langle \I'_{\e_{n}}(u_{n}), u_{n} \rangle=0$ and $\I_{\e_{n}}(u_{n})\rightarrow c_{V_{0}}$, we know that $\{u_{n}\}$ is bounded in $\h$. From $c_{V_{0}}>0$, we can infer that $\|u_{n}\|_{\e_{n}}\not \rightarrow 0$.
Therefore, as in the proof of Lemma \ref{lem2.2}, we can find a sequence $\{\tilde{y}_{n}\}\subset \R^{N}$ and constants $R, \beta>0$ such that
\begin{equation}\label{tv21}
\liminf_{n\rightarrow \infty}\int_{\B_{R}(\tilde{y}_{n})} |u_{n}|^{p} dx\geq \beta.
\end{equation}
Let us define
\begin{equation*}
v_{n}(x):=u_{n}(x+ \tilde{y}_{n}).
\end{equation*}
In view of the boundedness of $\{u_{n}\}$ and \eqref{tv21} we may assume that $v_{n}\rightharpoonup v$ in $W^{s,p}(\R^{N})$ for some $v\neq 0$.
Let $\{t_{n}\}\subset (0, +\infty)$ be such that $w_{n}=t_{n} v_{n}\in \N_{V_{0}}$, and we set $y_{n}:=\e_{n}\tilde{y}_{n}$.  \\
Thus, by using the change of variables $z\mapsto x+ \tilde{y}_{n}$, $V(x)\geq V_{0}$ and the invariance by translation, we can see that
\begin{align*}
c_{V_{0}}\leq \J_{V_{0}}(w_{n})\leq \I_{\e_{n}}(t_{n} v_{n})\leq \I_{\e_{n}}(u_{n})=c_{V_{0}}+ o_{n}(1).
\end{align*}
Then we can infer $\J_{V_{0}}(w_{n})\rightarrow c_{V_{0}}$. This fact and $\{w_{n}\}\subset \N_{V_{0}}$  imply that there exists $K>0$ such that $\|w_{n}\|_{V_{0}}\leq K$ for all $n\in \mathbb{N}$.
Moreover, we can prove that the sequence $\{t_{n}\}$ is bounded. In fact, $v_{n}\not \rightarrow 0$ in $W^{s, p}(\R^{N})$, so there exists $\alpha>0$ such that $\|v_{n}\|_{V_{0}}\geq \alpha$. Consequently, for all $n\in \mathbb{N}$, we have
$$
|t_{n}|\alpha\leq \|t_{n}v_{n}\|_{V_{0}}=\|w_{n}\|_{V_{0}}\leq K,
$$
which yields $|t_{n}|\leq \frac{K}{\alpha}$ for all $n\in \mathbb{N}$.
Therefore, up to a subsequence, we may suppose that $t_{n}\rightarrow t_{0}\geq 0$. Let us show that $t_{0}>0$. Otherwise, if $t_{0}=0$, from the boundedness of $\{v_{n}\}$, we get $w_{n}= t_{n}v_{n} \rightarrow 0$ in $W^{s,p}(\R^{N})$, that is $\J_{V_{0}}(w_{n})\rightarrow 0$ in contrast with the fact $c_{V_{0}}>0$. Thus $t_{0}>0$, and up to a subsequence, we may assume that $w_{n}\rightharpoonup w:= t_{0} v\neq 0$ in $W^{s,p}(\R^{N})$. \\
Hence, it holds
\begin{equation*}
\J_{V_{0}}(w_{n})\rightarrow c_{V_{0}} \quad \mbox{ and } \quad w_{n}\rightharpoonup w\neq 0 \mbox{ in } W^{s,p}(\R^{N}).
\end{equation*}
From Lemma \ref{lem4.3}, we deduce that $w_{n} \rightarrow w$ in $W^{s,p}(\R^{N})$, that is $v_{n}\rightarrow v$ in $W^{s,p}(\R^{N})$. \\
Now, we show that $\{y_{n}\}$ has a subsequence such that $y_{n}\rightarrow y\in M$.
Assume by contradiction that $\{y_{n}\}$ is not bounded, that is there exists a subsequence, still denoted by $\{y_{n}\}$, such that $|y_{n}|\rightarrow +\infty$. \\
Firstly, we deal with the case $V_{\infty}=\infty$. 
By using $\{u_{n}\}\subset \N_{\e_{n}}$ and a change of variable, we can see that
\begin{align*}
\int_{\R^{N}} V(\e_{n}x+y_{n})|v_{n}|^{p} dx &\leq [v_{n}]^{p}_{W^{s, p}(\R^{N})}+\int_{\R^{N}} V(\e_{n}x+y_{n})|v_{n}|^{p} dx\\
&= \|u_{n}\|^{p}_{\e_{n}}=\int_{\R^{N}} f(u_{n})u_{n} \, dx = \int_{\R^{N}} f(v_{n})v_{n} \, dx.
\end{align*}
By applying Fatou's Lemma and $v_{n}\rightarrow v$ in $W^{s, p}(\R^{N})$, we deduce that
\begin{align*}
\infty=\liminf_{n\rightarrow \infty} \int_{\R^{N}} \! V(\e_{n}x+y_{n})|v_{n}|^{p} dx \leq \liminf_{n\rightarrow \infty} \int_{\R^{N}} \! f(v_{n})v_{n} dx =\int_{\R^{N}} \! f(v) v \,dx <\infty,
\end{align*}
which gives a contradiction. \\
Let us consider the case $V_{\infty}<\infty$. 
Taking into account $w_{n}\rightarrow w$ strongly in $W^{s,p}(\R^{N})$, condition $(V)$ and using the change of variable $z=x+ \tilde{y}_{n}$, we have
\begin{align}\label{tv22}
c_{V_{0}}&= \J_{V_{0}}(w) < \J_{V_{\infty}} (w) \nonumber\\
&\leq \liminf_{n\rightarrow \infty} \left [ \frac{1}{p} [w_{n}]^{2}_{W^{s, p}(\R^{N})}+\frac{1}{p} \int_{\R^{N}} V(\e_{n} x+y_{n})|w_{n}|^{p} dx-\int_{\R^{N}} F(w_{n}) \,dx \right]\nonumber \\
&=\liminf_{n\rightarrow \infty} \left[\frac{t^{p}_{n}}{p} [u_{n}]^{p}_{W^{s, p}(\R^{N})}+\frac{t^{p}_{n}}{p} \int_{\R^{N}} V(\e_{n} z)|u_{n}|^{p} dz-\int_{\R^{N}} F(t_{n} u_{n}) \, dz\right] \nonumber \\
&= \liminf_{n\rightarrow \infty} \I_{\e_{n}}(t_{n}u_{n}) \leq \liminf_{n\rightarrow \infty} \I_{\e_{n}} (u_{n})=c_{V_{0}}
\end{align}
which is an absurd.
Thus $\{y_{n}\}$ is bounded and, up to a subsequence, we may assume that $y_{n}\rightarrow y$. If $y\notin M$, then $V_{0}<V(y)$ and we can argue as in \eqref{tv22} to get a contradiction. Therefore, we can conclude that $y\in M$.
\end{proof}

\noindent
Let $\delta>0$ be fixed and let $\psi\in C^{\infty}(\R^{+}, [0, 1])$ be a function such that $\psi=1$ in $[0, \frac{\delta}{2}]$ and $\psi=0$ in $[\delta, \infty)$. For any $y\in M$, we define
$$
\Upsilon_{\e, y}(x)=\psi(|\e x-y|) \omega\left(\frac{\e x-y}{\e}\right),
$$
where $\omega\in H^{s}(\R^{N})$ is a positive ground state solution to $(P_{V_{0}})$ by Lemma \ref{lem4.3}.

Let $t_{\e}>0$ be the unique positive number such that
$$
\I_{\e}(t_{\e}\Upsilon_{\e, y})=\max_{t\geq 0} \, \I_{\e}(t \Upsilon_{\e, y})
$$
and let us define the map $\Phi_{\e}:M\rightarrow \N_{\e}$ by setting $\Phi_{\e}(y):=t_{\e} \Upsilon_{\e, y}$.
Then we can prove that
\begin{lemma}\label{lem4.1}
The functional $\Phi_{\e}$ satisfies the following limit
\begin{equation}\label{3.2}
\lim_{\e\rightarrow 0} \I_{\e}(\Phi_{\e}(y))=c_{V_{0}} \mbox{ uniformly in } y\in M.
\end{equation}
\end{lemma}
\begin{proof}
Assume by contradiction that there exist $\delta_{0}>0$, $\{y_{n}\}\subset M$ and $\e_{n}\rightarrow 0$ such that
\begin{equation}\label{4.41}
|\I_{\e_{n}}(\Phi_{\e_{n}} (y_{n}))-c_{V_{0}}|\geq \delta_{0}.
\end{equation}
Let us recall that Lemma \ref{Psi} implies that
\begin{align}\label{4.42}
\lim_{n\rightarrow \infty} \|\Upsilon_{\e_{n}, y_{n}}\|^{p}_{\e_{n}}=\|\omega\|_{V_{0}}^{p}.
\end{align}
Since $\langle \I'_{\e_{n}}(t_{\e_{n}}\Upsilon_{\e_{n}, y_{n}}), t_{\e_{n}} \Upsilon_{\e_{n}, y_{n}}\rangle=0$,
we can use the change of variable $z=\frac{\e_{n}x-y_{n}}{\e_{n}}$ to see that
\begin{align}\label{4.411}
\|t_{\e_{n}}\Upsilon_{\e_{n}, y_{n}}\|^{p}_{\e_{n}}&=\int_{\R^{N}} f(t_{\e_{n}}\Upsilon_{\e_{n}}) t_{\e_{n}}\Upsilon_{\e_{n}} dx \nonumber\\
&=\int_{\R^{N}} f(t_{\e_{n}} \psi(|\e_{n}z|) \omega(z)) t_{\e_{n}} \psi(|\e_{n}z|) \omega(z) \, dz.
\end{align}
Now, let us prove that $t_{\e_{n}}\rightarrow 1$. Firstly we show that $t_{\e_{n}}\rightarrow t_{0}<\infty$. Assume by contradiction that $|t_{\e_{n}}|\rightarrow \infty$. Then, by using the fact that $\psi\equiv 1$ in $\B_{\frac{\delta}{2}}(0)$ and $\B_{\frac{\delta}{2}}(0)\subset \B_{\frac{\delta}{2\e_{n}}}(0)$ for $n$ sufficiently large, we can see that \eqref{4.411} and $(f_5)$ give
\begin{align}\label{4.44}
\|\Upsilon_{\e_{n}, y_{n}}\|^{p}_{\e_{n}}\geq \int_{\B_{\frac{\delta}{2}}(0)} \frac{f(t_{\e_{n}} \omega(z))}{(t_{\e_{n}} \omega(z))^{p-1}} |\omega(z)|^{p} dz \geq \frac{f(t_{\e_{n}} \omega(\bar{z}))}{(t_{\e_{n}} \omega(\bar{z}))^{p-1}} \int_{\B_{\frac{\delta}{2}}(0)} |\omega(z)|^{p}dz,
\end{align}
where $\bar{z}$ is such that $\omega(\bar{z})=\min\{\omega(z): |z|\leq \frac{\delta}{2}\}>0$.
By using $(f_4)$ and $t_{\e_{n}}\rightarrow \infty$, we can see that \eqref{4.44} implies that $\|\Upsilon_{\e_{n}, y_{n}}\|^{p}_{\e_{n}}\rightarrow \infty$, which is an absurd in view of \eqref{4.42}.
Therefore, up to a subsequence, we may assume that $t_{\e_{n}}\rightarrow t_{0}\geq 0$.
If $t_{0}=0$, \eqref{4.42}, \eqref{4.411} and $(f_2)$ lead to a contradiction. Hence, $t_{0}>0$.
Now, we show that $t_{0}=1$.
Taking the limit as $n\rightarrow \infty$ in \eqref{4.411}, we can see that
$$
\|\omega\|_{V_{0}}^{p}=\int_{\R^{N}} \frac{f(t_{0}\omega)}{t_{0}^{p-1}} \omega.
$$
Recalling that $\omega\in  \N_{V_{0}}$ and using $(f_{5})$, we can deduce that $t_{0}=1$.
This fact and the Dominated Convergence Theorem yield
\begin{equation}\label{4.45}
\lim_{n\rightarrow \infty}\int_{\R^{N}} F(t_{\e_{n}} \Upsilon_{\e_{n}, y_{n}})=\int_{\R^{N}} F(\omega).
\end{equation}
Hence, passing to the limit as $n\rightarrow \infty$ in
\begin{equation*}
\I_{\e}(\Phi_{\e_{n}}(y_{n}))=\frac{t_{\e_{n}}^{p}}{p} \|\Upsilon_{\e_{n}, y_{n}}\|^{p}_{\e_{n}}-\int_{\R^{N}} F(t_{\e_{n}} \Upsilon_{\e_{n}, y_{n}}) dx,
\end{equation*}
and exploiting \eqref{4.42} and \eqref{4.45}, we can infer that
$$
\lim_{n\rightarrow \infty} \I_{\e_{n}}(\Phi_{\e_{n}}(y_{n}))=\J_{V_{0}}(\omega)=c_{V_{0}}
$$
which is impossible in view of \eqref{4.41}.
\end{proof}

\noindent
Now, we are in the position to introduce the barycenter map. We take $\rho>0$ such that $M_{\delta}\subset \B_{\rho}(0)$, and we set $\chi: \R^{N}\rightarrow \R^{N}$ as
 \begin{equation*}
\chi(x)=
 \left\{
 \begin{array}{ll}
 x &\mbox{ if } |x|<\rho \\
 \frac{\rho x}{|x|} &\mbox{ if } |x|\geq \rho.
  \end{array}
 \right.
 \end{equation*}
We define the barycenter map $\beta_{\e}: \N_{\e}\rightarrow \R^{N}$ as follows
\begin{align*}
\beta_{\e}(u)=\frac{\int_{\R^{N}} \chi(\e x) |u(x)|^{p} dx}{\int_{\R^{N}} |u(x)|^{p} dx}.
\end{align*}

\begin{lemma}\label{lem4.2}
The functional $\Phi_{\e}$ verifies the following limit
\begin{equation}\label{3.3}
\lim_{\e \rightarrow 0} \, \beta_{\e}(\Phi_{\e}(y))=y \mbox{ uniformly in } y\in M.
\end{equation}
\end{lemma}
\begin{proof}
Suppose by contradiction that there exist $\delta_{0}>0$, $\{y_{n}\}\subset M$ and $\e_{n}\rightarrow 0$ such that
\begin{equation}\label{4.4}
|\beta_{\e_{n}}(\Phi_{\e_{n}}(y_{n}))-y_{n}|\geq \delta_{0}.
\end{equation}
By using the definitions of $\Phi_{\e_{n}}(y_{n})$, $\beta_{\e_{n}}$, $\psi$ and the change of variable $z= \frac{\e_{n} x-y_{n}}{\e_{n}}$, we can see that
$$
\beta_{\e_{n}}(\Phi_{\e_{n}}(y_{n}))=y_{n}+\frac{\int_{\R^{N}}[\chi(\e_{n}z+y_{n})-y_{n}] |\psi(|\e_{n}z|) \omega(z)|^{p} \, dz}{\int_{\R^{N}} |\psi(|\e_{n}z|) \omega(z)|^{p}\, dz}.
$$
Taking into account $\{y_{n}\}\subset M\subset \B_{\rho}(0)$ and the Dominated Convergence Theorem, we can infer that
$$
|\beta_{\e_{n}}(\Phi_{\e_{n}}(y_{n}))-y_{n}|=o_{n}(1)
$$
which contradicts (\ref{4.4}).
\end{proof}

\noindent
At this point, we introduce a subset $\widetilde{\N}_{\e}$ of $\N_{\e}$ by taking a function $h:\R_{+}\rightarrow \R_{+}$ such that $h(\e)\rightarrow 0$ as $\e \rightarrow 0$, and setting
$$
\widetilde{\N}_{\e}=\{u\in \N_{\e}: \I_{\e}(u)\leq c_{V_{0}}+h(\e)\}.
$$
Fixed $y\in M$, from Lemma \ref{lem4.1} we deduce that $h(\e)=|\I_{\e}(\Phi_{\e}(y))-c_{V_{0}}|\rightarrow 0$ as $\e \rightarrow 0$. Hence $\Phi_{\e}(y)\in \widetilde{\N}_{\e}$, and $\widetilde{\N}_{\e}\neq \emptyset$ for any $\e>0$. Moreover, we have the following lemma.

\begin{lemma}\label{lem4.4}
For any $\delta>0$, there holds that
$$
\lim_{\e \rightarrow 0} \sup_{u\in \widetilde{\mathcal{N}}_{\e}} dist(\beta_{\e}(u), M_{\delta})=0.
$$
\end{lemma}

\begin{proof}
Let $\e_{n}\rightarrow 0$ as $n\rightarrow \infty$. By definition, there exists $\{u_{n}\}\subset \widetilde{\N}_{\e_{n}}$ such that
$$
\sup_{u\in \widetilde{\N}_{\e_{n}}} \inf_{y\in M_{\delta}}|\beta_{\e_{n}}(u)-y|=\inf_{y\in M_{\delta}}|\beta_{\e_{n}}(u_{n})-y|+o_{n}(1).
$$
Therefore, it suffices to prove that there exists $\{y_{n}\}\subset M_{\delta}$ such that
\begin{equation}\label{3.13}
\lim_{n\rightarrow \infty} |\beta_{\e_{n}}(u_{n})-y_{n}|=0.
\end{equation}
Thus, recalling that $\{u_{n}\}\subset  \widetilde{\N}_{\e_{n}}\subset  \N_{\e_{n}}$, we deduce that
$$
c_{V_{0}}\leq c_{\e_{n}}\leq  \I_{\e_{n}}(u_{n})\leq c_{V_{0}}+h(\e_{n}),
$$
which implies that $\I_{\e_{n}}(u_{n})\rightarrow c_{V_{0}}$. By using Proposition \ref{prop4.1}, there exists $\{\tilde{y}_{n}\}\subset \R^{N}$ such that $y_{n}=\e_{n}\tilde{y}_{n}\in M_{\delta}$ for $n$ sufficiently large. Thus
$$
\beta_{\e_{n}}(u_{n})=y_{n}+\frac{\int_{\R^{N}}[\chi(\e_{n}z+y_{n})-y_{n}] |u_{n}(z+\tilde{y}_{n})|^{p} \, dz}{\int_{\R^{N}} |u_{n}(z+\tilde{y}_{n})|^{p} \, dz}.
$$
Since $u_{n}(\cdot+\tilde{y}_{n})$ converges strongly in $W^{s,p}(\R^{N})$ and $\e_{n}z+y_{n}\rightarrow y\in M$, we can infer that $\beta_{\e_{n}}(u_{n})=y_{n}+o_{n}(1)$, that is (\ref{3.13}) holds.
\end{proof}

\noindent
Now we show that ($P_{\e}$) admits at least $cat_{M_{\delta}}(M)$ positive solutions.
In order to achieve our aim, we recall the following result for critical points involving Ljusternik-Schnirelmann category. For more details one can see \cite{MW}.
\begin{theorem}\label{LSt}
Let $U$ be a $C^{1,1}$ complete Riemannian manifold (modelled on a Hilbert space). Assume that $h\in C^{1}(U, \R)$ bounded from below and satisfies $-\infty<\inf_{U} h<d<k<\infty$. Moreover, suppose that $h$ satisfies Palais-Smale condition on the sublevel $\{u\in U: h(u)\leq k\}$ and that $d$ is not a critical level for $h$. Then
$$
card\{u\in h^{d}: \nabla h(u)=0\}\geq cat_{h^{d}}(h^{d}).
$$
\end{theorem}
\noindent
Since $\mathcal{N}_{\e}$ is not a $C^{1}$ submanifold of $\h$, we can not directly apply Theorem \ref{LSt}. Fortunately, from Lemma \ref{SW2}, we know that the mapping $m_{\e}$ is a homeomorphism between $\mathcal{N}_{\e}$ and $\mathbb{S}_{\e}$, and $\mathbb{S}_{\e}$ is a $C^{1}$ submanifold of $\h$. So we can apply Theorem \ref{LSt} to
$\Psi_{\e}(u)=\I_{\e}(\hat{m}_{\e}(u))|_{\mathbb{S}_{\e}}=\I_{\e}(m_{\e}(u))$, where $\Psi_{\e}$ is given in Lemma \ref{SW3}.
\begin{theorem}\label{teorema}
Assume that $(V)$ and $(f_1)$-$(f_5)$ hold. Then, for any $\delta>0$ there exists $\bar{\e}_\delta>0$ such that, for any $\e \in (0, \bar{\e}_\delta)$, problem ($P_{\e}$) has at least $cat_{M_{\delta}}(M)$ positive solutions.
\end{theorem}

\begin{proof}
For any $\e>0$, we define $\alpha_{\e} : M \rightarrow \mathbb{S}_{\e}$ by setting $\alpha_{\e}(y)= m_{\e}^{-1}(\Phi_{\e}(y))$. By using Lemma \ref{lem4.1} and the definition of $\Psi_{\e}$, we can see that
\begin{equation*}
\lim_{\e \rightarrow 0} \Psi_{\e}(\alpha_{\e}(y)) = \lim_{\e \rightarrow 0} \I_{\e}(\Phi_{\e}(y))= c_{V_{0}} \quad \mbox{ uniformly in } y\in M.
\end{equation*}
Then, there exists $\tilde{\e}>0$ such that $\tilde{\mathbb{S}}_{\e}:=\{ w\in \mathbb{S}_{\e} : \Psi_{\e}(w) \leq c_{V_{0}} + h(\e)\} \neq \emptyset$ for all $\e \in (0, \tilde{\e})$, where $h(\e)= |\Psi_{\e}(\alpha_{\e}(y)) - c_{V_{0}}|\rightarrow 0$ as $\e\rightarrow 0$.

Taking into account Lemma \ref{lem4.1}, Lemma \ref{SW2}, Lemma \ref{lem4.2} and Lemma \ref{lem4.4}, we can find $\bar{\e}= \bar{\e}_{\delta}>0$ such that the following diagram
\begin{equation*}
M\stackrel{\Phi_{\e}}{\rightarrow} \widetilde{\mathcal{N}}_{\e} \stackrel{m_{\e}^{-1}}{\rightarrow} \tilde{\mathbb{S}}_{\e} \stackrel{m_{\e}}{\rightarrow} \widetilde{\mathcal{N}}_{\e} \stackrel{\beta_{\e}}{\rightarrow} M_{\delta}
\end{equation*}
is well defined for any $\e \in (0, \bar{\e})$.

By using Lemma \ref{lem4.2}, there exists a function $\theta(\e, y)$ with $|\theta(\e, y)|<\frac{\delta}{2}$ uniformly in $y\in M$ for all $\e \in (0, \bar{\e})$ such that $\beta_{\e}(\Phi_{\e}(y))= y+ \theta(\e, y)$ for all $y\in M$. Then, we can see that $H(t, y)= y+ (1-t)\theta(\e, y)$ with $(t, y)\in [0,1]\times M$ is a homotopy between $\beta_{\e} \circ \Phi_{\e}=(\beta_{\e}\circ m_{\e}) \circ \alpha_{\e}$ and the inclusion map $id: M \rightarrow M_{\delta}$. This fact and Lemma $4.3$ in \cite{BC} implies that $cat_{\tilde{\mathbb{S}}_{\e}} (\tilde{\mathbb{S}}_{\e})\geq cat_{M_{\delta}}(M)$.

On the other hand, let us choose a function $h(\e)>0$ such that $h(\e)\rightarrow 0$ as $\e\rightarrow 0$ and such that $c_{V_{0}}+h(\e)$ is not a critical level for $\I_{\e}$. For $\e>0$ small enough, we deduce from Proposition \ref{prop2.1} that $\I_{\e}$ satisfies the Palais-Smale condition in $\widetilde{\N}_{\e}$. So, by $(ii)$ of Lemma \ref{SW3} we infer that $\Psi_{\e}$ satisfies the Palais-Smale condition in $\tilde{\mathbb{S}}_{\e}$. Hence, by using Theorem \ref{LSt} we obtain that $\Psi_{\e}$ has at least $cat_{\tilde{\mathbb{S}}_{\e}}(\tilde{\mathbb{S}}_{\e})$ critical points on $\tilde{\mathbb{S}}_{\e}$. Then, in view of $(iii)$ of Lemma \ref{SW3} we can infer that $\I_{\e}$ admits at least $cat_{M_{\delta}}(M)$ critical points.
\end{proof}

\subsection{Concentration of solutions to \eqref{P}}
Let us prove the following result which will play a fundamental role to study the behavior of maximum points of solutions to \eqref{P}.
\begin{lemma}\label{lemMoser}
Let $v_{n}$ be a weak solution to
\begin{equation}\label{Pvn}
\left\{
\begin{array}{ll}
(-\Delta)_{p}^{s} v_{n} + V_{n}(x)|v_{n}|^{p-2}v_{n} = f(v_{n}) &\mbox{ in } \R^{N} \\
v_{n}\in W^{s, p}(\R^{N}) \\
v_{n}(x)>0 &\mbox{ for all } x\in \R^{N},
\end{array}
\right.
\end{equation}
where $V_{n}(x)=V(\e_{n}x+\e_{n}\tilde{y}_{n})$, $\{\tilde{y}_{n}\}$ is given in Proposition \ref{prop4.1} and $v_{n}\rightarrow v$ in $W^{s, p}(\R^{N})$ with $v\not\equiv 0$. Then $v_{n}\in L^{\infty}(\R^{N})$ and there exists $C>0$ such that $|v_{n}|_{L^{\infty}(\R^{N})}\leq C$ for all $n\in \mathbb{N}$. Moreover, $\lim_{|x|\rightarrow \infty} v_{n}(x)=0$ uniformly in $n\in \mathbb{N}$.
\end{lemma}
\begin{proof}
For any $L>0$ and $\beta>1$, let us consider the function
\begin{equation*}
\gamma(v_{n})=\gamma_{L, \beta}(v_{n})=v_{n} v_{L,n}^{p(\beta-1)}\in \h,
\end{equation*}
where  $v_{L,n}=\min\{v_{n}, L\}$.
Let us observe that, since $\gamma$ is an increasing function, it holds
\begin{align*}
(a-b)(\gamma(a)- \gamma(b))\geq 0 \quad \mbox{ for any } a, b\in \R.
\end{align*}
Define the functions
\begin{equation*}
\Lambda(t)=\frac{|t|^{p}}{p} \quad \mbox{ and } \quad \Gamma(t)=\int_{0}^{t} (\gamma'(\tau))^{\frac{1}{p}} d\tau.
\end{equation*}
Fix $a, b\in \R$ such that $a>b$. Then, from the above definitions and applying Jensen inequality we get
\begin{align*}
\Lambda'(a-b)(\gamma(a)-\gamma(b)) &=(a-b)^{p-1} (\gamma(a)-\gamma(b)) \\
&= (a-b)^{p-1} \int_{b}^{a} \gamma'(t) dt \\
&= (a-b)^{p-1} \int_{b}^{a} (\Gamma'(t))^{p} dt \\
&\geq \left(\int_{b}^{a} (\Gamma'(t)) dt\right)^{p}=( \Gamma(a)-\Gamma(b))^{p}.
\end{align*}
In similar fashion, we can prove that the above inequality is true for any $a\leq b$. Thus we can infer that
\begin{equation}\label{Gg}
\Lambda'(a-b)(\gamma(a)-\gamma(b))\geq |\Gamma(a)-\Gamma(b)|^{p} \quad \mbox{ for any } a, b\in\R.
\end{equation}
In particular, by \eqref{Gg} it follows that
\begin{align}\label{Gg1}
&|\Gamma(v_{n}(x))-\Gamma(v_{n}(y))|^{p} \nonumber \\
&\leq |v_{n}(x)- v_{n}(y)|^{p-2}(v_{n}(x)- v_{n}(y))((v_{n}v_{L,n}^{p(\beta-1)})(x)- (v_{n}v_{L,n}^{p(\beta-1)})(y)).
\end{align}
Therefore, taking $\gamma(v_{n})=v_{n} v_{L,n}^{p(\beta-1)}$ as test-function in \eqref{Pvn}, in view of \eqref{Gg1} we have
\begin{align}\label{BMS}
&[\Gamma(v_{n})]_{W^{s, p}(\R^{N})}^{p}+\int_{\R^{N}} V_{n}(x)|v_{n}|^{p}v_{L,n}^{p(\beta-1)} dx \nonumber \\
&\leq \iint_{\R^{2N}} \frac{|v_{n}(x)-v_{n}(y)|^{p-2}(v_{n}(x)- v_{n}(y))}{|x-y|^{N+sp}} ((v_{n}v_{L,n}^{p(\beta-1)})(x)-(v_{n} v_{L,n}^{p(\beta-1)})(y)) \,dx dy \nonumber \\
&\quad +\int_{\R^{N}} V_{n}(x)|v_{n}|^{p}v_{L,n}^{p(\beta-1)} dx \nonumber\\
&=\int_{\R^{N}} f(v_{n}) v_{n} v_{L,n}^{p(\beta-1)} dx.
\end{align}
Since $\Gamma(v_{n})\geq \frac{1}{\beta} v_{n} v_{L,n}^{\beta-1}$, from the Sobolev inequality we can deduce that
\begin{equation}\label{SS1}
[\Gamma(v_{n})]_{W^{s, p}(\R^{N})}^{p}\geq S_{*} |\Gamma(v_{n})|^{p}_{L^{\p}(\R^{N})}\geq \left(\frac{1}{\beta}\right)^{p} S_{*}|v_{n} v_{L,n}^{\beta-1}|^{p}_{L^{\p}(\R^{N})}.
\end{equation}
On the other hand, from assumptions $(f_2)$-$(f_3)$, we know that for any $\xi>0$ there exists $C_{\xi}>0$ such that
\begin{equation}\label{SS2}
f(t)\leq \xi |t|^{p-1}+C_{\xi}|t|^{\p-1}  \quad \mbox{ for all } t\in \R.
\end{equation}
Choosing $\xi\in (0, V_{0})$, and using \eqref{SS1} and \eqref{SS2}, we can see that \eqref{BMS} yields
\begin{align}\label{simo1}
|w_{L,n}|^{p}_{L^{\p}(\R^{N})}&\leq C\beta^{p} \int_{\R^{N}} |v_{n}|^{\p} v_{L,n}^{p(\beta-1)} dx,
\end{align}
where $w_{L,n}:=v_{n} v_{L,n}^{\beta-1}$.
Now, we take $\beta=\frac{\p}{p}$ and fix $R>0$. Observing that $0\leq v_{L,n}\leq v_{n}$, we can deduce that
\begin{align}\label{simo2}
&\int_{\R^{N}} v^{\p}_{n}v_{L,n}^{p(\beta-1)}dx\\
&=\int_{\R^{N}} v^{\p-p}_{n} v^{p}_{n} v_{L,n}^{\p-p}dx \nonumber\\
&=\int_{\R^{N}} v^{\p-p}_{n} (v_{n} v_{L,n}^{\frac{\p-p}{p}})^{p}dx \nonumber\\
&\leq \int_{\{v_{n}<R\}} R^{\p-p} v^{\p}_{n} dx+\int_{\{v_{n}>R\}} v^{\p-p}_{n} (v_{n} v_{L,n}^{\frac{\p-p}{p}})^{p}dx \nonumber\\
&\leq \int_{\{v_{n}<R\}} R^{\p-p} v^{\p}_{n} dx+\left(\int_{\{v_{n}>R\}} v^{\p}_{n} dx\right)^{\frac{\p-p}{\p}} \left(\int_{\R^{N}} (v_{n} v_{L,n}^{\frac{\p-p}{p}})^{\p}dx\right)^{\frac{p}{\p}}.
\end{align}
Since $v_{n}\rightarrow v$ in $W^{s, p}(\R^{N})$, we can see that for any $R$ sufficiently large
\begin{equation}\label{simo3}
\left(\int_{\{v_{n}>R\}} v^{\p}_{n} dx\right)^{\frac{\p-p}{\p}}\leq \e \beta^{-p}.
\end{equation}
Putting together \eqref{simo1}, \eqref{simo2} and \eqref{simo3} we get
\begin{equation*}
\left(\int_{\R^{N}} (v_{n} v_{L,n}^{\frac{\p-p}{p}})^{\p} \right)^{\frac{p}{\p}}\leq C\beta^{p} \int_{\R^{N}} R^{\p-p} v^{\p}_{n} dx<\infty
\end{equation*}
and taking the limit as $L\rightarrow \infty$, we obtain $v_{n}\in L^{\frac{(\p)^{2}}{p}}(\R^{N})$.

Now, using $0\leq v_{L,n}\leq v_{n}$ and by passing to the limit as $L\rightarrow \infty$ in \eqref{simo1}, we have
\begin{equation*}
|v_{n}|_{L^{\beta\p}(\R^{N})}^{\beta p}\leq C \beta^{p} \int_{\R^{N}} v^{\p+p(\beta-1)}_{n} \, dx,
\end{equation*}
from which we deduce that
\begin{equation*}
\left(\int_{\R^{N}} v^{\beta\p}_{n} dx\right)^{\frac{1}{(\beta-1)\p}}\leq C \beta^{\frac{1}{\beta-1}} \left(\int_{\R^{N}} v^{\p+p(\beta-1)}_{n}\, dx\right)^{\frac{1}{p(\beta-1)}}.
\end{equation*}
For $m\geq 1$ we define $\beta_{m+1}$ inductively so that $\p+p(\beta_{m+1}-1)=\p \beta_{m}$ and $\beta_{1}=\frac{\p}{p}$. Then 
\begin{equation*}
\left(\int_{\R^{N}} v_{n}^{\beta_{m+1}\p} dx\right)^{\frac{1}{(\beta_{m+1}-1)\p}}\leq C \beta_{m+1}^{\frac{1}{\beta_{m+1}-1}} \left(\int_{\R^{N}} v_{n}^{\p\beta_{m}}\, dx\right)^{\frac{1}{\p(\beta_{m}-1)}}.
\end{equation*}
Let us define
$$
D_{m}=\left(\int_{\R^{N}} v_{n}^{\p\beta_{m}}\, dx\right)^{\frac{1}{\p(\beta_{m}-1)}}.
$$
By using an iteration argument, we can find $C_{0}>0$ independent of $m$ such that
$$
D_{m+1}\leq \prod_{k=1}^{m} C \beta_{k+1}^{\frac{1}{\beta_{k+1}-1}}  D_{1}\leq C_{0} D_{1}.
$$
Taking the limit as $m\rightarrow \infty$ we get $|v_{n}|_{L^{\infty}(\R^{N})}\leq K$ for all $n\in \mathbb{N}$.
Moreover, by using Corollary $5.5$ in \cite{IMS}, we can deduce that $v_{n}\in \mathcal{C}^{0, \alpha}(\R^{N})$ for some $\alpha>0$ (independent of $n$) and $[v_{n}]_{\C^{0, \alpha}(\R^{N})}\leq C$, with $C$ independent of $n$. Since $v_{n}\rightarrow v$ in $W^{s, p}(\R^{N})$, we can infer that $\lim_{|x|\rightarrow\infty}v_{n}(x)=0$ uniformly in $n\in \mathbb{N}$.
\end{proof}

\begin{remark}\label{Rem3}
We can also provide a more precise estimate on the decay of $v_{n}$ at infinity. Indeed, by using $(f_2)$ and $\lim_{|x|\rightarrow\infty}v_{n}(x)=0$, we can see that there exists $R>0$ such that $f(v_{n})\leq \frac{V_{0}}{2}v_{n}^{p-1}$ for all $x\in \mathcal{B}^{c}_{R}(0)$. Therefore
\begin{equation}\label{BBMP1}
(-\Delta)^{s}_{p}v_{n}+\frac{V_{0}}{2}v_{n}^{p-1}=f(v_{n})-\left(V_{n}-\frac{V_{0}}{2}\right)v_{n}^{p-1}\leq 0 \quad \mbox{ in } \mathcal{B}^{c}_{R}(0).
\end{equation}
By using Theorem $A.4$ in \cite{BMS}, we know that $\Gamma(x)=|x|^{-\frac{N-sp}{p-1}}$ is a weak solution to
\begin{equation}\label{BBMP2}
(-\Delta)^{s}_{p}\Gamma+\frac{V_{0}}{2}\Gamma^{p-1}=\frac{V_{0}}{2}\Gamma^{p-1}\geq 0 \quad \mbox{ in }  \mathcal{B}^{c}_{r}(0),
\end{equation}
for all $r>0$. In view of the continuity of $v_{n}$ and $\Gamma$, there exists $C_{1}>0$ such that $w_{n}(x)=v_{n}(x)- C_{1}\Gamma(x)\leq 0$ for all $|x|=R$ (with $R$ larger if necessary). Taking $\phi=\max\{w_{n}, 0\}\in W^{s,p}_{0}(\mathcal{B}^{c}_{R}(0))$ as test function in \eqref{BBMP1} and using \eqref{BBMP2} with $\tilde{\Gamma}=C_{1}\Gamma$, we can deduce that
\begin{align}\label{BBMP3}
0&\geq \iint_{\R^{2N}} \frac{|v_{n}(x)-v_{n}(y)|^{p-2}(v_{n}(x)-v_{n}(y))(\phi(x)-\phi(y))}{|x-y|^{N+sp}}dx dy+\frac{V_{0}}{2}\int_{\R^{N}} v_{n}^{p-1}\phi \,dx \nonumber\\
&\geq  \iint_{\R^{2N}} \frac{\mathcal{G}_{n}(x,y)}{|x-y|^{N+sp}}(\phi(x)-\phi(y))\,dx dy+\frac{V_{0}}{2}\int_{\R^{N}} [v_{n}^{p-1}-\widetilde{\Gamma}^{p-1}]\phi\, dx,
\end{align}
where 
$$
\mathcal{G}_{n}(x,y):=|v_{n}(x)-v_{n}(y)|^{p-2}(v_{n}(x)-v_{n}(y))-|\widetilde{\Gamma}(x)-\widetilde{\Gamma}(y)|^{p-2}(\widetilde{\Gamma}(x)-\widetilde{\Gamma}(y)).
$$
Therefore, if we prove that
\begin{align}\label{BBMP0}
\iint_{\R^{2N}} \frac{\mathcal{G}_{n}(x,y)}{|x-y|^{N+sp}} (\phi(x)-\phi(y))dx dy\geq 0,
\end{align}
it follows from \eqref{BBMP3} that $0\geq \frac{V_{0}}{2}\int_{\{v_{n}\geq \widetilde{\Gamma}\}} [v_{n}^{p-1}-\widetilde{\Gamma}^{p-1}] (v_{n}-\widetilde{\Gamma}) dx\geq 0$, which yields that
$$
\{x\in \R^{N}: |x|\geq R \mbox{ and } v_{n}(x)\geq \widetilde{\Gamma}\}=\emptyset.
$$
To achieve our purpose, we first note that for all $a, b\in \R$ it holds
$$
|b|^{p-2}b-|a|^{p-2}a=(p-1)(b-a)\int_{0}^{1}|a+t(b-a)|^{p-2}dt.
$$
Taking $b=v_{n}(x)-v_{n}(y)$ and $a=\widetilde{\Gamma}(x)-\widetilde{\Gamma}(y)$ we can see that
$$
|b|^{p-2}b-|a|^{p-2}a=(p-1)(b-a) I(x,y),
$$
where $I(x,y)\geq 0$ stands for the integral. Since
\begin{align*}
(b-a)(\phi(x)-\phi(y))&=[(v_{n}-\widetilde{\Gamma})(x)-(v_{n}-\widetilde{\Gamma})(y)] [(v_{n}-\widetilde{\Gamma})^{+}(x)-(v_{n}-\widetilde{\Gamma})^{+}(y)] \\
&\geq |(v_{n}-\widetilde{\Gamma})^{+}(x)-(v_{n}-\widetilde{\Gamma})^{+}(y)|^{2},
\end{align*}
we can infer that $(|b|^{p-2}b-|a|^{p-2}a)(\phi(x)-\phi(y))\geq 0$, that is \eqref{BBMP0} holds true.
As a consequence, we can conclude that
$v_{n}(x)\leq C|x|^{-\frac{N-sp}{p-1}}$ for all $|x|$ large enough.
\end{remark}

\begin{lemma}\label{UBlemAF}
There exists $\delta>0$ such that $|v_{n}|_{L^{\infty}(\R^{N})}\geq \delta$ for all $n\in \mathbb{N}$.
\end{lemma}
\begin{proof}
Assume by contradiction that $|v_{n}|_{L^{\infty}(\R^{N})}\rightarrow 0$ as $n\rightarrow \infty$. By using $(f_2)$, there exists $n_{0}\in \mathbb{N}$ such that $\frac{f(|v_{n}|_{L^{\infty}(\R^{N})})}{|v_{n}|_{L^{\infty}(\R^{N})}^{p-1}}<\frac{V_{0}}{2}$ for all $n\geq n_{0}$.
Therefore, in view of $(f_{5})$ we can see that
$$
[v_{n}]_{W^{s,p}(\R^{N})}^{p}+V_{0}|v_{n}|^{p}_{L^{p}(\R^{N})}\leq \int_{\R^{N}}\frac{f(|v_{n}|_{L^{\infty}(\R^{N})})}{|v_{n}|_{L^{\infty}(\R^{N})}^{p-1}} |v_{n}|^{p} dx <\frac{V_{0}}{2} |v_{n}|^{p}_{L^{p}(\R^{N})},
$$
which is impossible.
\end{proof}

\noindent
Now, we end this section studying the behavior of maximum points of solutions to \eqref{P}. If $u_{\e_{n}}$ is a solution to ($P_{\e_{n}}$), then $v_{n}(x)=u_{\e_{n}}(x+\tilde{y}_{n})$ is a solution to \eqref{Pvn}. Moreover, up to subsequence, $v_{n}\rightarrow v$ in $W^{s,p}(\R^{N})$ and $y_{n}=\e_{n}\tilde{y}_{n}\rightarrow y\in M$ in view of Proposition \ref{prop4.1}.
If $p_{n}$ denotes a global maximum point of $v_{n}$, we can use Lemma \ref{lemMoser} and Lemma \ref{UBlemAF} to see that $p_{n}\in \B_{R}(0)$ for some $R>0$. As a consequence, the point of maximum of $u_{\e_{n}}$ is of the type $z_{\e_{n}}=p_{n}+\tilde{y}_{n}$ and then $\e_{n}z_{\e_{n}}=\e_{n}p_{n}+\e_{n}\tilde{y}_{n}\rightarrow y$ because $\{p_{n}\}$ is bounded. This fact and the continuity of $V$ yield $V(\e_{n}z_{\e_{n}})\rightarrow V(y)=V_{0}$ as $n\rightarrow \infty$.

\section{Critical case}\label{Sect4}

\subsection{Functional setting in the critical case}\label{Sect4.1}
\noindent
In this section we deal with critical problem \eqref{Pc}. Since many calculations are adaptations to that presented in the early sections, we will emphasize only the differences between the subcritical and the critical case.

By using a change of variable we consider the following problem
\begin{equation}\tag{$P_{\e}^{*}$}
\left\{
\begin{array}{ll}
(-\Delta)_{p}^{s} u +V(\e x) |u|^{p-2}u = f(u) + |u|^{\p-2} u &\mbox{ in } \R^{N}\\
u\in W^{s, p}(\R^{N}) \\
u(x)>0 &\mbox{ for all } x\in \R^{N}.
\end{array}
\right.
\end{equation}
The functional associated to ($P_{\e}^{*}$) is given by
\begin{equation*}
\I_{\e}(u)=\frac{1}{p} \|u\|^{p}_{\e} - \int_{\R^{N}} F(u) \, dx - \frac{1}{\p} |u|^{\p}_{L^{\p}(\R^{N})}
\end{equation*}
which is well defined on $\h$. Let us introduce the Nehari manifold associated to $\I_{\e}$, that is
\begin{equation*}
\N_{\e}= \left\{u\in \h\setminus \{0\} : \langle \I'_{\e}(u), u\rangle =0  \right\}.
\end{equation*}

\noindent
Arguing as in Section \ref{Sect3.1} we can prove that the following lemmas hold true.
\begin{lemma}\label{lem2.1c}
The functional $\I_{\e}$ satisfies the following conditions:
\begin{compactenum}[$(i)$]
\item there exist $\alpha, \rho >0$ such that $\I_{\e}(u)\geq \alpha$ with $\|u\|_{\e}=\rho$;
\item there exists $e\in \h$ with $\|e\|_{\e}>\rho$ such that $\I_{\e}(e)<0$.
\end{compactenum}
\end{lemma}

\begin{lemma}\label{SW1c}
Under assumptions $(V)$ and $(f_1)$-$(f_5)$ and $(f'_6)$ we have for any $\e>0$:
\begin{compactenum}[$(i)$]
\item $\I'_{\e}$ maps bounded sets of $\h$ into bounded sets of $\h$.
\item $\I'_{\e}$ is weakly sequentially continuous in $\h$.
\item $\I_{\e}(t_{n}u_{n})\rightarrow -\infty$ as $t_{n}\rightarrow \infty$, where $u_{n}\in K$ and $K\subset \h\setminus\{0\}$ is a compact subset.
\end{compactenum}
\end{lemma}

\begin{lemma}\label{SW2c}
Under the assumptions of Lemma \ref{SW1c}, for $\e>0$ we have:
\begin{compactenum}[$(i)$]
\item for all $u\in \mathbb{S}_{\e}$, there exists a unique $t_{u}>0$ such that $t_{u}u\in \N_{\e}$. Moreover, $m_{\e}(u)=t_{u}u$ is the unique maximum of $\I_{\e}$ on $\h$, where $\mathbb{S}_{\e}=\{u\in \h: \|u\|_{\e}=1\}$.
\item The set $\N_{\e}$ is bounded away from $0$. Furthermore $\N_{\e}$ is closed in $\h$.
\item There exists $\alpha>0$ such that $t_{u}\geq \alpha$ for each $u\in \mathbb{S}_{\e}$ and, for each compact subset $W\subset \mathbb{S}_{\e}$, there exists $C_{W}>0$ such that $t_{u}\leq C_{W}$ for all $u\in W$.
\item For each $u\in \N_{\e}$, $m_{\e}^{-1}(u)=\frac{u}{\|u\|_{\e}}\in \N_{\e}$. In particular, $\N_{\e}$ is a regular manifold diffeomorphic to the sphere in $\h$.
\item $c_{\e}=\inf_{\N_{\e}} \I_{\e}\geq \rho>0$ and $\I_{\e}$ is bounded below on $\N_{\e}$, where $\rho$ is independent of $\e$.
\end{compactenum}
\end{lemma}

\noindent
Now, we introduce the functionals $\hat{\Psi}_{\e}: \h\setminus\{0\} \rightarrow \R$ and $\Psi_{\e}: \mathbb{S}_{\e}\rightarrow \R$ defined by
\begin{align*}
\hat{\Psi}_{\e}= \I_{\e}(\hat{m}_{\e}(u)) \quad \mbox{ and } \quad \Psi_{\e}= \hat{\Psi}_{\e}|_{\mathbb{S}_{\e}}.
\end{align*}
Then we have the following result:
\begin{lemma}\label{SW3c}
Under the assumptions of Lemma \ref{SW1c}, we have that for $\e>0$:
\begin{compactenum}[$(i)$]
\item $\Psi_{\e}\in \C^{1}(\mathbb{S}_{\e}, \R)$, and
\begin{equation*}
\Psi'_{\e}(w)v= \|m_{\e}(w)\|_{\e} \I'_{\e}(m_{\e}(w))v \quad \mbox{ for } v\in T_{w}(\mathbb{S}_{\e}).
\end{equation*}
\item $\{w_{n}\}$ is a Palais-Smale sequence for $\Psi_{\e}$ if and only if $\{m_{\e}(w_{n})\}$ is a Palais-Smale sequence for $\I_{\e}$. If $\{u_{n}\}\subset \N_{\e}$ is a bounded Palais-Smale sequence for $\I_{\e}$, then $\{m_{\e}^{-1}(u_{n})\}$ is a Palais-Smale sequence for $\Psi_{\e}$.
\item $u\in \mathbb{S}_{\e}$ is a critical point of $\Psi_{\e}$ if and only if $m_{\e}(u)$ is a critical point of $\I_{\e}$. Moreover the corresponding critical values coincide and
\begin{equation*}
\inf_{\mathbb{S}_{\e}} \Psi_{\e}=\inf_{\N_{\e}} \I_{\e}=c_{\e}.
\end{equation*}
\end{compactenum}
\end{lemma}

\noindent
Finally, it is easy to prove that
\begin{lemma}\label{lemBc}
If $\{u_{n}\}$ is a Palais-Smale sequence of $\I_{\e}$ at level $c$, then $\{u_{n}\}$ is bounded in $\h$.
\end{lemma}

\subsection{Autonomous critical problem}\label{Sect4.2}
Let us consider the following autonomous critical problem
\begin{equation*}\tag{$P_{\mu}^{*}$}
\left\{
\begin{array}{ll}
(-\Delta)_{p}^{s} u + \mu |u|^{p-2}u = f(u) + |u|^{\p-2} u &\mbox{ in } \R^{N} \\
u\in W^{s, p}(\R^{N}) \\
u(x)>0 &\mbox{ for all } x\in \R^{N},
\end{array}
\right.
\end{equation*}
with $N\geq sp^{2}$.
The functional associated to the above problem is defined as
\begin{equation*}
\J_{\mu}(u)=\frac{1}{p} \|u\|_{\mu}^{p} - \int_{\R^{N}} F(u) \, dx - \frac{1}{\p}|u|^{\p}_{L^{\p}(\R^{N})},
\end{equation*}
and the Nehari manifold associated to $\J_{\mu}$ is given by
\begin{equation*}
\N_{\mu}= \left\{u\in \X_{\mu}\setminus \{0\} : \langle \J'_{\mu}(u), u\rangle =0  \right\}.
\end{equation*}
It is standard to check that $\J_{\mu}$ has a mountain pass geometry. Moreover we have the following useful results:
\begin{lemma}\label{SW2Ac}
Under the assumptions of Lemma \ref{SW1c}, for $\mu>0$ we have:
\begin{compactenum}[$(i)$]
\item for all $u\in \mathbb{S}_{\mu}$, there exists a unique $t_{u}>0$ such that $t_{u}u\in \N_{\mu}$. Moreover, $m_{\mu}(u)=t_{u}u$ is the unique maximum of $\J_{\mu}$ on $\h$, where $\mathbb{S}_{\mu}=\{u\in \X_{\mu}: \|u\|_{\mu}=1\}$.
\item The set $\N_{\mu}$ is bounded away from $0$. Furthermore $\N_{\mu}$ is closed in $\X_{\mu}$.
\item There exists $\alpha>0$ such that $t_{u}\geq \alpha$ for each $u\in \mathbb{S}_{\mu}$ and, for each compact subset $W\subset \mathbb{S}_{\mu}$, there exists $C_{W}>0$ such that $t_{u}\leq C_{W}$ for all $u\in W$.
\item $\N_{\mu}$ is a regular manifold diffeomorphic to the sphere in $\X_{\mu}$.
\item $c_{\mu}=\inf_{\N_{\mu}} \J_{\mu}>0$ and $\J_{\mu}$ is bounded below on $\N_{\mu}$ by some positive constant.
\item $\J_{\mu}$ is coercive on $\N_{\mu}$.
\end{compactenum}
\end{lemma}

\noindent
Let us define the functionals $\hat{\Psi}_{\mu}: \X_{\mu}\setminus\{0\} \rightarrow \R$ and $\Psi_{\mu}: \mathbb{S}_{\mu}\rightarrow \R$ by setting
\begin{align*}
\hat{\Psi}_{\mu}= \J_{\mu}(\hat{m}_{\mu}(u)) \quad \mbox{ and } \quad \Psi_{\mu}= \hat{\Psi}_{\mu}|_{\mathbb{S}_{\mu}}.
\end{align*}
Then we obtain the following result:
\begin{lemma}\label{SW3Ac}
Under the assumptions of Lemma \ref{SW1c}, we have that for $\mu>0$:
\begin{compactenum}[$(i)$]
\item $\Psi_{\mu}\in \C^{1}(\mathbb{S}_{\mu}, \R)$, and
\begin{equation*}
\Psi'_{\mu}(w)v= \|m_{\mu}(w)\|_{\mu} \J'_{\mu}(m_{\mu}(w))v \quad \mbox{ for } v\in T_{w}(\mathbb{S}_{\mu}).
\end{equation*}
\item $\{w_{n}\}$ is a Palais-Smale sequence for $\Psi_{\mu}$ if and only if $\{m_{\mu}(w_{n})\}$ is a Palais-Smale sequence for $\J_{\mu}$. If $\{u_{n}\}\subset \N_{\mu}$ is a bounded Palais-Smale sequence for $\J_{\mu}$, then $\{m_{\mu}^{-1}(u_{n})\}$ is a Palais-Smale sequence for $\Psi_{\mu}$.
\item $u\in \mathbb{S}_{\mu}$ is a critical point of $\Psi_{\mu}$ if and only if $m_{\mu}(u)$ is a critical point of $\J_{\mu}$. Moreover the corresponding critical values coincide and
\begin{equation*}
\inf_{\mathbb{S}_{\mu}} \Psi_{\mu}=\inf_{\N_{\mu}} \J_{\mu}=c_{\mu}.
\end{equation*}
\end{compactenum}
\end{lemma}
\begin{remark}
As in the previous section we have the following variational characterization of the infimum of $\J_{\mu}$ over $\N_{\mu}$:
\begin{align*}
c_{\mu}= \inf_{u\in \N_{\mu}} \J_{\mu}(u)= \inf_{u\in \X_{\mu}\setminus \{0\}} \max_{t\geq 0} \J_{\mu}(tu) = \inf_{u\in \mathbb{S}_{\mu}} \max_{t\geq 0} \J_{\mu}(tu).
\end{align*}
\end{remark}

\noindent
In order to obtain the existence of a nontrivial solution to the autonomous critical problem, we need to prove the following fundamental result.
\begin{lemma}\label{lemC}
For any $\mu>0$, there exists $v\in \X_{\mu}\setminus \{0\}$ such that
$$
\max_{t\geq 0} \J_{\mu}(tv)<\frac{s}{N} S_{*}^{\frac{N}{sp}}.
$$
In particular $c_{\mu}<\frac{s}{N} S_{*}^{\frac{N}{sp}}$.
\end{lemma}

\noindent
Before giving the proof of the above lemma, we recall some facts which will be crucial to estimate the mountain pass level $c_{\mu}$.

For any $\e>0$, let us define
\begin{equation*}
U_{\e}(x)= \frac{1}{\e^{\frac{N-sp}{p}}} U\left(\frac{|x|}{\e}\right),
\end{equation*}
where $U\in \mathcal{D}^{s, p}(\R^{N})$ is a solution to
$$
(-\Delta)^{s}_{p}U=S_{*} U^{\p-1} \quad \mbox{ in } \R^{N}.
$$
As showed in \cite{BMS}, we know that $U\in L^{\infty}(\R^{N})\cap \mathcal{C}^{0}(\R^{N})$ is a positive, radially symmetric and decreasing function with
$$
\lim_{|x|\rightarrow \infty} |x|^{\frac{N-sp}{p-1}} U(x)=U_{\infty}\in \R\setminus\{0\}.
$$
We also have the following interesting estimates:
\begin{lemma}\label{lemest1}\cite{BMS}
There exist constants $c_{1}, c_{2}>0$ and $\theta>1$ such that for all $r\geq 1$,
\begin{equation}\label{est1}
\frac{c_{1}}{r^{\frac{N-sp}{p-1}}}\leq U(r) \leq \frac{c_{2}}{r^{\frac{N-sp}{p-1}}}
\end{equation}
and
\begin{equation}\label{est2}
\frac{U(\theta r)}{U(r)}\leq \frac{1}{2}.
\end{equation}
\end{lemma}
Let $\theta$ be the universal constant in Lemma \ref{lemest1} that depends only on $N, p$ and $s$. For $\e, \delta>0$, set
\begin{equation*}
m_{\e, \delta} := \frac{U_{\e}(\delta)}{U_{\e}(\delta) - U_{\e}(\theta \delta)}.
\end{equation*}
Define
\begin{equation*}
g_{\e, \delta}(t):=
\left\{
\begin{array}{ll}
0 &\mbox{ if } 0\leq t\leq U_{\e}(\theta \delta) \\
m_{\e, \delta}^{p} (t- U_{\e}(\theta \delta)) &\mbox{ if } U_{\e}(\theta \delta)\leq t\leq U_{\e}(\delta) \\
t+ U_{\e}(\delta) (m_{\e, \delta}^{p-1}-1) &\mbox{ if } t\geq U_{\e}(\delta),
\end{array}
\right.
\end{equation*}
and
\begin{equation*}
G_{\e, \delta}(t):= \int_{0}^{t} (g'_{\e, \delta}(\tau))^{\frac{1}{p}} d\tau =
\left\{
\begin{array}{ll}
0 &\mbox{ if } 0\leq t\leq U_{\e}(\theta \delta) \\
m_{\e, \delta} (t- U_{\e}(\theta \delta)) &\mbox{ if } U_{\e}(\theta \delta)\leq t\leq U_{\e}(\delta) \\
t &\mbox{ if } t\geq U_{\e}(\delta).
\end{array}
\right.
\end{equation*}
Let us observe that $g_{\e, \delta}$ and $G_{\e, \delta}$ are nondecreasing and absolutely continuous functions. Now, we consider the radially symmetric nonincreasing function
\begin{equation*}
u_{\e, \delta}(r)= G_{\e, \delta}(U_{\e}(r)),
\end{equation*}
which, in view of the definition of $G_{\e, \delta}$,  satisfies
\begin{equation}\label{Tran}
u_{\e, \delta}(r)=
\left\{
\begin{array}{ll}
U_{\e}(r) &\mbox{ if } r\leq \delta \\
0 &\mbox{ if } r\geq \theta \delta.
\end{array}
\right.
\end{equation}
We recall the following useful estimates established in Lemma $2.7$ in \cite{MPSY}:
\begin{lemma}\label{MPSYlem}
There exists $C=C(N, p, s)>0$ such that for any $\e\leq \frac{\delta}{2}$ the following estimates hold
\begin{align*}
& [u_{\e, \delta}]_{W^{s, p}(\R^{N})}^{p}\leq S_{*}^{\frac{N}{sp}}+C\left(\left(\frac{\e}{\delta}\right)^{\frac{N-sp}{p-1}}\right), \\
&|u_{\e, \delta}|^{\p}_{L^{\p}(\R^{N})}\geq S_{*}^{\frac{N}{sp}}-C\left(\left(\frac{\e}{\delta}\right)^{\frac{N}{p-1}}\right),
\end{align*}
and
\begin{equation*}
|u_{\e, \delta}|^{p}_{L^{p}(\R^{N})}\geq
\left\{
\begin{array}{ll}
\frac{1}{C}\e^{sp} &\mbox{ if } N>sp^{2} \\
\frac{1}{C} \e^{sp} \log\left(\frac{\delta}{\e}\right) &\mbox{ if } N=sp^{2}.
\end{array}
\right.
\end{equation*}
\end{lemma}
In what follows, we prove an upper bound for the $L^{p}$-norm of $u_{\e, \delta}$:
\begin{lemma}\label{lemest2}
There exists a constant $C=C(N, p, s)>0$ such that for any $\e\leq \frac{\delta}{2}$
\begin{equation*}
|u_{\e, \delta}|_{L^{p}(\R^{N})}^{p} \leq
\left\{
\begin{array}{ll}
C\e^{sp} & \mbox{ if } N>sp^{2} \\
C\e^{sp} \log\left(\frac{\delta}{\e}\right) & \mbox{ if } N=sp^{2}.
\end{array}
\right.
\end{equation*}
\end{lemma}
\begin{proof}
Firstly, we consider the case $N>s p^{2}$. Let us observe that from the definition of $u_{\e, \delta}$ it follows that
\begin{align}\begin{split}\label{est3}
\int_{\R^{N}} u_{\e, \delta}^{p} dx &= \int_{\B_{\delta}(0)} u_{\e, \delta}^{p} dx+ \int_{\B_{\theta \delta}(0) \setminus \B_{\delta}(0)} u_{\e, \delta}^{p} dx + \int_{\B_{\theta \delta}^{c}(0)} u_{\e, \delta}^{p} dx\\
&= \int_{\B_{\delta}(0)} u_{\e, \delta}^{p} dx+ \int_{\B_{\theta \delta}(0) \setminus \B_{\delta}(0)} u_{\e, \delta}^{p} dx=:I+II.
\end{split}\end{align}
Now we estimate the two integrals on the right hand side of \eqref{est3}. By using a change of variable, Lemma \ref{lemest1} and the fact that $\e\leq \frac{\delta}{2}$, we can infer that
\begin{align}\label{est4}
I&= \int_{\B_{\delta}(0)} U_{\e}^{p}(x) \, dx = \e^{sp} \int_{\B_{\delta}(0)} U^{p}(x)\, dx \nonumber \\
&=\e^{sp} \omega_{N-1}\int_{1}^{\frac{\delta}{\e}} U^{p}(r) r^{N-1} \, dr \nonumber \\
&\leq c \e^{sp} \int_{1}^{\frac{\delta}{\e}} r^{-\frac{(N-sp)p}{p-1}+N-1} dr \nonumber \\
&=\frac{c \e^{sp}}{\frac{N-sp^{2}}{p-1}} \left[ 1- \left(\frac{\e}{\delta}\right)^{\frac{N-sp^{2}}{p-1}}\right] \leq C \e^{sp},
\end{align}
where $C$ is a positive constant.

Since $U_{\e}$ is radially nonincreasing, for any $\delta \leq r \leq \theta \delta$, we have
\begin{align*}
0\leq m_{\e, \delta} \left(U_{\e}(r)- U_{\e}(\theta \delta)\right) &= U_{\e}(\delta) \left[\frac{U_{\e}(r)- U_{\e}(\theta \delta)}{U_{\e}(\delta) - U_{\e}(\theta \delta)} \right]\leq U_{\e}(\delta).
\end{align*}
By using the definition of $U_{\e}$, $\frac{\delta}{\e}\geq 2$ and Lemma \ref{lemest1} we obtain
\begin{align}\label{est5}
II&=\int_{\B_{\theta \delta}(0) \setminus \B_{\delta}(0)} u_{\e, \delta}^{p} dx = \int_{\B_{\theta \delta}(0) \setminus \B_{\delta}(0)} \left[m_{\e, \delta} \left(U_{\e}(r)- U_{\e}(\theta \delta)\right) \right]^{p} dx \nonumber\\
&< \int_{\B_{\theta \delta}(0) \setminus \B_{\delta}(0)}  U_{\e}^{p}(\delta)\, dx \nonumber \\
&=|U_{\e}(\delta)|^{p} |\B_{\theta \delta}(0) \setminus \B_{\delta}(0)| \nonumber \\
&= C \frac{\delta^{N}}{\e^{N-sp}} \left|U\left(\frac{\delta}{\e}\right)\right|^{p} \nonumber \\
&\leq C \frac{\delta^{N}}{\e^{N-sp}} \left(\frac{\delta}{\e}\right)^{- \frac{(N-sp)p}{p-1}} c_{2}^{p} \nonumber \\
&= C \delta^{- \frac{N-sp^{2}}{p-1}} \e^{\frac{N-sp}{p-1}} \nonumber \\
&\leq C \e^{- \frac{N-sp^{2}}{p-1} + \frac{N-sp}{p-1}} = C\e^{sp}.
\end{align}
Putting together \eqref{est3}-\eqref{est5} we get the thesis. \\
Let us consider the case $N=s p^{2}$. Then, we can see that
\begin{align*}
I\leq c_{2}^{p}\e^{sp} \int_{1}^{\frac{\delta}{\e}} r^{-1} dr = c_{2}^{p} \e^{sp} \log\left(\frac{\delta}{\e}\right)
\end{align*}
and
\begin{align*}
II\leq C\e^{sp}.
\end{align*}
Therefore, being $\log(\frac{\delta}{\e})\geq \log(2)$ if $\e\leq \frac{\delta}{2}$, we can conclude that
\begin{align*}
|u_{\e, \delta}|_{L^{p}(\R^{N})}^{p} \leq \e^{sp} \log\left(\frac{\delta}{\e}\right) \left[ 1+ \frac{1}{\log\left(\frac{\delta}{\e}\right)}\right] \leq C \e^{sp} \log\left(\frac{\delta}{\e}\right).
\end{align*}
\end{proof}

\begin{lemma}\label{lemmaq}
If $q>\frac{N(p-1)}{N-sp}$, then for any $\e\leq \frac{\delta}{2}$ it holds
\begin{equation}
|u_{\e, \delta}|_{L^{q}(\R^{N})}^{q} \geq C\e^{N-\frac{(N-sp)q}{p}}.
\end{equation}
\end{lemma}
\begin{proof}
From the definitions of $u_{\e, \delta}$ and $U_{\e}$, and by using Lemma \ref{est1}, we have for any $\e\leq \frac{\delta}{2}$
\begin{align*}
\int_{\R^{N}} u_{\e, \delta}^{q}dx&\geq \int_{\B_{\delta}(0)} u_{\e, \delta}^{q}dx=\int_{\B_{\delta}(0)} U_{\e}^{q}dx=\e^{N-\frac{(N-sp)}{p}q} \int_{\B_{\frac{\delta}{\e}}(0)} U^{q} dx \\
&\geq \e^{N-\frac{(N-sp)}{p}q} \omega_{N-1} \int_{1}^{\frac{\delta}{\e}}  U(r)^{q} r^{N-1}dr \\
&\geq c_{1}^{q}\e^{N-\frac{(N-sp)}{p}q} \omega_{N-1}\int_{1}^{\frac{\delta}{\e}} r^{N-\frac{(N-sp)}{p}q-1} dr \\
&\geq C  \e^{N-\frac{(N-sp)}{p}q},
\end{align*}
because of $q>\frac{N(p-1)}{N-sp}$.
\end{proof}

Now, we are ready to give the proof of Lemma \ref{lemC}.
\begin{proof}[Proof of Lemma \ref{lemC}]
Take $\delta=1$ and let $u_{\e}:=u_{\e, 1}$ be the function defined in \eqref{Tran}.
Then, by using $(f'_6)$, we can see that
\begin{align*}
\J_{\mu}(tu_{\e})&=\frac{t^{p}}{p}\|u_{\e}\|^{p}_{\mu}-\int_{\R^{N}} F(t u_{\e})dx-\frac{t^{\p}}{\p} |u_{\e}|_{L^{\p}(\R^{N})}^{\p} \\
&\leq \frac{t^{p}}{p}\|u_{\e}\|^{p}_{\mu}-\lambda t^{q_{1}} |u_{\e}|_{L^{q_{1}}(\R^{N})}^{q_{1}} dx-\frac{t^{\p}}{\p} |u_{\e}|_{L^{\p}(\R^{N})}^{\p}\rightarrow -\infty \, \mbox{ as } t\rightarrow \infty,
\end{align*}
so there exists $t_{\e}>0$ such that
$$
\J_{\mu}(t_{\e}u_{\e})=\max_{t\geq 0} \J_{\mu}(t u_{\e}).
$$
Let us show that there exist $A, B>0$ such that
\begin{equation}\label{bounds}
A\leq t_{\e} \leq B \quad \mbox{ for } \e>0 \mbox{ sufficiently small. }
\end{equation}
Since $\frac{d}{dt}J_{\mu}(t_{\e}u_{\e})=0$, we deduce that
\begin{align}\label{pda1}
(t_{\e})^{p-1}\|u_{\e}\|^{p}_{\mu}=\int_{\R^{N}} f(t_{\e} u_{\e}) u_{\e}dx+(t_{\e})^{\p-1} |u_{\e}|_{L^{\p}(\R^{N})}^{\p}.
\end{align}
If $t_{\e_{n}}\rightarrow \infty$ as $\e_{n}\rightarrow 0$, from \eqref{pda1} it follows that
\begin{align*}
(t_{\e_{n}})^{p-1}\|u_{\e_{n}}\|^{p}_{\mu}\geq (t_{\e_{n}})^{\p-1} |u_{\e_{n}}|_{L^{\p}(\R^{N})}^{\p},
\end{align*}
which gives a contradiction in view of $\p>p$ and Lemma \ref{MPSYlem}.
Now, assume that there exists $t'_{\e_{n}}\rightarrow 0$ as $\e_{n}\rightarrow 0$. From $(f_{2})$ and $(f_{3})$ we can see that for any $\xi>0$ there exists $C_{\xi}>0$ such that
\begin{align}\label{pda2}
\int_{\R^{N}} f(t'_{\e_{n}} u_{\e_{n}}) u_{\e_{n}}\, dx&\leq \xi (t'_{\e_{n}})^{p-1} |u_{\e_{n}}|_{L^{p}(\R^{N})}^{p}+C_{\xi}  (t'_{\e_{n}})^{\p-1} |u_{\e_{n}}|_{L^{\p}(\R^{N})}^{\p} \nonumber \\
&\leq \frac{\xi}{\mu}  (t'_{\e_{n}})^{p-1} \|u_{\e_{n}}\|^{p}_{\mu}+C_{\xi}  (t'_{\e_{n}})^{\p-1} |u_{\e_{n}}|_{L^{\p}(\R^{N})}^{\p}.
\end{align}
Choosing $\xi=\frac{\mu}{2}$, and using \eqref{pda1} and \eqref{pda2}, we obtain
\begin{align*}
\frac{(t'_{\e_{n}})^{p-1}}{2}\|u_{\e_{n}}\|^{p}_{\mu}\leq C_{\xi}  (t'_{\e_{n}})^{\p-1}|u_{\e_{n}}|_{L^{\p}(\R^{N})}^{\p}+ (t'_{\e_{n}})^{\p-1}  |u_{\e_{n}}|_{L^{\p}(\R^{N})}^{\p}
\end{align*}
which is impossible because $\p>p$. Therefore, \eqref{bounds} holds true.

Thus, recalling that for $C, D>0$ it holds
$$
\frac{t^{p}}{p}C-\frac{t^{\p}}{\p}D\leq \frac{s}{N}\left(\frac{C}{D^{\frac{N-sp}{N}}}\right)^{\frac{N}{sp}} \quad \mbox{ for all } t\geq 0,
$$
and using \eqref{bounds}, we can see that
\begin{align*}
\J_{\mu}(t_{\e} u_{\e})&\leq \frac{t^{p}_{\e}}{p}\|u_{\e}\|^{p}_{\mu}-\lambda t_{\e}^{q_{1}}|u_{\e}|_{L^{q_{1}}(\R^{N})}^{q_{1}}-\frac{t_{\e}^{\p}}{\p} |u_{\e}|^{\p}_{L^{\p}(\R^{N})} \\
&\leq \frac{s}{N} \left(  \frac{[u_{\e}]^{p}_{W^{s, p}(\R^{N})}+\mu |u_{\e}|^{p}_{L^{p}(\R^{N})}}{|u_{\e}|^{p}_{L^{\p}(\R^{N})}}\right)^{\frac{N}{sp}}-\lambda A^{q_{1}} |u_{\e}|_{L^{q_{1}}(\R^{N})}^{q_{1}}.
\end{align*}
Now, in view of the following elementary inequality
$$
(a+b)^{r}\leq a^{r}+r(a+b)^{r-1}b \quad \mbox{ for all } a, b> 0, r\geq 1,
$$
and gathering the estimates in Lemma  \ref{MPSYlem} and Lemma \ref{lemest2}, we get
\begin{equation*}
\J_{\mu}(t_{\e} u_{\e})\leq
\left\{
\begin{array}{ll}
\frac{s}{N} S_{*}^{\frac{N}{sp}}+C_{1}\left(\e^{\frac{(N-sp)}{p-1}}\right)+C_{2} \e^{sp}-\lambda A^{q_{1}} |u_{\e}|_{L^{q_{1}}(\R^{N})}^{q_{1}} & \mbox{ if } N>sp^{2} \\
\frac{s}{N} S_{*}^{\frac{N}{sp}}+C_{3}\left(\e^{sp}\left(1+\log(1/\e) \right)\right)-\lambda A^{q_{1}} |u_{\e}|_{L^{q_{1}}(\R^{N})}^{q_{1}}  & \mbox{ if } N=sp^{2}.
\end{array}
\right.
\end{equation*}
Hence, if $N>s p^{2}$, we deduce that $q_{1}>p>\frac{N(p-1)}{N-sp}$ and by using Lemma \ref{lemmaq},
we have
\begin{align*}
\J_{\mu}(t_{\e} u_{\e})\leq S_{*}^{\frac{N}{sp}}+C_{1}\left(\e^{\frac{(N-sp)}{p-1}}\right)+C_{2} \e^{sp}-C_{4} \e^{N-\frac{(N-sp)}{p}q_{1}}.
\end{align*}
Since $N-\frac{(N-sp)}{p}q_{1}<sp<\frac{N-sp}{p-1}$ because of $q_{1}>p$ and $N>sp^{2}$, we can infer that
$$
\J_{\mu}(t_{\e} u_{\e})< \frac{s}{N} S_{*}^{\frac{N}{sp}},
$$
provided that $\e>0$ is sufficiently small.

When $N=sp^{2}$, we get $q_{1}>p=\frac{N(p-1)}{N-sp}$, and in view of Lemma \ref{lemmaq} we obtain
$$
\J_{\mu}(t_{\e} u_{\e})\leq \frac{s}{N} S_{*}^{\frac{N}{sp}}+C_{3}\left(\e^{sp}\left(1+\log(1/\e) \right)\right)-C_{5} \e^{sp^{2}-s(p-1)q_{1}}.
$$
Observing that $q_{1}>p$ yields
$$
\lim_{\e\rightarrow 0} \frac{\e^{sp^{2}-s(p-1)q_{1}}}{\e^{sp}\left(1+\log(1/\e) \right)}=\infty,
$$
we again get the conclusion for $\e$ small enough.
\end{proof}

Now, we prove the following lemma.
\begin{lemma}\label{lem2.2ac}
Let $\{u_{n}\}\subset \N_{\mu}$ be a minimizing sequence for $\J_{\mu}$. Then $\{u_{n}\}$ is bounded and there are a sequence $\{y_{n}\}\subset \R^{N}$ and constants $R, \beta>0$ such that
\begin{equation*}
\liminf_{n\rightarrow \infty} \int_{\B_{R}(y_{n})} |u_{n}|^{p} dx \geq \beta >0.
\end{equation*}
\end{lemma}
\begin{proof}
It is easy to check that $\{u_{n}\}$ is bounded in $\X_{\mu}$.
Now, we assume that for any $R>0$ it holds
\begin{equation*}
\lim_{n\rightarrow \infty} \sup_{y\in \R^{N}} \int_{\B_{R}(y)} |u_{n}|^{p} dx=0.
\end{equation*}
From the boundedness of $\{u_{n}\}$ and Lemma \ref{Lions} it follows that
\begin{equation}\label{tv4c}
u_{n}\rightarrow 0 \mbox{ in } L^{r}(\R^{N}) \quad \mbox{ for any } r\in (p, \p).
\end{equation}
By using \eqref{tv1}, \eqref{tv2}  and \eqref{tv4c} we deduce that
\begin{equation}\label{cc1}
0\leq \int_{\R^{N}} f(u_{n})u_{n} \, dx \leq \xi |u_{n}|_{L^{p}(\R^{N})}^{p} + o_{n}(1)
\end{equation}
and
\begin{equation}\label{cc2}
0\leq \int_{\R^{N}} F(u_{n}) \, dx \leq C \xi  |u_{n}|_{L^{p}(\R^{N})}^{p} + o_{n}(1).
\end{equation}
Since $V(x)\geq V_{0}$ and $\{u_{n}\}$ is bounded in $\X_{\mu}$, we can pass to the limit as $\xi\rightarrow 0$ in \eqref{cc1} and \eqref{cc2} to see that
\begin{equation}\label{cc3}
\int_{\R^{N}} f(u_{n})u_{n} \, dx =o_{n}(1) \quad \mbox{ and } \quad \int_{\R^{N}} F(u_{n}) \, dx=o_{n}(1).
\end{equation}
Taking into account that $\langle \J'_{\e}(u_{n}), u_{n}\rangle =0$ and by using \eqref{cc3}, we obtain
\begin{equation*}
\|u_{n}\|_{\mu}^{p}- |u_{n}|_{L^{\p}(\R^{N})}^{\p} =o_{n}(1).
\end{equation*}
Since $\{u_{n}\}$ is bounded in $\X_{\mu}$, we may assume that
\begin{equation}\label{cc4}
\|u_{n}\|_{\mu}^{p} \rightarrow \ell \geq 0 \quad \mbox{ and } \quad |u_{n}|_{L^{\p}(\R^{N})}^{\p} \rightarrow \ell \geq 0.
\end{equation}
Suppose by contradiction that $\ell>0$. By using \eqref{cc3} and \eqref{cc4} we get
\begin{align*}
c_{\mu}&= \J_{\mu}(u_{n}) + o_{n}(1) = \frac{1}{p} \|u_{n}\|_{\mu}^{p} - \int_{\R^{N}} F(u_{n}) dx - \frac{1}{\p} |u_{n}|_{L^{\p}(\R^{N})}^{\p} +o_{n}(1)\\
&= \frac{\ell}{p}- \frac{\ell}{\p} + o_{n}(1)= \frac{s}{N} \ell + o_{n}(1),
\end{align*}
that is $\ell = \frac{N}{s} c_{\mu}$.
On the other hand, by using Theorem \ref{Sembedding}, it follows that
\begin{equation*}
S_{*} |u_{n}|_{L^{\p}(\R^{N})}^{p} \leq [u_{n}]_{W^{s, p}(\R^{N})}^{p} + \mu |u_{n}|_{L^{p}(\R^{N})}^{p} = \|u_{n}\|_{\mu}^{p}
\end{equation*}
and taking the limit as $n\rightarrow \infty$ we find
\begin{equation*}
S_{*} \ell^{\frac{p}{\p}}\leq \ell.
\end{equation*}
Since $\ell = \frac{N}{s} c_{\mu}$, we get $c_{\mu}\geq \frac{s}{N} S_{*}^{\frac{N}{sp}}$ and this is impossible in view of Lemma \ref{lemC}.
\end{proof}

Therefore, we can prove the main result for the autonomous critical case.
\begin{lemma}\label{lem4.3c}
The problem $(P_{\mu}^{*})$ has at least one positive ground state solution.
\end{lemma}
\begin{proof}
The proof follows the arguments used in Lemma \ref{lem4.3}. We only need to replace \eqref{tervince1} by
\begin{align*}
c_{\mu}+o_{n}(1) &= \J_{\mu}(u_{n})-\frac{1}{\vartheta} \langle \J'_{\mu}(u_{n}), u_{n}\rangle \\
&=\Bigl(\frac{1}{p}-\frac{1}{\vartheta}\Bigr)\|u_{n}\|^{p}_{\mu}+ \int_{\R^{N}} \!\frac{1}{\vartheta} f(u_{n})u_{n}-F(u_{n}) \,dx + \Bigl(\frac{1}{\vartheta} - \frac{1}{\p}\Bigr) |u_{n}|_{L^{\p}(\R^{N})}^{\p},
\end{align*}
and recalling that
\begin{align*}
\limsup_{n\rightarrow \infty}\, (a_{n}+b_{n}+c_{n})&\geq \limsup_{n\rightarrow \infty} a_{n}+\liminf_{n\rightarrow \infty} (b_{n}+c_{n}) \\
&\geq \limsup_{n\rightarrow \infty} a_{n}+\liminf_{n\rightarrow \infty} b_{n}+\liminf_{n\rightarrow \infty} c_{n},
\end{align*}
we deduce that
\begin{align*}
c_{\mu}&\geq \left(\frac{1}{p}-\frac{1}{\vartheta}\right)\limsup_{n\rightarrow \infty} \|u_{n}\|^{p}_{\mu}+\liminf_{n\rightarrow \infty} \int_{\R^{N}} \left[\frac{1}{\vartheta} f(u_{n})u_{n}-F(u_{n})\right] dx \\
&+ \left(\frac{1}{\vartheta} - \frac{1}{\p}\right) \liminf_{n\rightarrow \infty} |u_{n}|_{L^{\p}(\R^{N})}^{\p}.
\end{align*}
Moreover, we use Lemma \ref{lem2.2ac} instead of Lemma \ref{lem2.2a}.
\end{proof}

\subsection{Existence result for the critical case}
Arguing as in Lemma \ref{lem2.2ac} we can prove the ``critical'' version of Lemma \ref{lem2.2}.
\begin{lemma}\label{lem2.2c}
Let $d<\frac{s}{N} S_{*}^{\frac{N}{sp}}$ and let $\{u_{n}\}\subset \N_{\e}$ be a sequence such that $\I_{\e}(u_{n})\rightarrow d$ and $u_{n}\rightharpoonup 0$ in $\h$. Then, one of the following alternatives occurs
\begin{compactenum}[$(a)$]
\item $u_{n}\rightarrow 0$ in $\h$;
\item there are a sequence $\{y_{n}\}\subset \R^{N}$ and constants $R, \beta>0$ such that
\begin{equation*}
\liminf_{n\rightarrow \infty} \int_{\B_{R}(y_{n})} |u_{n}|^{p} dx \geq \beta >0.
\end{equation*}
\end{compactenum}
\end{lemma}

The next result can be obtained following the lines of the proof of Lemma \ref{lem2.3}.
\begin{lemma}\label{lem2.3c}
Assume that $V_{\infty}<\infty$ and let $\{v_{n}\}\subset \N_{\e}$ be a sequence such that $\I_{\e}(v_{n})\rightarrow d$ with $d<\frac{s}{N} S_{*}^{\frac{N}{sp}}$ and $v_{n}\rightharpoonup 0$ in $\h$. If $v_{n}\not \rightarrow 0$ in $\h$, then $d\geq c_{V_{\infty}}$, where $c_{V_{\infty}}$ is the infimum of $\J_{V_{\infty}}$ over $\N_{V_{\infty}}$.
\end{lemma}

\noindent
Now, we establish the following compactness result in the critical case.
\begin{proposition}\label{prop2.1c}
Let $\{u_{n}\}\subset \N_{\e}$ be a sequence such that $\I_{\e}(u_{n})\rightarrow c$ where $c<c_{V_{\infty}}$ if $V_{\infty}<\infty$
and $c< \frac{s}{N}S_{*}^{\frac{N}{sp}}$ if $V_{\infty}=\infty$. Then $\{u_{n}\}$ admits a convergent subsequence.
\end{proposition}

\begin{proof}
Since $\I_{\e}(u_{n})\rightarrow c$ and $\I'_{\e}(u_{n})=0$, we can see that $\{u_{n}\}$ is bounded in $\h$ and, up to a subsequence, we may assume that $u_{n}\rightharpoonup u \mbox{ in } \h$.
Clearly $\I'_{\e}(u)=0$. \\
Now, let $v_{n}= u_{n}-u$. By using the Brezis-Lieb Lemma \cite{BL} and Lemma $3.3$ in \cite{MeW}, we know that
$$
|v_{n}|_{L^{\p}(\R^{N})}^{\p}= |u_{n}|_{L^{\p}(\R^{N})}^{\p}- |u|_{L^{\p}(\R^{N})}^{\p}+o_{n}(1)
$$
and
$$
\int_{\R^{N}} \left||v_{n}|^{\p-2}v_{n}-|u_{n}|^{\p-2}u_{n}+|u|^{\p-2}u\right|^{\frac{\p}{\p-1}} dx=o_{n}(1).
$$
Then, arguing as in Proposition \ref{prop2.1}, we can see that
\begin{align*}
\I_{\e}(v_{n})= \I_{\e}(u_{n})- \I_{\e}(u)+ o_{n}(1)= c-\I_{\e}(u)+ o_{n}(1):=d+o_{n}(1)
\end{align*}
and
\begin{align*}
\I'_{\e}(v_{n})=o_{n}(1).
\end{align*}
Now, by using $(f_{4})$ we get
\begin{equation}\label{tv199c}
\I_{\e}(u)=\I_{\e}(u)-\frac{1}{p} \langle\I'_{\e}(u),u\rangle=\int_{\R^{N}} \! \frac{1}{p} f(u)u-F(u) \,dx + \Bigl(\frac{1}{p} - \frac{1}{\p}\Bigr) |u|_{L^{\p}(\R^{N})}^{\p} \geq 0.
\end{equation}
If we suppose that $V_{\infty}<\infty$, from \eqref{tv199c} we deduce that $d\leq c<c_{V_{\infty}}$. In view of Lemma \ref{lem2.3c} we can see that $v_{n}\rightarrow 0$ in $\h$, that is $u_{n}\rightarrow u$ in $\h$.\\
Let us consider the case $V_{\infty}=\infty$. Then, we can use Lemma \ref{Cheng} to deduce that $v_{n}\rightarrow 0$ in $L^{r}(\R^{N})$ for all $r\in [p, \p)$. This combined with assumptions $(f_2)$ and $(f_3)$ implies that
\begin{equation}\label{cc10}
\int_{\R^{N}} f(v_{n})v_{n}dx=o_{n}(1) \quad \mbox{ and } \quad \int_{\R^{N}} F(v_{n})\, dx=o_{n}(1).
\end{equation}
Putting together \eqref{cc10} and $\langle\I'_{\e}(v_{n}), v_{n}\rangle=o_{n}(1)$, we deduce
$$
\|v_{n}\|^{p}_{\e}= |v_{n}|_{L^{\p}(\R^{N})}^{\p} +o_{n}(1).
$$
Since $\{v_{n}\}$ is bounded in $\h$, we may assume that $\|v_{n}\|^{p}_{\e}\rightarrow \ell$ and $|v_{n}|_{L^{\p}(\R^{N})}^{\p}\rightarrow \ell$, for some $\ell \geq 0$. Let us show that $\ell=0$. If by contradiction $\ell>0$, by using the fact that $\I_{\e}(v_{n})=d+ o_{n}(1)$, we get
\begin{equation*}
\frac{1}{p} \|v_{n}\|_{\e}^{p}- \frac{1}{\p} |v_{n}|_{L^{\p}(\R^{N})}^{\p}= d+o_{n}(1),
\end{equation*}
which implies
\begin{equation*}
\frac{s}{N} \|v_{n}\|_{\e}^{p}= d+o_{n}(1).
\end{equation*}
Taking the limit as $n\rightarrow \infty$ we have that $\frac{s}{N}\ell= d$, that is $\ell= d\frac{N}{s}$.
Therefore
\begin{equation*}
\|v_{n}\|_{\e}^{p} \geq S_{*}|v_{n}|_{L^{\p}(\R^{N})}^{p} = S_{*} \left(|v_{n}|_{L^{\p}(\R^{N})}^{\p} \right)^{\frac{p}{\p}},
\end{equation*}
and by passing to the limit as $n\rightarrow \infty$ we get $\ell \geq S_{*}^{\frac{N}{sp}}$. Thus we have $d\geq \frac{s}{N}S_{*}^{\frac{N}{sp}}$. Since $d\leq c < \frac{s}{N}S_{*}^{\frac{N}{sp}}$ we get a contradiction. Hence, $\ell=0$ and $u_{n}\rightarrow u$ in $\h$.
\end{proof}

\noindent
Finally we have the existence result for problem \eqref{Pc} for $\e>0$ small enough.
\begin{theorem}\label{thm3.1c}
Assume that $(V)$ and $(f_1)$-$(f_5)$ and $(f'_6)$ hold. Then there exists $\e_{0}>0$ such that problem ($P_{\e}^{*}$) admits a ground state solution for any $\e\in (0, \e_{0})$.
\end{theorem}
\begin{proof}
It is enough to proceed as in the proof of Theorem \ref{thm3.1} once replaced Lemma \ref{SW2}, Lemma \ref{SW3}, Proposition \ref{prop2.1}, Lemma \ref{lem4.3} by Lemma \ref{SW2c}, Lemma \ref{SW3c}, Proposition \ref{prop2.1c}, Lemma \ref{lem4.3c} respectively.
\end{proof}

\subsection{Multiple solutions for \eqref{Pc}}\label{Sect4.3}
In this section we study the multiplicity of solutions to \eqref{Pc}.
Arguing as in the proof of Proposition \ref{prop4.1} and using Lemma \ref{lem4.3c} and Lemma \ref{lem2.2c} instead of Lemma \ref{lem4.3} and Lemma \ref{lem2.2} respectively, we can deduce the following result.
\begin{proposition}\label{prop4.1c}
Let $\e_{n}\rightarrow 0^{+}$ and $\{u_{n}\}\subset \N_{\e_{n}}$ be such that $\I_{\e_{n}}(u_{n})\rightarrow c_{V_{0}}$. Then there exists $\{\tilde{y}_{n}\}\subset \R^{N}$ such that the translated sequence
\begin{equation*}
v_{n}(x):=u_{n}(x+ \tilde{y}_{n})
\end{equation*}
has a subsequence which converges in $W^{s, p}(\R^{N})$. Moreover, up to a subsequence, $\{y_{n}\}:=\{\e_{n}\tilde{y}_{n}\}$ is such that $y_{n}\rightarrow y\in M$.
\end{proposition}

\noindent
As in Section \ref{Sect3}, fix $\delta>0$ and let $\omega\in H^{s}(\R^{N})$ be a positive ground state solution to problem $(P_{V_{0}}^{*})$, which exists in view of Lemma \ref{lem4.3c}. Let $\psi\in C^{\infty}(\R^{+}, [0, 1])$ be a function such that $\psi=1$ in $[0, \frac{\delta}{2}]$ and $\psi=0$ in $[\delta, \infty)$. For any $y\in M$, let us define
$$
\Upsilon_{\e, y}(x)=\psi(|\e x-y|) \omega\left(\frac{\e x-y}{\e}\right),
$$
and $t_{\e}>0$ satisfying
$$
\I_{\e}(t_{\e}\Upsilon_{\e, y})=\max_{t\geq 0} \I_{\e}(t_{\e}\Upsilon_{\e, y}),
$$
and $\Phi_{\e}:M\rightarrow \N_{\e}$ by $\Phi_{\e}(y):=t_{\e} \Upsilon_{\e, y}$.
\begin{lemma}\label{lem4.1c}
The functional $\Phi_{\e}$ satisfies the following limit
\begin{equation}\label{3.2c}
\lim_{\e\rightarrow 0} \I_{\e}(\Phi_{\e}(y))=c_{V_{0}} \mbox{ uniformly in } y\in M.
\end{equation}
\end{lemma}
\begin{proof}
Assume by contradiction that there exist $\delta_{0}>0$, $\{y_{n}\}\subset M$ and $\e_{n}\rightarrow 0$ such that
\begin{equation}\label{4.41c}
|\I_{\e_{n}}(\Phi_{\e_{n}} (y_{n}))-c_{V_{0}}|\geq \delta_{0}.
\end{equation}
By using Lemma \ref{Psi} we know that
\begin{align}\label{4.42c}
\lim_{n\rightarrow \infty} \|\Upsilon_{\e_{n}, y_{n}}\|^{p}_{\e_{n}}= \|\omega\|_{V_{0}}^{p}.
\end{align}
On the other hand, from the definition of $t_{\e}$ we obtain $\langle \I'_{\e_{n}}(t_{\e_{n}}\Upsilon_{\e_{n}, y_{n}}), t_{\e_{n}} \Upsilon_{\e_{n}, y_{n}}\rangle=0$, which implies
\begin{align}\label{4.411c}
\|t_{\e_{n}}\Upsilon_{\e_{n}, y_{n}}\|^{p}_{\e_{n}}&=\int_{\R^{N}} f(t_{\e_{n}}\Upsilon_{\e_{n}}) t_{\e_{n}}\Upsilon_{\e_{n}} dx + \int_{\R^{N}} |t_{\e_{n}}\Upsilon_{\e_{n}}|^{\p} dx\nonumber\\
&=\int_{\R^{N}} f(t_{\e_{n}} \psi(|\e_{n}z|) \omega(z)) t_{\e_{n}} \psi(|\e_{n}z|) \omega(z) \, dz \nonumber \\
&\quad + t_{\e_{n}}^{\p} \int_{\R^{N}} |\psi(|\e_{n}z|) \omega(z)|^{\p} dz.
\end{align}
Let us prove that $t_{\e_{n}}\rightarrow 1$. Firstly we show that $t_{\e_{n}}\rightarrow t_{0}<\infty$. If by contradiction $|t_{\e_{n}}|\rightarrow \infty$, recalling that $\psi\equiv 1$ in $\B_{\frac{\delta}{2}}(0)$ and $\B_{\frac{\delta}{2}}(0)\subset \B_{\frac{\delta}{2\e_{n}}}(0)$ for $n$ sufficiently large, we can see that \eqref{4.411c} yields
\begin{align}\label{4.44c}
\|\Upsilon_{\e_{n}, y_{n}}\|^{p}_{\e_{n}}&\geq t_{\e_{n}}^{\p-p} \int_{\B_{\frac{\delta}{2}}(0)} |\omega(z)|^{\p} dz \nonumber \\
&\geq  t_{\e_{n}}^{\p-p} \left|\B_{\frac{\delta}{2}}(0)\right| \min_{|z|\leq \frac{\delta}{2}}|\omega(z)|^{\p} \rightarrow +\infty,
\end{align}
which is impossible in view of \eqref{4.42c}.
Hence, we can suppose that $t_{\e_{n}}\rightarrow t_{0}\geq 0$. By using the growth assumptions on $f$ and \eqref{4.42c}, we deduce that $t_{0}>0$. \\
Let us prove that $t_{0}=1$.
Taking the limit as $n\rightarrow \infty$ in \eqref{4.411c}, we can see that
$$
\|\omega\|_{V_{0}}^{p}=\int_{\R^{N}} \frac{f(t_{0}\omega)}{t_{0}^{p-1}} \omega \, dx + t_{0}^{\p-p} \int_{\R^{N}} |\omega|^{\p} \, dx,
$$
and by using the fact that $\omega\in \N_{V_{0}}$ and $(f_5)$, we can deduce that $t_{0}=1$.
This and the Dominated Convergence Theorem produce
\begin{equation*}
\lim_{n\rightarrow \infty}\int_{\R^{N}} F(t_{\e_{n}} \Upsilon_{\e_{n}, y_{n}})=\int_{\R^{N}} F(\omega) \, dx
\end{equation*}
and
\begin{equation*}
\lim_{n\rightarrow \infty}\int_{\R^{N}} |t_{\e_{n}} \Upsilon_{\e_{n}, y_{n}}|^{\p} dx =  \int_{\R^{N}} |\omega|^{\p} \, dx,
\end{equation*}
from which we obtain that
$$
\lim_{n\rightarrow \infty} \I_{\e_{n}}(\Phi_{\e_{n}}(y_{n}))=\J_{V_{0}}(\omega)=c_{V_{0}}
$$
which leads to a contradiction because of \eqref{4.41c}.
\end{proof}

\noindent
For any $\delta>0$, let $\rho>0$ be such that $M_{\delta}\subset \B_{\rho}(0)$. Let $\chi: \R^{N}\rightarrow \R^{N}$ be defined as
 \begin{equation*}
 \chi(x)=
 \left\{
 \begin{array}{ll}
 x &\mbox{ if } |x|<\rho \\
 \frac{\rho x}{|x|} &\mbox{ if } |x|\geq \rho.
  \end{array}
\right.
\end{equation*}
Let us consider the barycenter map $\beta_{\e}: \N_{\e}\rightarrow \R^{N}$ given by
\begin{align*}
\beta_{\e}(u)=\frac{\int_{\R^{N}} \chi(\e x) |u(x)|^{p} dx}{\int_{\R^{N}} |u(x)|^{p} dx}.
\end{align*}
Arguing as in the proof of Lemma \ref{lem4.2} we can prove the following result.
\begin{lemma}\label{lem4.2c}
The functional $\Phi_{\e}$ verifies the following limit
\begin{equation}\label{3.3c}
\lim_{\e \rightarrow 0} \beta_{\e}(\Phi_{\e}(y))=y \mbox{ uniformly in } y\in M.
\end{equation}
\end{lemma}

\noindent
Now, we take a function $h:\R_{+}\rightarrow \R_{+}$ such that $h(\e)\rightarrow 0$ as $\e \rightarrow 0$, and we set
$$
\widetilde{\N}_{\e}=\{u\in \N_{\e}: \I_{\e}(u)\leq c_{V_{0}}+h(\e)\}.
$$
Given $y\in M$, we can use Lemma \ref{lem4.1c} to infer that $h(\e)=|\I_{\e}(\Phi_{\e}(y))-c_{V_{0}}|\rightarrow 0$ as $\e \rightarrow 0$. Thus, $\Phi_{\e}(y)\in \widetilde{\N}_{\e}$, and we have that $\widetilde{\N}_{\e}\neq \emptyset$ for any $\e>0$. Moreover, proceeding as in Lemma \ref{lem4.4}, we get the following lemma.
\begin{lemma}\label{lem4.4c}
For any $\delta>0$, there holds that
$$
\lim_{\e \rightarrow 0} \sup_{u\in \widetilde{\mathcal{N}}_{\e}} dist(\beta_{\e}(u), M_{\delta})=0.
$$
\end{lemma}

\noindent
Finally, we give the proof of the main result related to \eqref{Pc}.
\begin{proof}[Proof of thm \ref{thmAI2}]
Given $\delta>0$ and taking into account Lemma \ref{SW2c}, Lemma \ref{SW3c}, Lemma \ref{lem4.1c}, Lemma \ref{lem4.2c}, Lemma \ref{lem4.4c}, Proposition \ref{prop2.1c}, Theorem \ref{thm3.1c} and recalling that $c_{V_{0}}<\frac{s}{N} S_{*}^{\frac{N}{sp}}$ (see Lemma \ref{lemC}), we can argue as in the proof of Theorem \ref{teorema} to deduce the existence of at least $cat_{M_{\delta}}(M)$ positive solutions for \eqref{Pc} for all $\e>0$ sufficiently small. The concentration of solutions is obtained following the arguments used in subsection $3.5$. Indeed, the proof of Lemma \ref{lemMoser} works also in the critical case and the proof of Lemma \ref{UBlemAF} can be easily adapted in this case.
\end{proof}

\section*{Acknowledgments.}
The authors warmly thank the anonymous referee for her/his useful and nice comments on the paper.

\renewcommand{\refname}{REFERENCES}


\begin{thebibliography}{99}

\bibitem{Alv}
\newblock C.O. Alves,
\newblock Existence of positive solutions for a problem with lack of compactness involving the $p$-Laplacian,
\newblock \emph{Nonlinear Anal.}, \textbf{51} (2002), no. 7, 1187--1206.

\bibitem{AA}
\newblock C.O. Alves and V. Ambrosio,
\newblock A multiplicity result for a nonlinear fractional Schr\"odinger equation in $\R^{N}$ without the Ambrosetti-Rabinowitz condition,
\newblock \emph{J. Math. Anal. Appl.}, \textbf{ 466} (2018), no. 1, 498--522.

\bibitem{AF}
\newblock C.O. Alves and G.M. Figueiredo,
\newblock Existence and multiplicity of positive solutions to a $p$-Laplacian equation in $\R^{N}$,
\newblock \emph{Differential Integral Equations}, \textbf{19} (2006), no. 2, 143--162.

\bibitem{AM}
\newblock C. O. Alves and O. H. Miyagaki,
\newblock Existence and concentration of solution for a class of fractional elliptic equation in $\R^{N}$ via penalization method,
\newblock \emph{Calc. Var. Partial Differential Equations}, \textbf{55} (2016), no. 3, Paper No. 47, 19 pp.

\bibitem{A1}
\newblock V. Ambrosio,
\newblock Multiple solutions for a fractional $p$-Laplacian equation with sign-changing potential,
\newblock \emph{Electron. J. Diff. Equ.}, vol. 2016 (2016), no. 151, pp. 1--12.

\bibitem{A2}
\newblock V. Ambrosio,
\newblock Multiplicity of positive solutions for a class of fractional Schr\"odinger equations via penalization method,
\newblock \emph{Ann. Mat. Pura Appl.}, (4) \textbf{196} (2017), no. 6, 2043--2062.

\bibitem{A4}
\newblock V. Ambrosio,
\newblock Concentration phenomena for critical fractional Schr\"odinger systems,
\newblock \emph{Commun. Pure Appl. Anal.}, \textbf{17} (2018), no. 5, 2085--2123.

\bibitem{A3}
\newblock V. Ambrosio,
\newblock Concentrating solutions for a class of nonlinear fractional Schr\"odinger equations in $\R^{N}$,
\newblock \emph{Rev. Mat. Iberoam.} (in press), Preprint. arXiv{1612.02388}.

\bibitem{AMF}
\newblock V. Ambrosio and G. M. Figueiredo,
\newblock Ground state solutions for a fractional Schr\"odinger equation with critical growth,
\newblock \emph{Asymptot. Anal.}, \textbf{105} (2017), no. 3--4, 159--191.

\bibitem{AI}
\newblock V. Ambrosio and T. Isernia,
\newblock Concentration phenomena for a fractional Schr\"odinger-Kirchhoff  type problem,
\newblock \emph{Math. Methods Appl. Sci.}, \textbf{41} (2018), no.2, 615--645.

\bibitem{BBMP}
\newblock P. Belchior, H. Bueno, O.H. Miyagaki and G.A. Pereira,
\newblock Remarks about a fractional Choquard equation: ground state, regularity and polynomial decay,
\newblock \emph{Nonlinear Analysis}, \textbf{164} (2017), 38--53.

\bibitem{BC}
\newblock V. Benci and G. Cerami,
\newblock Multiple positive solutions of some elliptic problems via the Morse theory and the domain topology,
\newblock \emph{Calc. Var. Partial Differential Equations}, \textbf{2} (1994), no. 1, 29--48.

\bibitem{BMS}
\newblock L. Brasco, S. Mosconi and M. Squassina,
\newblock Optimal decay of extremals for the fractional Sobolev inequality,
\newblock \emph{Calc. Var. Partial Differential Equations}, \textbf{55} (2016), no. 2, Paper No. 23, 32 pp.

\bibitem{BL}
\newblock H. Br\'ezis and E. Lieb,
\newblock A relation between pointwise convergence of functions and convergence of functionals,
\newblock \emph{Proc. Amer. Math. Soc.}, \textbf{88} (1983), no. 3, 486--490.

\bibitem{CS}
\newblock L. Caffarelli and L. Silvestre,
\newblock An extension problem related to the fractional Laplacian,
\newblock \emph{Comm. Partial Differential Equations}, \textbf{32} (2007), no. 7-9, 1245--1260.

\bibitem{DPQ}
\newblock L.M. Del Pezzo and A. Quaas,
\newblock A Hopf's lemma and a strong minimum principle for the fractional $p$-Laplacian,
\newblock \emph{J. Differential Equations}, \textbf{263} (2017), no. 1, 765--778.

\bibitem{DCKP}
\newblock A. Di Castro, T. Kuusi and G. Palatucci,
\newblock Local behavior of fractional $p$-minimizers,
\newblock \emph{Ann. Inst. H. Poincar\'e Anal. Non Lin\'eaire}, \textbf{33} (2016), no. 5, 1279--1299.

\bibitem{DPV}
\newblock E. Di Nezza, G. Palatucci and E. Valdinoci,
\newblock Hitchhiker's guide to the fractional Sobolev spaces,
\newblock \emph{Bull. Sci. math.}, \textbf{136} (2012), 521--573.

\bibitem{Ding}
\newblock Y. Ding,
\newblock \emph{Variational methods for strongly indefinite problems},
\newblock Interdisciplinary Mathematical Sciences, 7. World Scientific Publishing Co. Pte. Ltd., Hackensack, NJ, 2007. viii+168 pp.

\bibitem{DMV}
\newblock S. Dipierro, M. Medina and E. Valdinoci,
\newblock \emph{Fractional elliptic problems with critical growth in the whole of $\R^{n}$}.
\newblock Appunti. Scuola Normale Superiore di Pisa (Nuova Serie) [Lecture Notes. Scuola Normale Superiore di Pisa (New Series)], 15. Edizioni della Normale, Pisa, 2017. viii+152 pp.

\bibitem{Ek}
\newblock I. Ekeland,
\newblock On the variational principle,
\newblock \emph{J. Math. Anal. Appl.}, \textbf{47} (1974), 324--353.

\bibitem{FQT}
\newblock P. Felmer, A. Quaas and J. Tan,
\newblock Positive solutions of the nonlinear {S}chr{\"o}dinger equation with the fractional Laplacian,
\newblock \emph{Proc. Roy. Soc. Edinburgh Sect. A}, \textbf{142} (2012), 1237--1262.

\bibitem{F}
\newblock G.M. Figueiredo,
\newblock Existence, multiplicity and concentration of positive solutions for a class of quasilinear problems with critical growth,
\newblock \emph{Comm. Appl. Nonlinear Anal.}, \textbf{13} (2006), no. 4, 79--99.

\bibitem{FS}
\newblock G.M. Figueiredo and G. Siciliano,
\newblock A multiplicity result via Ljusternick-Schnirelmann category and Morse theory for a fractional Schr\"odinger equation in $\R^{N}$,
\newblock \emph{NoDEA Nonlinear Differential Equations Appl.}, \textbf{23} (2016), no. 2, Art. 12, 22 pp.

\bibitem{FiP}
\newblock A. Fiscella and P. Pucci,
\newblock Kirchhoff-Hardy fractional problems with lack of compactness,
\newblock \emph{Adv. Nonlinear Stud.}, \textbf{17} (2017), no. 3, 429--456.

\bibitem{FrP}
\newblock G. Franzina and G. Palatucci,
\newblock Fractional $p$-eigenvalues,
\newblock \emph{Riv. Math. Univ. Parma (N.S.)}, \textbf{5} (2014), no. 2, 373--386.

\bibitem{IMS}
\newblock A. Iannizzotto, S. Mosconi and M. Squassina,
\newblock Global H\"older regularity for the fractional $p$-Laplacian,
\newblock \emph{Rev. Mat. Iberoam.}, \textbf{32} (2016), no. 4, 1353--1392.

\bibitem{IT}
\newblock T. Isernia,
\newblock Positive solution for nonhomogeneous sublinear fractional equations in $\mathbb{R}^{N}$,
\newblock \emph{Complex Var. Elliptic Equ.}, \textbf{63} (2018), no. 5, 689--714.

\bibitem{Laskin1}
\newblock N. Laskin,
\newblock Fractional quantum mechanics and L\`evy path integrals,
\newblock \emph{Phys. Lett. A}, \textbf{268} (2000), 298--305.

\bibitem{Laskin2}
\newblock N. Laskin,
\newblock Fractional Schr\"odinger equation,
\newblock \emph{Phys. Rev. E}, \textbf{66} (2002), 056108.

\bibitem{LL}
\newblock E. Lindgren and P. Lindqvist,
\newblock Fractional eigenvalues,
\newblock \emph{Calc. Var. Partial Differential Equations}, \textbf{49} (2014), 795--826.

\bibitem{MW}
\newblock J. Mawhin and M. Willem,
\newblock \emph{Critical Point Theory and Hamiltonian Systems},
\newblock Springer-Verlag, 1989.

\bibitem{MeW}
\newblock C. Mercuri and M. Willem,
\newblock A global compactness result for the $p$-Laplacian involving critical nonlinearities,
\newblock \emph{Discrete Contin. Dyn. Syst.}, \textbf{28} (2010), no. 2, 469--493.

\bibitem{MRS}
\newblock G. Molica Bisci, V. R\u{a}dulescu and R. Servadei,
\newblock \emph{Variational Methods for Nonlocal Fractional Problems},
\newblock Cambridge University Press, \textbf{162} Cambridge, 2016.

\bibitem{MPSY}
\newblock S. Mosconi, K. Perera, M. Squassina and Y. Yang,
\newblock The Brezis-Nirenberg problem for the fractional $p$-Laplacian,
\newblock \emph{Calc. Var. Partial Differential Equations}, \textbf{55} (2016), no. 4, Paper No. 105, 25 pp.


\bibitem{Moser}
\newblock J. Moser,
\newblock A new proof of De Giorgi's theorem concerning the regularity problem for elliptic differential equations,
\newblock \emph{Comm. Pure Appl. Math.}, \textbf{13} (1960), 457--468.

\bibitem{PP}
\newblock G. Palatucci and A. Pisante,
\newblock Improved Sobolev embeddings, profile decomposition, and concentration-compactness for fractional Sobolev spaces,
\newblock \emph{Calc. Var. Partial Differential Equations}, \textbf{50} (2014), no. 3-4, 799--829.

\bibitem{Rab}
\newblock P. Rabinowitz,
\newblock On a class of nonlinear Schr\"odinger equations,
\newblock \emph{Z. Angew. Math. Phys.}, \textbf{43} (1992), no. 2, 270--291.

\bibitem{Secchi}
\newblock S. Secchi,
\newblock Ground state solutions for nonlinear fractional Schr\"odinger equations in $\R^{N}$,
\newblock \emph{J. Math. Phys.}, \textbf{54} (2013), 031501.

\bibitem{SZY}
\newblock X. Shang, J. Zhang and Y. Yang,
\newblock On fractional Schr\"odinger equations with critical growth,
\newblock \emph{J. Math. Phys.}, \textbf{54} (2013), no. 12, 121502, 20 pp.

\bibitem{SW}
\newblock A. Szulkin and T. Weth,
\newblock The method of Nehari manifold,
\newblock in \emph{Handbook of Nonconvex Analysis and Applications}, edited by D. Y. Gao and D. Montreanu (International Press, Boston, 2010), pp. 597--632.

\bibitem{Zh}
\newblock J. Zhang,
\newblock Stability of standing waves for nonlinear Schr\"odinger equations with unbounded potentials,
\newblock \emph{Z. Angew. Math. Phys.}, \textbf{51} (2000), no. 3, 498--503.

\end{thebibliography}
\end{document}